\documentclass[10pt]{puthesis}

\usepackage{amssymb}
\usepackage{amsmath}
\usepackage{amsthm}
\usepackage{graphicx}
\usepackage{algorithm}
\usepackage{algorithmic}
\usepackage[sanserif,full]{complexity}
\usepackage{pict2e}
\usepackage{url}

\newcommand{\bbC}{\mathbb{C}}

\newcommand{\bbR}{\mathbb{R}}
\newcommand{\bbZ}{\mathbb{Z}}
\newcommand{\calB}{\mathcal{B}}

\newcommand{\rme}{\mathrm{e}}
\newcommand{\rmi}{\mathrm{i}}

\newcommand{\rmO}{\mathrm{O}}

\newcommand{\rmT}{\mathrm{T}}

\newcommand{\abs}[1]{|{#1}|}

\newcommand{\set}[1]{\{{#1}\}}

\newcommand{\ip}[2]{\langle{#1},{#2}\rangle}

\newtheorem{thm}{Theorem}
\newtheorem{lem}[thm]{Lemma}
\newtheorem{cor}[thm]{Corollary}

\newtheorem{conj}[thm]{Conjecture}
\newtheorem{prob}[thm]{Problem}

\theoremstyle{definition}
\newtheorem{exmp}[thm]{Example}
\newtheorem{defn}[thm]{Definition}

\author{Dustin G.~Mixon}
\adviser{Robert Calderbank}
\title{Sparse Signal Processing with Frame Theory}
\submitted{June 2012} %graduation date
\abstract{Many emerging applications involve sparse signals, and their processing is a subject of active research.
We desire a large class of sensing matrices which allow the user to discern important properties of the measured sparse signal.
Of particular interest are matrices with the restricted isometry property (RIP).
RIP matrices are known to enable efficient and stable reconstruction of sufficiently sparse signals, but the deterministic construction of such matrices has proven very difficult.
In this thesis, we discuss this matrix design problem in the context of a growing field of study known as frame theory.
In the first two chapters, we build large families of equiangular tight frames and full spark frames, and we discuss their relationship to RIP matrices as well as their utility in other aspects of sparse signal processing.
In Chapter~3, we pave the road to deterministic RIP matrices, evaluating various techniques to demonstrate RIP, and making interesting connections with graph theory and number theory.
We conclude in Chapter~4 with a coherence-based alternative to RIP, which provides near-optimal probabilistic guarantees for various aspects of sparse signal processing while at the same time admitting a whole host of deterministic constructions.
}
\acknowledgements{This thesis is based on a series of papers I coauthored with a long list of friends, colleagues and mentors:
Boris Alexeev,
Waheed U. Bajwa,
Afonso S. Bandeira,
Jameson Cahill,
Robert Calderbank,
Matthew Fickus,
Negar Kiyavash,
Christopher J. Quinn,
Janet Tremain, and
Percy Wong.
Each member of this list taught me a thing or two throughout the course of my thesis research, and I very much appreciate it!

My time at Princeton has been a lot of fun, thanks in large part to the good friends I've made here.
From eating sushi, to playing board games, to solving fun math riddles, the experience has been a blast, and I'll always remember it.
My wife has a gift for filling my life with beauty and love, and last year, she gave me a beautiful new life to love.
Thank you, Tessia and Charlotte, for making my life wonderful.
Finally, I thank my parents for their unfailing love and support, and I thank God for His role in all of these things.

This research was supported in part by the A.B. Krongard Fellowship. 
The views expressed in this thesis are those of the author and do not reflect the official policy or position of the United States Air Force, Department of Defense, or the U.S. Government.
}
\dedication{To all those who never dedicated a dissertation to themselves.\\And to my daughter, Charlotte.}

\begin{document}

\section{Overview}

In several applications, data is traditionally collected in massive quantities before employing a reasonable compression strategy. 
The result is a storage bottleneck that can be prevented with a data collection alternative known as \emph{compressed sensing}. 
The philosophy behind compressed sensing is that we might as well target the meaningful data features up front instead of spending our storage budget on less-telling measurements. 
As an example, natural images tend to have a highly compressible wavelet decomposition because many of the wavelet cofficients are typically quite small. 
In this case, one might consider targeting large wavelet coefficients as desired image features; in fact, removing the contribution of the smallest wavelet coefficients will have little qualitative effect on the image~\cite{DavenportDEK:11}, and so using sparsity in this way is intuitively reasonable.

Let $x$ be an unknown $N$-dimensional vector with the property that at most $K$ of its entries are nonzero, that is, $x$ is $K$\emph{-sparse}.
The goal of compressed sensing is to construct relatively few non-adaptive linear measurements along with a stable and efficient reconstruction algorithm that exploits this sparsity structure.
Expressing each measurement as a row of an $M\times N$ matrix $\Phi$, we have the following noisy system:
\begin{equation}
\label{eq.cs eqn}
y=\Phi x+z.
\end{equation}
In the spirit of \emph{compressed} sensing, we only want a few measurements: $M\ll N$.
Also, in order for there to exist an inversion process for \eqref{eq.cs eqn}, $\Phi$ must map $K$-sparse vectors injectively, or equivalently, every subcollection of $2K$ columns of $\Phi$ must be linearly independent.
Unfortunately, the natural reconstruction method in this general case, i.e., finding the sparsest approximation of $y$ from the dictionary of columns of $\Phi$, is known to be $\NP$-hard~\cite{Natarajan:95}.
Moreover, the independence requirement does not impose any sort of dissimilarity between the columns of $\Phi$, meaning distinct identity basis elements could lead to similar measurements, thereby bringing instability in reconstruction.

To get around the $\NP$-hardness of sparse approximation, we need more structure in the matrix~$\Phi$.
Indeed, several efficient reconstruction algorithms have been considered (e.g., Basis Pursuit~\cite{DonohoE:03,DonohoET:06,GribonvalN:03}, Orthogonal Matching Pursuit~\cite{DonohoET:06,Tropp:04}, and the Least Absolute Shrinkage and Selection Operator~\cite{BenHaimEE:10}), and their original performance guarantees depend on the additional structure that the columns of $\Phi$ are nearly orthogonal to each other.
Depending on the algorithm, this structure in the sensing matrix enables successful reconstruction when noise term $z$ in \eqref{eq.cs eqn} is zero, adversarial, or stochastic, but for any of the original guarantees to apply, the sparsity level must be $K=\mathrm{O}(\sqrt{M})$.
To reconstruct signals with larger sparsity levels, Cand\`{e}s and Tao~\cite{CandesT:06} impose a much stronger requirement on the sensing matrix: that every submatrix of $2K$ columns of $\Phi$ be well-conditioned.
To be explicit, we have the following definition:

\begin{defn}
\label{defn.rip}
The matrix $\Phi$ has the \emph{$(K,\delta)$-restricted isometry property (RIP)} if 
\begin{equation*}
(1-\delta)\|x\|^2\leq\|\Phi x\|^2\leq(1+\delta)\|x\|^2
\end{equation*}
for every $K$-sparse vector $x$.
The smallest $\delta$ for which $\Phi$ is $(K,\delta)$-RIP is the \emph{restricted isometry constant (RIC)} $\delta_K$.
\end{defn}

In words, matrices which satisfy RIP act as a near-isometry on sufficiently sparse vectors.
Among other things, this structure imposes near-orthogonality between the columns of $\Phi$, and so in light of the previous results, it is not surprising that RIP sensing matrices enable efficient reconstruction:

\begin{thm}[Theorem 1.3 in \cite{Candes:08}]
\label{thm.rip use}
Suppose an $M\times N$ matrix $\Phi$ has the $(2K,\delta)$-restricted isometry property for some $\delta<\sqrt{2}-1$.
Assuming $\|z\|\leq\varepsilon$, then for every $K$-sparse vector $x\in\mathbb{R}^N$, the following reconstruction from \eqref{eq.cs eqn}:
\begin{equation*}
\tilde{x}=\arg\min\|\hat{x}\|_1\qquad\mbox{s.t. }\|y-\Phi\hat{x}\|\leq\varepsilon
\end{equation*}
satisfies $\|\tilde{x}-x\|\leq C\varepsilon$, where $C$ only depends on $\delta$.
\end{thm}

The exciting part about this guarantee is how the sparsity level $K$ of recoverable signals scales with the number of measurements $M$.
Certainly, we expect at least $K\sim \sqrt{M}$ since RIP is a stronger matrix requirement than near-orthogonality between columns.
In analyzing the sparsity level, random matrices have found the most success, specifically matrices with independent Gaussian or Bernoulli entries~\cite{BaraniukDDW:08}, or matrices whose rows were randomly selected from the discrete Fourier transform matrix~\cite{RudelsonV:08}.
With high probability, these random constructions support sparsity levels $K$ on the order of $\smash{\frac{M}{\log^\alpha N}}$ for some $\alpha\geq1$.
Intuitively, this level of sparsity is near-optimal because $K$ cannot exceed $\smash{\frac{M}{2}}$ by the linear independence condition.
Thus, Theorem~\ref{thm.rip use} is a substantial improvement over the previous guarantees, and this has prompted further investigation of RIP matrices.
Unfortunately, it is difficult to check whether a particular instance of a random matrix is $(K,\delta)$-RIP, as this involves the calculation of singular values for all $\smash{\binom{N}{K}}$ submatrices of $K$ columns of the matrix.
For this reason, and for the sake of reliable sensing standards, many have pursued deterministic RIP matrix constructions; Tao discusses the significance of this open problem in~\cite{Tao:07}.

Throughout this thesis, we consider the problem from a variety of directions.
In Chapter~1, we observe a technique which is commonly used to analyze the restricted isometry of deterministic constructions: the Gershgorin circle theorem.
This technique fails to demonstrate RIP for large sparsity levels; it is only capable of showing RIP for sparity levels on the order of $\sqrt{M}$, as opposed to $M$.
This limitation has become known as the ``square-root bottleneck.''
To illustrate that this bottleneck is not merely an artifact of the Gershgorin analysis, we consider a construction which is optimal in the Gershgorin sense, and we establish that this construction is $(K,\delta)$-RIP for every $K\leq\delta\sqrt{M}$ but is not $(K,1-\varepsilon)$-RIP for any $K>\sqrt{2M}$.
The first inequality is proved by the Gershgorin circle theorem, while the second uses the \emph{spark} of the matrix, that is, the number of nonzero entries in the sparsest vector in its nullspace.
While this disparity between $\sqrt{M}$ and $M$ is significant in many applications, such constructions are particularly well-suited for the sparse signal processing application of digital fingerprinting, and so we briefly investigate this application.

For the applications with larger sparsity levels, we note that spark deficiency is incompatible with restricted isometry; indeed, any matrix which is $(K,1-\varepsilon)$-RIP necessarily has spark strictly greater than $K$.  
As such, in Chapter~2, we consider $M\times N$ \emph{full spark} matrices, that is, matrices whose spark is as large as possible: $M+1$.
We start by finding various full spark constructions using Vandermonde matrices and discrete Fourier transforms.
These deterministic constructions are particularly attractive as RIP candidates because they satisfy the necessary condition of large spark, a property which is difficult to verify in general.
To solidify this notion of difficulty, we also show that the problem of testing whether a matrix is full spark is hard for $\NP$ under randomized polynomial-time reductions; this contrasts with the similar problem of testing for RIP, which currently has unknown computational complexity~\cite{KoiranZ:11}.
To demonstrate that full spark matrices are useful in their own right, we use them to solve another important problem in sparse signal processing: signal recovery without phase.

To date, the only deterministic RIP construction that manages to go beyond the square-root bottleneck is given by Bourgain et al.~\cite{BourgainDFKK:11}.
In Chapter~3, we discuss the technique they use to demonstrate RIP.
It is important to stress the significance of their contribution: 
Before~\cite{BourgainDFKK:11}, it was unclear how deterministic analysis might break the bottleneck, and as such, their result is a major theoretical achievement. 
On the other hand, their improvement over the square-root bottleneck is notably slight compared to what random matrices provide. 
However, we show that their technique can actually be used to demonstrate RIP for sparsity levels much larger than $\sqrt{M}$, meaning one could very well demonstrate random-like performance given the proper construction. 
Our result applies their technique to random matrices, and it inadvertently serves as a simple alternative proof that certain random matrices are RIP. 
We also introduce another technique, and we show that it can demonstrate RIP for similarly large sparsity levels.
Later, we propose a specific class of full spark matrices as candidates for being RIP. 
Using a correspondence between these matrices and the Paley graphs, we observe certain combinatorial and number-theoretic implications; this lends some probabilistic intuition for a new bound on the clique number of Paley graphs of prime order.

After investigating deterministic RIP matrices in Chapters~1--3, we have yet to find deterministic $M\times N$ sensing matrices which provably allow for the efficient reconstruction of signals with sparsity level $K\sim\smash{\frac{M}{\log^\alpha N}}$ for some $\alpha\geq1$.
To fill this gap, in Chapter~4, we consider an alternative model for the sparsity in our signal, namely, that the locations of the nonzero entries are drawn uniformly at random.
With this model, we show that a particularly simple algorithm called \emph{one-step thresholding} can reconstruct the signal with high probability provided $K=\mathrm{O}(\frac{M}{\log N})$.
In fact, this performance guarantee requires relatively modest structure in the sensing matrix: that the columns are nearly orthogonal to each other and well-distributed over the unit sphere.
Indeed, this structural requirement is much less stringent than RIP, and we provide a catalog of random and \emph{deterministic} sensing matrices which satisfy these conditions.
Later, we further analyze the two conditions separately, finding new fundamental limits on near-orthogonality and illustrating how to manipulate a given sensing matrix to achieve good distribution over the sphere.

Throughout this thesis, we use ideas from \emph{frame theory}, and so it is fitting to take some time to review the basics:

\section{A brief introduction to frame theory}
\label{section.intro to frame theory}

A \emph{frame} is a sequence $\{\varphi_i\}_{i\in\mathcal{I}}$ in a Hilbert space $\mathcal{H}$ with \emph{frame bounds} $0<A\leq B<\infty$ that satisfy
\begin{equation*}
A\|x\|^2\leq\sum_{i\in\mathcal{I}}|\langle x,\varphi_i\rangle|^2\leq B\|x\|^2
\qquad\forall x\in\mathcal{H}.
\end{equation*}
Frames were introduced by Duffin and Schaeffer~\cite{DuffinS:52} in the context of nonharmonic Fourier analysis, where $\mathcal{H}=L^2(-\pi,\pi)$ and the frame elements $\varphi_i$ are sinusoids of irregularly spaced frequencies.
However, the modern application of frame theory to signal processing came decades later after the landmark paper of Daubechies et al.~\cite{DaubechiesGM:86}. 
This paper gave the first nontrivial examples of \emph{tight frames}, that is, frames with equal frame bounds $A=B$.
The utility of tight frames lies partially in their painless reconstruction formula:
\begin{equation*}
x=\frac{1}{A}\sum_{i\in\mathcal{I}}\langle x,\varphi_i\rangle\varphi_i.
\end{equation*}
Note that orthonormal bases are tight frames with $A=B=1$; in this way, frames form a natural and useful generalization.
While this founding research in frame theory concerned frames over infinite-dimensional Hilbert spaces, many of today's applications of frames require a finite-dimensional treatment.
In fact, finite frame theory has found some important progress in the past decade~\cite{BenedettoF:03,CahillFMPS:11,CasazzaFM:12,CasazzaFMWZ:11,CasazzaT:06,StrohmerH:03}, and the remainder of this section will discuss the basics of this field.

In finite dimensions, say, $\mathcal{H}=\mathbb{C}^M$, a frame is given by the columns of a full-rank $M\times N$ matrix $\Phi=[\varphi_1\cdots\varphi_N]$ with $N\geq M$.
Here, the extreme eigenvalues of $\Phi\Phi^*$ are the frame bounds, and a tight frame has equal frame bounds; equivalently, a frame $\Phi$ is tight if
\begin{itemize}
 \item[(i)]  the rows are equal-norm and orthogonal.
\end{itemize}
As established above, tight frames $\Phi$ are useful because they give a redundant linear encoding $y=\Phi^*x$ of a signal $x$ that permits painless recovery: $x=\frac{1}{A}\Phi y$, where $A$ is the common squared-norm of the rows.  
Constructing tight frames is rather simple: perform Gram-Schmidt on the rows of any frame to orthogonalize with equal norms.  
For the sake of democracy in the entries of the encoding $y$, some applications opt for a \emph{unit norm tight frame (UNTF)}~\cite{CasazzaK:03}, which has the additional property that
\begin{itemize}
 \item[(ii)]  the columns are unit-norm.
\end{itemize}
Constructing UNTFs has proven a bit more difficult, and there has been a lot of research to characterize these~\cite{BenedettoF:03,CahillFMPS:11,Strawn:11}.  
As a special example of a UNTF, take any rows from a discrete Fourier transform matrix and normalize the resulting columns.  
In addition to unit-norm tightness, it is often beneficial to have the columns of $\Phi$ be incoherent, and this occurs when $\Phi$ is an \emph{equiangular tight frame (ETF)}, that is, a UNTF with the final property that
\begin{itemize}
 \item[(iii)]  the sizes of the inner products between distinct columns are equal.
\end{itemize}
ETFs do not exist for all matrix dimensions~\cite{BenedettoK:06}, and there are only three general constructions to date~\cite{FickusMT:10,Waldron:09,XiaZG:05}; these invoke block designs, strongly regular graphs, and difference sets, respectively.

To mitigate any confusion, the reader should be aware that throughout the literature, both UNTFs and ETFs are referred to as \emph{Welch-bound equality sequences}~\cite{Sarwate:99}.  
As one might expect, each achieves equality in one of two important inequalities, and it is important to review them.  
Consider $M\times N$ matrices $\Phi=[\varphi_1\cdots \varphi_N]$ which have (ii), but not necessarily (i) or (iii).  
As such, $\Phi$ might not be a frame, but we can still take the Hilbert-Schmidt norm of the Gram matrix of its columns:
\begin{equation*}
\|\Phi^*\Phi\|_{\mathrm{HS}}^2=\sum_{n=1}^N\sum_{n'=1}^N|\langle \varphi_n,\varphi_{n'}\rangle|^2.
\end{equation*}
This is oftentimes called the \emph{frame potential} of $\Phi$~\cite{BenedettoF:03}, and its significance will become apparent shortly.  
Since the columns of $\Phi$ have unit norm, and since $\Phi^*\Phi$ has at most $M$ nonzero eigenvalues, we have
\begin{equation*}
N^2=\big(\mathrm{Tr}(\Phi^*\Phi)\big)^2=\bigg(\sum_{m=1}^M\lambda_m(\Phi^*\Phi)\bigg)^2\leq M\sum_{m=1}^M\big(\lambda_m(\Phi^*\Phi)\big)^2=M\|\Phi^*\Phi\|_{\mathrm{HS}}^2,
\end{equation*}
where the inequality follows from the Cauchy-Schwarz inequality with the all-ones vector.  
As such, equality is achieved if and only if the $M$ largest eigenvalues of $\Phi^*\Phi$ are equal; since these are also the eigenvalues of $\Phi\Phi^*$, this implies that $\Phi\Phi^*$ is a multiple identity, and so $\Phi$ satisfies (ii).  
Thus, the frame potential of $\Phi$ satisfies $\|\Phi^*\Phi\|_{\mathrm{HS}}^2\geq \frac{N^2}{M}$, with equality if and only if $\Phi$ is a UNTF.  
Some call this the \emph{Welch bound}, and therefore say that UNTFs have Welch-bound equality.

Another bound is also (more correctly) referred to as the Welch bound, and its derivation uses the previous one.  
It concerns the \emph{worst-case coherence} of an $M\times N$ matrix $\Phi=[\varphi_1\cdots \varphi_N]$ that satisfies (ii):
\begin{equation*}
\mu:=\max_{\substack{n,n'\in\{1,\ldots,N\}\\n\neq n'}}|\langle \varphi_{n},\varphi_{n'}\rangle|.
\end{equation*}
Since the columns of $\Phi$ have unit norm, we have
\begin{equation*}
\frac{N^2}{M}
\leq\|\Phi^*\Phi\|_{\mathrm{HS}}^2
=\sum_{n=1}^N\sum_{n'=1}^N|\langle \varphi_n,\varphi_{n'}\rangle|^2
\leq N+N(N-1)\mu^2.
\end{equation*}
Again, equality is achieved in the first inequality if and only if $\Phi$ satisfies (i).  
Also, equality is achieved in the second inequality if and only if $\Phi$ satisfies (iii).  Rearranging gives the following:

\begin{thm}[Welch bound~\cite{StrohmerH:03,Welch:74}]
\label{thm.welch bound}
Every $M\times N$ matrix $\Phi$ with unit-norm columns has worst-case coherence
\begin{equation*}
\mu\geq\sqrt{\frac{N-M}{M(N-1)}},
\end{equation*}
with equality if and only if $\Phi$ is an equiangular tight frame.
\end{thm}

Equiangular lines have long been a subject of interest~\cite{LemmensS:73}, and since equiangular tight frames have minimal coherence, they are particularly useful in a number of applications.  Recent work on ETFs was spurred by results inspired by communication theory~\cite{BodmannP:05,HolmesP:04,StrohmerH:03} that show that the linear encoders provided by ETFs are optimally robust against channel erasures.  In the real setting, the existence of an ETF of a given size is equivalent to the existence of a strongly regular graph with certain corresponding parameters~\cite{HolmesP:04,Seidel:73}.  Such graphs have a rich history and remain an active topic of research~\cite{Brouwer:07}; the specific ETFs which arise from particular graphs are detailed in~\cite{Waldron:09}.  Some of this theory generalizes to the complex-variable setting in the guise of complex Seidel matrices~\cite{BodmannE:10,BodmannPT:09,DuncanHS:10}.  Many approaches to constructing ETFs have focused on the special case in which every entry of $\Phi$ is a root of unity~\cite{Kalra:06,Renes:07,Strohmer:08,SustikTDH:07,XiaZG:05}.  Other approaches are given in~\cite{CasazzaRT:08,Singh:10,TroppDHS:05}.  In the complex setting, much attention has focused on the \textit{maximal} case of $M^2$ vectors in $\mathbb{C}^M$~\cite{Appleby:05,Fickus:09,Khatirinejad:08,RenesBSC:04,ScottG:10}.

In the next chapter, we construct one of three known general families of ETFs, and we evaluate their performance as RIP matrices.
Having reviewed the frame-theoretic background for this thesis, the interested reader is encouraged to discover more about frame theory in~\cite{Christensen:02}.

\chapter{Steiner equiangular tight frames}

In this chapter, we provide a new method for constructing equiangular tight frames (ETFs), that is, matrices $\Phi$ with orthogonal and equal-norm rows, and unit-norm columns whose inner products are equal in modulus.
As discussed earlier, such frames have minimal worst-case coherence, and are therefore quite useful in applications.
However, up to this point, they have proven notoriously difficult to construct. 
By contrast, the construction of \emph{Steiner equianglar tight frames} is particularly simple: a tensor-like combination of a Steiner system and a regular simplex.  This simplicity permits us to resolve an open question regarding ETFs and the restricted isometry property (RIP): we show that the RIP performance of some ETFs is unfortunately no better than the so-called ``square-root bottleneck.''

In the next section, we provide some simple tests for demonstrating whether a given matrix is RIP; not only will this clarify the notion of the square-root bottleneck, it will show how ETFs are in some sense optimal as deterministic RIP matrices, thereby motivating the construction of ETFs.
Later, we provide the main result of this chapter, namely Theorem~\ref{theorem.steiner etfs}, which shows how certain Steiner systems may be combined with regular simplices to produce ETFs~\cite{FickusMT:11,FickusMT:10}.  In the third section, we discuss each of the known infinite families of such Steiner systems, and compute the corresponding infinite families of ETFs they generate.  We further provide some necessary and asymptotically sufficient conditions, namely Theorem~\ref{theorem.necessary conditions}, to aid in the quest for discovering other examples of such frames that lie outside of the known infinite families.   Finally, after demonstrating that Steiner ETFs fail to break the square-root bottleneck, we consider their application to the design of digital fingerprints to combat data piracy~\cite{MixonQKF:11,MixonQKF:12}.

\section{Simple tests for restricted isometry}

Before formally defining Steiner equiangular tight frames, we motivate their construction by reviewing a couple common methods for determining whether a matrix is RIP:
\medskip
\begin{center}
\begin{tabular}{ll}
\textbf{Positive test for RIP:} & Apply the Gershgorin circle theorem to the submatrices $\Phi_\mathcal{K}^*\Phi_\mathcal{K}^{}$. \\
\textbf{Negative test for RIP:} & Find a sparse vector in the nullspace of $\Phi$.
\end{tabular}
\end{center}
\medskip
In what follows, we discuss each of these tests in more detail, and later, we will use these tests to analyze Steiner ETFs as RIP matrices.

\subsection{Applying Gershgorin's circle thoerem}

Take an $M\times N$ matrix $\Phi$, and recall Definition~\ref{defn.rip}.
For a given $K$, we wish to find some $\delta$ for which $\Phi$ is $(K,\delta)$-RIP.
To this end, it is useful to consider the following expression for the restricted isometry constant:

\begin{lem} 
\label{lem:delta_min_bnd}
The smallest $\delta$ for which $\Phi$ is $(K,\delta)$-RIP is given by
\begin{equation}
\label{eq.delta min}
\delta_K
=\max_{\substack{\mathcal{K}\subseteq\{1,\ldots,N\}\\|\mathcal{K}|=K}}\|\Phi_\mathcal{K}^*\Phi_\mathcal{K}^{}-I_K\|_2,
\end{equation}
where $\Phi_\mathcal{K}$ denotes the submatrix consisting of columns of $\Phi$ indexed by $\mathcal{K}$.
\end{lem}

\begin{proof}
We first note that $\Phi$ being $(K,\delta)$-RIP trivially implies that $\Phi$ is $(K,\delta+\varepsilon)$-RIP for every $\varepsilon>0$.
It therefore suffices to show that the expression for $\delta_K$ in \eqref{eq.delta min} satisfies two criteria: (i) $\Phi$ is $(K,\delta_K)$-RIP, and (ii) $\Phi$ is not $(K,\delta)$-RIP for any $\delta<\delta_K$.
To this end, pick some $K$-sparse vector~$x$.
To prove (i), we need to show that
\begin{equation}
\label{eq.rip min}
(1-\delta_K)\|x\|^2
\leq\|\Phi x\|^2
\leq(1+\delta_K)\|x\|^2.
\end{equation}
Let $\mathcal{K}\subseteq\{1,\dots,N\}$ be the size-$K$ support of $x$, and let $x_\mathcal{K}$ be the corresponding subvector.
Then rearranging \eqref{eq.rip min} gives
\begin{equation}
\label{eq.delta min bound}
\delta_K
\geq\Big|\tfrac{\|\Phi x\|^2}{\|x\|^2}-1\Big|
=\Big|\tfrac{\langle \Phi_\mathcal{K}x_\mathcal{K},\Phi_\mathcal{K}x_\mathcal{K}\rangle-\langle x_\mathcal{K},x_\mathcal{K}\rangle}{\|x_\mathcal{K}\|^2}\Big|
=\Big|\Big\langle \tfrac{x_\mathcal{K}}{\|x_\mathcal{K}\|},(\Phi_\mathcal{K}^*\Phi_\mathcal{K}^{}-I_K)\tfrac{x_\mathcal{K}}{\|x_\mathcal{K}\|} \Big\rangle\Big|.
\end{equation}
Since the expression for $\delta_K$ in \eqref{eq.delta min} maximizes \eqref{eq.delta min bound} over all supports $\mathcal{K}$ and entry values $x_\mathcal{K}$, the inequality necessarily holds; that is, $\Phi$ is necessarily $(K,\delta_K)$-RIP.
Furthermore, equality is achieved by the support $\mathcal{K}$ which maximizes \eqref{eq.delta min} and the eigenvector $x_\mathcal{K}$ corresponding to the largest eigenvalue of $\Phi_\mathcal{K}^*\Phi_\mathcal{K}^{}-I_K$; this proves (ii).
\end{proof}

Note that we are not tasked with actually computing $\delta_K$; rather, we recognize that $\Phi$ is $(K,\delta)$-RIP for every $\delta\geq\delta_K$, and so we seek an upper bound on $\delta_K$.
The following classical result offers a particularly easy-to-calculate bound on eigenvalues:

\begin{thm}[Gershgorin circle theorem~\cite{Gerschgorin:31}]
For each eigenvalue $\lambda$ of a $K\times K$ matrix $A$, there is an index $i\in\{1,\ldots,K\}$ such that
\begin{equation*}
\Big|\lambda-A[i,i]\Big|\leq\sum_{\substack{j=1\\j\neq i}}^K\Big|A[i,j]\Big|.
\end{equation*}
\end{thm}

To use this theorem, take some $\Phi$ with unit-norm columns.
Note that $\Phi_\mathcal{K}^*\Phi_\mathcal{K}^{}$ is the Gram matrix of the columns indexed by $\mathcal{K}$, and as such, the diagonal entries are $1$, and the off-diagonal entries are inner products between distinct columns of $\Phi$.
Let $\mu$ denote the \emph{worst-case coherence} of $\Phi=[\varphi_1\cdots \varphi_N]$:
\begin{equation*}
\mu:=\max_{\substack{i,j\in\{1,\ldots,N\}\\i\neq j}}|\langle \varphi_i,\varphi_j\rangle|.
\end{equation*}
Then the size of each off-diagonal entry of $\Phi_\mathcal{K}^*\Phi_\mathcal{K}^{}$ is $\leq\mu$, regardless of our choice for $\mathcal{K}$.
Therefore, for every eigenvalue $\lambda$ of $\Phi_\mathcal{K}^*\Phi_\mathcal{K}^{}-I_K$, the Gershgorin circle theorem gives
\begin{equation}
\label{eq.bound}
|\lambda|
=|\lambda-0|
\leq\sum_{\substack{j=1\\j\neq i}}^K|\langle \varphi_i,\varphi_j\rangle|
\leq(K-1)\mu.
\end{equation}
Since \eqref{eq.bound} holds for every eigenvalue $\lambda$ of $\Phi_\mathcal{K}^*\Phi_\mathcal{K}^{}-I_K$ and every choice of $\mathcal{K}\subseteq\{1,\ldots,N\}$, we conclude from \eqref{eq.delta min} that $\delta_K\leq(K-1)\mu$, i.e., $\Phi$ is $(K,(K-1)\mu)$-RIP.
This process of using the Gershgorin circle theorem to demonstrate RIP for deterministic constructions has become standard in the community~\cite{ApplebaumHSC:09,DeVore:07,FickusMT:10}.

Recall that random RIP constructions support sparsity levels $K$ on the order of $\smash{\frac{M}{\log^\alpha N}}$ for some $\alpha\geq1$.
To see how well the Gershgorin circle theorem demonstrates RIP, we need to express $\mu$ in terms of $M$ and $N$.
To this end, we consider the Welch bound (Theorem~\ref{thm.welch bound}):
\begin{equation*}
\mu\geq\sqrt{\frac{N-M}{M(N-1)}}.
\end{equation*}
Since equiangular tight frames (ETFs) achieve equality in the Welch bound (as demonstrated in Section~\ref{section.intro to frame theory}), we can further analyze what it means for an $M\times N$ ETF $\Phi$ to be $(K,(K-1)\mu)$-RIP.
In particular, since Theorem~\ref{thm.rip use} requires that $\Phi$ be $(2K,\delta)$-RIP for $\delta<\sqrt{2}-1$, it suffices to have
$\smash{\frac{2K}{\sqrt{M}}<\sqrt{2}-1}$, since this implies
\begin{equation}
\label{eq.square root bottleneck}
\delta=(2K-1)\mu=(2K-1)\sqrt{\frac{N-M}{M(N-1)}}\leq\frac{2K}{\sqrt{M}}<\sqrt{2}-1.
\end{equation}
That is, ETFs form sensing matrices that support sparsity levels $K$ on the order of $\sqrt{M}$.
Most other deterministic constructions have identical bounds on sparsity levels~\cite{ApplebaumHSC:09,DeVore:07,FickusMT:10}.
In fact, since ETFs minimize coherence, they are necessarily optimal constructions in terms of the Gershgorin demonstration of RIP, but the question remains whether they are actually RIP for larger sparsity levels; the Gershgorin demonstration fails to account for cancellations in the sub-Gram matrices $\Phi_\mathcal{K}^*\Phi_\mathcal{K}^{}$, and so this technique is too weak to indicate either possibility.

\subsection{Spark considerations}

Recall that, in order for an inversion process for \eqref{eq.cs eqn} to exist, $\Phi$ must map $K$-sparse vectors injectively, or equivalently, every subcollection of $2K$ columns of $\Phi$ must be linearly independent.
This linear independence condition can be nicely expressed in more general terms, as the following definition provides:

\begin{defn}
The \emph{spark} of a matrix $\Phi$ is the size of the smallest linearly dependent subset of columns, i.e.,
\begin{equation*}
\mathrm{Spark}(\Phi)
=\min\Big\{\|x\|_0:\Phi x=0,~x\neq0\Big\}.
\end{equation*}
\end{defn}

This definition was introduced by Dohono and Elad~\cite{DonohoE:03} to help build a theory of sparse representation that later gave birth to modern compressed sensing.
The concept of spark is also found in matroid theory, where it goes by the name \emph{girth}.
The condition that every subcollection of $2K$ columns of $\Phi$ is linearly independent is equivalent to $\mathrm{Spark}(\Phi)>2K$.
Relating spark to RIP, suppose $\Phi$ is $(K,\delta)$-RIP with $\mathrm{Spark}(\Phi)\leq K$.
Then there exists a nonzero $K$-sparse vector $x$ such that $(1-\delta)\|x\|^2\leq\|\Phi x\|^2=0$, and so $\delta\geq1$.
The reason behind this stems from our necessary linear independence condition: RIP implies linear independence, and so small spark implies linear dependence, which in turn implies not RIP.

As an example of using spark to test RIP, consider the $M\times 2M$ matrix $\Phi=[I~~F]$ that comes from concatenating the identity matrix $I$ with the unitary discrete Fourier transform matrix $F$.
In this example, columns from a common orthonormal basis are orthogonal, while columns from different bases have an inner product of size $\frac{1}{\sqrt{M}}$.
As such, the Gershgorin analysis gives that $\Phi$ is $(K,\delta)$-RIP for all $\delta\geq\frac{K-1}{\sqrt{M}}$.
However, when $M$ is a perfect square, the Dirac comb $x$ of $\sqrt{M}$ Kronecker deltas is an eigenvector of $F$, and so concatenating $Fx$ with $-x$ produces a $2\sqrt{M}$-sparse vector in the nullspace of $\Phi$.
In other words, $\mathrm{Spark}(\Phi)\leq2\sqrt{M}$, and so $\Phi$ is not $(K,1-\varepsilon)$-RIP for any $K\geq2\sqrt{M}$.
After building Steiner equiangular tight frames, we will see that they perform similarly as RIP matrices.

\section{Constructing Steiner equiangular tight frames}

Steiner systems and block designs have been studied for over a century; the background facts presented here on these topics are taken from~\cite{AbelG:07,ColbournM:07}.  In short, a $(v,b,r,k,\lambda)$-\textit{block design} is a $v$-element set $V$ along with a collection $\calB$ of $b$ size-$k$ subsets of $V$, dubbed \textit{blocks}, that have the property that any element of $V$ lies in exactly $r$ blocks and that any $2$-element subset of $V$ is contained in exactly $\lambda$ blocks.  The corresponding \textit{incidence matrix} is a $v\times b$ matrix $A$ that is one in a given entry if that block contains the corresponding point, and is otherwise zero; in this chapter, it is more convenient for us to work with the $b\times v$ transpose $A^\rmT$ of this incidence matrix.  Our particular construction of ETFs involves a special class of block designs known as $(2,k,v)$-\textit{Steiner systems}.  These have the property that any $2$-element subset of $V$ is contained in exactly one block, that is, $\lambda=1$.  With respect to our purposes, the crucial facts are the following:
\medskip
\begin{quote}
The transpose $A^\rmT$ of the $\set{0,1}$-incidence matrix $A$ of a $(2,k,v)$-Steiner system:
\begin{enumerate}
\item[(i)]
is of size $\smash{\frac{v(v-1)}{k(k-1)}}\times v$,
\item[(ii)]
has $k$ ones in each row,
\item[(iii)]
has $\smash{\frac{v-1}{k-1}}$ ones in each column, and
\item[(iv)]
has the property that any two of its columns have a inner product of one.
\end{enumerate} 
\end{quote}
\medskip
The first three facts follow immediately from solving for $\smash{b=\frac{v(v-1)}{k(k-1)}}$ and $\smash{r=\frac{v-1}{k-1}}$, using the well-known relations $vr=bk$ and $r(k-1)=\lambda(v-1)$.  Meanwhile, (iv) comes from the fact that $\lambda=1$: each column of $A^\rmT$ corresponds to an element of the set, and the inner product of any two columns computes the number of blocks that contains the corresponding pair of points.  This in hand, we present the main result of this chapter;  here, the \textit{density} of a matrix is the ratio of the number of nonzero entries of that matrix to the total number of its entries:

\begin{thm}
\label{theorem.steiner etfs}
Every $(2,k,v)$-Steiner system generates an equiangular tight frame consisting of $N=v(1+\frac{v-1}{k-1})$ vectors in $M=\frac{v(v-1)}{k(k-1)}$-dimensional space with redundancy $\smash{\frac NM=k(1+\frac{k-1}{v-1})}$ and density $\smash{\frac{k}{v}=(\frac{N-1}{M(N-M)})^{\frac12}}$.\medskip

\noindent
Moreover, if there exists a real Hadamard matrix of size $\smash{1+\frac{v-1}{k-1}}$, then such frames are real.\medskip

\noindent
Specifically, a $\frac{v(v-1)}{k(k-1)}\times v(1+\frac{v-1}{k-1})$ ETF matrix $\Phi$ may be constructed as follows:
\begin{enumerate}
\item
Let $A^\rmT$ be the $\smash{\frac{v(v-1)}{k(k-1)}}\times v$ transpose of the adjacency matrix of a $(2,k,v)$-Steiner system.\smallskip
\item
For each $j=1,\dotsc,v$, let $H_j$ be any $(1+\frac{v-1}{k-1})\times(1+\frac{v-1}{k-1})$ matrix that has orthogonal rows and unimodular entries, such as a possibly complex Hadamard matrix.\smallskip
\item
For each $j=1,\dotsc,v$, let $\Phi_j$ be the $\smash{\frac{v(v-1)}{k(k-1)}\times(1+\frac{v-1}{k-1})}$ matrix obtained from the $j$th column of $A^\rmT$ by replacing each of the one-valued entries with a distinct row of $H_j$, and every zero-valued entry with a row of zeros.\smallskip
\item
Concatenate and rescale the $\Phi_j$'s to form  $\Phi=(\frac{k-1}{v-1})^\frac12[\Phi_1 \cdots \Phi_v]$.
\end{enumerate}
\end{thm}

It is important to note that a version of this ETF construction was previously employed by Seidel in Theorem~12.1 of~\cite{Seidel:73} to prove the existence of certain strongly regular graphs.  In the context of that result, our contributions are as follows: (i) the realization that when Seidel's block design arises from a particular type of Steiner system, the resulting strongly regular graph indeed corresponds to a real ETF; (ii) noting that in this case, the graph theory may be completely bypassed, as the idea itself directly produces the requisite frame $\Phi$; and (iii) having bypassed the graph theory, realizing that this construction immediately generalizes to the complex-variable setting if Seidel's requisite Hadamard matrix is permitted to become complex.  These realizations permit us to exploit the vast literature on Steiner systems~\cite{ColbournM:07} to construct several new infinite families of ETFs, in both the real and complex settings.  Moreover, these ETFs are extremely sparse in their native space; sparse tight frames have recently become a subject of interest in their own right~\cite{CasazzaHKK:10}.

We refer to the ETFs produced by Theorem~\ref{theorem.steiner etfs} as \textit{$(2,k,v)$-Steiner ETFs}.  In essence, the idea of the construction is that the nonzero rows of any particular $\Phi_j$ form a regular simplex in $\smash{\frac{v-1}{k-1}}$-dimensional space; these vectors are automatically equiangular amongst themselves; by requiring the entries of these simplices to be unimodular, and requiring that distinct blocks have only one entry of mutual support, one can further control the inner products of vectors arising from distinct blocks.  This idea is best understood by considering a simple example, such as the ETF that arises from a $(2,2,4)$-Steiner system whose transposed incidence matrix is
\begin{equation*}
A^\rmT=\begin{bmatrix}+&+&&\\+&&+&\\+&&&+\\&+&+&\\&+&&+\\&&+&+\end{bmatrix}.
\end{equation*}
One can immediately verify that $A^\rmT$ corresponds to a block design: there is a set $V$ of $v=4$ elements, each corresponding to a column of $A^\rmT$; there is also a collection $\calB$ of $b=6$ subsets of $V$, each corresponding to a row of $A^\rmT$; every row contains $k=2$ elements; every column contains $r=3$ elements; any given pair of elements is contained in exactly one row, that is, $\lambda=1$, a fact which is equivalent to having the inner product of any two distinct columns of $A^\rmT$ being $1$.  To form an ETF, for each of the four columns of $A^\rmT$ we must choose a $4\times 4$ matrix $H$ with unimodular entries and orthogonal rows; the size of $H$ is always one more than the number $r$ of ones in a given column of $A^\rmT$.  Though in principle one may choose a different $H$ for each column, we choose them all to be the same, namely the Hadamard matrix:
\begin{equation*}
H=\begin{bmatrix}+&+&+&+\\+&-&+&-\\+&+&-&-\\+&-&-&+\end{bmatrix}.
\end{equation*}
To form the ETF, for each column of $A^\rmT$ we replace each of its $1$-valued entries with a distinct row of $H$.  Again, though in principle one may choose a different sequence of rows of $H$ for each column, we simply decide to use the second, third and fourth rows, in that order.  The result is a real ETF of $N=16$ elements of dimension $M=6$:
\begin{equation}
\label{eq.steiner etf example}
\Phi=\frac1{\sqrt{3}}\left[\begin{array}{cccccccccccccccc}+&-&+&-&+&-&+&-\\+&+&-&-&&&&&+&-&+&-\\+&-&-&+&&&&&&&&&+&-&+&-\\&&&&+&+&-&-&+&+&-&-\\&&&&+&-&-&+&&&&&+&+&-&-\\&&&&&&&&+&-&-&+&+&-&-&+  \end{array}\right].
\end{equation}
One can immediately verify that the rows of $\Phi$ are orthogonal and have constant norm, implying $\Phi$ is indeed a tight frame.  One can also easily see that the inner products of two columns from the same block are $-\frac13$, while the inner products of columns from distinct blocks are $\pm\frac13$.  Theorem~\ref{theorem.steiner etfs} states that this behavior holds in general for any appropriate choice of $A^\rmT$ and $H$.

\begin{proof}[Proof of Theorem~\ref{theorem.steiner etfs}]
To verify $\Phi$ is a tight frame, note that the inner product of any two distinct rows of $\Phi$ is zero, as they are the sum of the inner products of the corresponding rows of the $\Phi_j$'s over all $j=1,\dotsc,v$; for any $j$, these shorter inner products are necessarily zero, as they either correspond to inner products of distinct rows of $H_j$ or to inner products with zero vectors.  Moreover, the rows of $\Phi$ have constant norm: as noted in (ii) above, each row of $A^\rmT$ contains $k$ ones; since each $H_j$ has unimodular entries, the squared-norm of any row of $\Phi$ is the squared-scaling factor $\frac{k-1}{v-1}$ times a sum of $\smash{k(1+\frac{v-1}{k-1})}$ ones, which, as is necessary for any unit norm tight frame, equals the redundancy $\smash{\frac NM=k(1+\frac{k-1}{v-1})}$.

Having that $\Phi$ is tight, we show $\Phi$ is also equiangular.  We first note that the columns of $\Phi$ have unit norm: the squared-norm of any column of $\Phi$ is $\smash{\frac{k-1}{v-1}}$ times the squared-norm of a column of one of the $\Phi_j$'s; since the entries of $H_j$ are unimodular and (iii) above gives that each column of $A^\rmT$ contains $\smash{\frac{v-1}{k-1}}$ ones, the squared-norm of any column of $\Phi$ is $\smash{(\frac{k-1}{v-1})(\frac{v-1}{k-1})1=1}$, as claimed.  Moreover, the inner products of any two distinct columns of $\Phi$ has constant modulus.  Indeed, the fact (iv) that any two distinct columns of $A^\rmT$ have but a single entry of mutual support implies the same is true for columns of $\Phi$ that arise from distinct $\Phi_j$ blocks, implying the inner product of such columns is $\smash{\frac{k-1}{v-1}}$ times the product of two unimodular numbers.  That is, the squared-magnitude of the inner products of two columns that arise from distinct blocks is $\smash{\frac{N-M}{M(N-1)}=(\frac{k-1}{v-1})^2}$, as needed.  Meanwhile, the same holds true for columns that arise from the same block $\Phi_j$.  To see this, note that since $H_j$ is a scalar multiple of a unitary matrix, its columns are orthogonal.  Moreover, $\Phi_j$ contains all but one of the $H_j$'s rows, namely one for each of the $1$-valued entries of $A^\rmT$, \`{a} la (iii).  Thus, the inner products of the portions of $H_j$ that lie in $\Phi_j$ are their entire inner product of zero, less the contribution from the left-over entries.  Overall, the inner product of two columns of $\Phi$ that arise from the same $\Phi_j$ block is $\smash{\frac{k-1}{v-1}}$ times the negated product of one entry of $H_j$ and the conjugate of another; since the entries of $H_j$ are unimodular, we have that the squared-magnitude of such inner products is $\smash{\frac{N-M}{M(N-1)}=(\frac{k-1}{v-1})^2}$, as needed.

Thus $\Phi$ is an ETF.  Moreover, as noted above, its redundancy is $\smash{\frac NM=k(1+\frac{k-1}{v-1})}$.  All that remains to verify is its density: as the entries of each $H_j$ are all nonzero, the proportion of $\Phi$'s nonzero entries is the same as that of the incidence matrix $A$, which is clearly $\frac kv$, having $k$ ones in each $v$-dimensional row.  Moreover, substituting $\smash{N=v(1+\frac{v-1}{k-1})}$ and $\smash{M=\frac{v(v-1)}{k(k-1)}}$ into the quantity $\smash{\frac{N-1}{M(N-M)}}$ reveals it to be $\frac{k^2}{v^2}$, and so the density can be alternatively expressed as $\smash{(\frac{N-1}{M(N-M)})^{\frac12}}$.
\end{proof}

In the next section, we apply Theorem~\ref{theorem.steiner etfs} to produce several infinite families of Steiner ETFs.  Before doing so, however, we pause to remark on the redundancy and sparsity of such frames.  In particular, note that since the parameters $k$ and $v$ of the requisite Steiner system always satisfy $2\leq k\leq v$, the redundancy $k(1+\frac{k-1}{v-1})$ of Steiner ETFs is always between $k$ and $2k$; the redundancy is therefore on the order of $k$, and is always strictly greater than $2$.  If a low-redundancy ETF is desired, one can always take the Naimark complement~\cite{CasazzaFMWZ:11} of an ETF of $N$ elements in $M$-dimensional space to produce a new ETF of $N$ elements in $(N-M)$-dimensional space; though the complement process does not preserve sparsity, it nevertheless transforms any Steiner ETF into a new ETF whose redundancy is strictly less than $2$.  However, such a loss of sparsity should not be taken lightly.  Indeed, the low density of Steiner ETFs gives them a large computational advantage over their non-sparse brethren.  

To clarify, the most common operation in frame-theoretic applications is the evaluation of the \textit{analysis} operator $\Phi^*$ on a given $x\in\mathbb{C}^M$.  For a non-sparse $\Phi$, this act of computing $\Phi^*x$ requires $\rmO(MN)$ operations; for a frame $\Phi$ of density $D$, this cost is reduced to $\rmO(DMN)$.  Indeed, using the explicit value of $\smash{D=(\frac{N-1}{M(N-M)})^{\frac12}}$ given in Theorem~\ref{theorem.steiner etfs} as well as the aforementioned fact that the redundancy of such frames necessarily satisfies $\frac NM>2$, we see that the cost of evaluating $\Phi^*x$ when $\Phi$ is a Steiner ETF is on the order of $\smash{(\frac{M(N-1)}{N-M})^{\frac12}N <(2M)^\frac{1}{2}N}$ operations, a dramatic cost savings when $M$ is large.  Further efficiency is gained when $\Phi$ is real, as its nonzero elements are but a fixed scaling factor times the entries of a real Hadamard matrix, implying $\Phi^*x$ can be evaluated using only additions and subtractions.  The fact that every entry of $\Phi$ is either $0$ or $\pm 1$ further makes real Steiner ETFs potentially useful for applications that require binary measurements, such as design of experiments.

\section{Examples of Steiner equiangular tight frames}
\label{section.examples of Steiner ETFs}
In this section, we apply Theorem~\ref{theorem.steiner etfs} to produce several infinite families of Steiner ETFs.   When designing frames for real-world applications, three considerations reign supreme: size, redundancy and sparsity.  As noted above, every Steiner ETF is very sparse, a serious computational advantage in high-dimensional signal processing.  Moreover, some of these infinite families, such as those arising from finite affine and projective geometries, provide great flexibility in choosing the ETF's size and redundancy.  Indeed, these constructions provide the first known guarantee that for a given application, one is always able to find ETFs whose frame elements lie in a space whose dimension matches, up to an order of magnitude, that of one's desired class of signals, while simultaneously permitting one to have an almost arbitrary fixed level of redundancy, a handy weapon in the fight against noise.  To be clear, recall that the redundancy of a Steiner ETF is always strictly greater than~$2$.  Moreover, general bounds on the maximal number of equiangular lines~\cite{LemmensS:73} require that any real $M\times N$ ETF satisfy $\smash{N\leq\frac{M(M+1)}{2}}$ and any complex ETF satisfy $N\leq M^2$; thus, the redundancy of an ETF is never truly arbitrary.  Nevertheless, if one prescribes a given level of redundancy in advance, the Steiner method can produce arbitrarily large ETFs whose redundancy is approximately the prime power closest to the desired level.  

\subsection{Infinite families of Steiner equiangular tight frames}
We now detail eight infinite families of ETFs, each generated by applying Theorem~\ref{theorem.steiner etfs} to one of the eight completely understood infinite families of $(2,k,v)$-Steiner systems.  Table~\ref{table.infinite families} summarizes the most important features of each family, and Table~\ref{table.low-dimensional examples} gives the first few examples of each type, summarizing those that lie in 100 dimensions or less.

\subsubsection{All two-element blocks: $(2,2,v)$-Steiner ETFs for any $v\geq2$.}

The first infinite family of Steiner systems is so simple that it is usually not discussed in the design-theory literature.  For any $v\geq 2$, let $V$ be a $v$-element set, and let $\calB$ be the collection of all $2$-element subsets of $V$.  Clearly, we have $\smash{b=\frac{v(v-1)}{2}}$ blocks, each of which contains $k=2$ elements; each point is contained in $r=v-1$ blocks, and each pair of points is indeed contained in but a single block, that is, $\lambda=1$.

By Theorem~\ref{theorem.steiner etfs}, the ETFs arising from these $(2,2,v)$-Steiner systems consist of $N=v(1+\frac{v-1}{k-1})=v^2$ vectors in $\smash{M=\frac{v(v-1)}{k(k-1)}=\frac{v(v-1)}{2}}$-dimensional space.  Though these frames can become arbitrarily large, they do not provide any freedom with respect to redundancy: $\smash{\frac NM=2\frac{v}{v-1}}$ is essentially $2$.  These frames have density $\smash{\frac kv=\frac2v}$.  Moreover, these ETFs can be real-valued if there exists a real Hadamard matrix of size $\smash{1+\frac{v-1}{k-1}}=v$.  In particular, it suffices to have $v$ to be a power of $2$; should the Hadamard conjecture prove true, it would suffice to have $v$ divisible by $4$.

One example of such an ETF with $v=4$ was given in the previous section.  For a complex example, consider $v=3$.  The $b\times v$ transposed incidence matrix $A^\rmT$ is $3\times 3$, with each row corresponding to a given $2$-element subset of $\set{0,1,2}$:
\begin{equation*}
A^\rmT=\begin{bmatrix}+&+&\\+&&+\\&+&+\end{bmatrix}.
\end{equation*}
To form the corresponding $3\times 9$ ETF $\Phi$, we need a $3\times 3$ unimodular matrix with orthogonal rows, such as a DFT; letting $\smash{\omega=\rme^{2\pi\rmi/3}}$, we can take
\begin{equation*}
H=\left[\begin{array}{lll}1&1&1\\1&\omega^2&\omega\\1&\omega&\omega^2\end{array}\right].
\end{equation*}
To form $\Phi$, in each column of $A^\rmT$, we replace each $1$-valued entry with a distinct row of $H$.  Always choosing the second and third rows yields an ETF of $9$ elements in $\bbC^3$:
\begin{equation*}
\Phi=\frac1{\sqrt{2}}\left[\begin{array}{lllllllll}1&\omega^2&\omega&1&\omega^2&\omega&&&\\1&\omega&\omega^2&&&&1&\omega^2&\omega\\&&&1&\omega&\omega^2&1&\omega&\omega^2\end{array}\right].
\end{equation*}
This is the only known instance of when the Steiner-based construction of Theorem~\ref{theorem.steiner etfs} produces a maximal ETF, that is, one that has $N=M^2$.

\subsubsection{Steiner triple systems: $(2,3,v)$-Steiner ETFs for any $v\equiv 1,3\bmod 6$.}

\textit{Steiner triple systems}, namely $(2,3,v)$-Steiner systems, have been a subject of interest for over a century, and are known to exist precisely when $v\equiv 1,3\bmod 6$~\cite{ColbournM:07}.  Each of the $\smash{b=\frac{v(v-1)}{6}}$ blocks contains $k=3$ points, while each point is contained in $\smash{r=\frac{v-1}{2}}$ blocks.  The corresponding ETFs produced by Theorem~\ref{theorem.steiner etfs} consist of $\smash{\frac{v(v+1)}{2}}$ vectors in $\smash{\frac{v(v-1)}6}$-dimensional space.  The density of such frames is $\frac3v$.  As with ETFs stemming from $2$-element blocks, Steiner triple systems offer little freedom in terms of redundancy: $\smash{\frac NM=3\frac{v+1}{v-1}}$ is always approximately $3$.  Such ETFs can be real if there exists a real Hadamard matrix of size $\smash{\frac{v+1}{2}}$.

\subsubsection{Four element blocks: $(2,4,v)$-Steiner ETFs for any $v\equiv 1,4\bmod 12$.}

It is known that $(2,4,v)$-Steiner systems exist precisely when $v\equiv 1,4\bmod 12$~\cite{AbelG:07}.  Continuing the trend of the previous two families, these ETFs can vary in size but not in redundancy: they consist of $\smash{\frac{v(v+2)}{3}}$ vectors in $\smash{\frac{v(v-1)}{12}}$-dimensional space, having redundancy $\smash{4\frac{v+2}{v-1}}$ and density $\smash{\frac4v}$.  Interestingly, such frames can never be real: with the exception of the trivial $1\times 1$ and $2\times 2$ cases, the dimensions of all real Hadamard matrices are divisible by $4$; since $v\equiv 1,4\bmod 12$, the requisite matrices $H$ here are of size $\smash{\frac{v+2}{3}}\equiv1,2\bmod4$.

\subsubsection{Five element blocks: $(2,5,v)$-Steiner ETFs for any $v\equiv 1,5\bmod 20$.}

It is also known that $(2,5,v)$-Steiner systems exist precisely when $v\equiv 1,5\bmod 20$~\cite{AbelG:07}.  The corresponding ETFs consist of $\smash{\frac{v(v+3)}{4}}$ vectors in $\smash{\frac{v(v-1)}{20}}$-dimensional space, having redundancy $\smash{5\frac{v+3}{v-1}}$ and density $\smash{\frac5v}$.  Such frames can be real whenever there exists a real Hadamard matrix of size $\frac{v+3}4$.  In particular, letting $v=45$, we see that there exists a real Steiner ETF of $540$ vectors in $99$-dimensional space, a fact not obtained from any other known infinite family.

\subsubsection{Affine geometries: $(2,q,q^n)$-Steiner ETFs for any prime power $q$, $n\geq 2$.}

At this point, the constructions depart from those previously considered, allowing both $k$ and $v$ to vary.  In particular, using techniques from finite geometry, one can show that for any prime power $q$ and any $n\geq2$, there exists a $(2,k,v)$-Steiner system with $k=q$ and $v=q^n$~\cite{ColbournM:07}.  The corresponding ETFs consist of $\smash{q^n(1+\frac{q^n-1}{q-1})}$ vectors in $\smash{q^{n-1}(\frac{q^n-1}{q-1}})$-dimensional space.  Like the preceding four classes of Steiner ETFs, these frames can grow arbitrarily large: fixing any prime power $q$, one may manipulate $n$ to produce ETFs of varying orders of magnitude.  However, unlike the four preceding classes, these affine Steiner ETFs also provide great flexibility in choosing redundancy.  That is, they provide the ability to pick $M$ and $N$ somewhat independently.  Indeed, the redundancy of such frames $q(1+\frac{q-1}{q^n-1})$ is essentially $q$, which may be an arbitrary prime power.  Moreover, as these frames grow large, they also become increasingly sparse: their density is $\smash{\frac1{q^{n-1}}}$.  Because of their high sparsity and flexibility with regards to size and redundancy, these frames, along with their projective geometry-based cousins detailed below, are perhaps the best known candidates for use in ETF-based applications.  Such ETFs can be real if there exists a real Hadamard matrix of size $1+\frac{q^n-1}{q-1}$, such as whenever $q=2$, or when $q=5$ and $n=3$.

\subsubsection{Projective geometries: $(2,q+1,\frac{q^{n+1}-1}{q-1})$-Steiner ETFs for any prime power $q$, $n\geq 2$.}

With finite geometry, one can show that for any prime power $q$ and any $n\geq2$, there exists a $(2,k,v)$-Steiner system with $k=q+1$ and $\smash{v=\frac{q^{n+1}-1}{q-1}}$~\cite{ColbournM:07}.  Qualitatively speaking, the ETFs that these projective geometries generate share much in common with their affinely generated cousins, possessing very high sparsity and great flexibility with respect to size and redundancy.  The technical details are as follows: they consist of $\smash{\frac{q^{n+1}-1}{q-1}(1+\frac{q^n-1}{q-1})}$ vectors in $\smash{\frac{(q^n-1)(q^{n+1}-1)}{(q+1)(q-1)^2}}$-dimensional space, with density $\smash{\frac{q^2-1}{q^{n+1}-1}}$ and redundancy $\smash{(q+1)(1+\frac{q-1}{q^n-1})}$.  These frames can be real if there exists a real Hadamard matrix of size $1+\frac{q^n-1}{q-1}$; note this restriction is identical to the one for ETFs generated by affine geometries for the same $q$ and $n$, implying that real Steiner ETFs generated by finite geometries always come in pairs, such as the $6\times 16$ and $7\times 28$ ETFs generated when $q=2$, $n=2$, and the $28\times 64$ and $35\times 120$ ETFs generated when $q=2$, $n=3$.

\subsubsection{Unitals: $(2,q+1,q^3+1)$-Steiner ETFs for any prime power $q$.}

For any prime power $q$, one can show that there exists a $(2,k,v)$-Steiner system with $k=q+1$ and $v=q^3+1$~\cite{ColbournM:07}.  Though one may pick a redundancy of one's liking, such a choice confines one to ETFs of a given size: they consist of $(q^2+1)(q^3+1)$ vectors in $\smash{\frac{q^2(q^3+1)}{q+1}}$-dimensional space, having redundancy $\smash{(q+1)(1+\frac1{q^2})}$ and density $\smash{\frac{q+1}{q^3+1}}$.  These ETFs can never be real: the requisite Hadamard matrices are of size $q^2+1$ which is never divisible by $4$ since $0$ and $1$ are the only squares in $\bbZ_4$.

\subsubsection{Denniston designs: $(2,2^r,2^{r+s}+2^r-2^s)$-Steiner ETFs for any $2\leq r<s$.}

For any $2\leq r<s$, one can show that there exists a $(2,k,v)$-Steiner system with $k=2^r$ and $v=2^{r+s}+2^r-2^s$~\cite{ColbournM:07}.  By manipulating $r$ and $s$, one can independently determine the order of magnitude of redundancy and size: the corresponding ETFs consist of $(2^s+2)(2^{r+s}+2^r-2^s)$ vectors in $\smash{\frac{(2^s+1)(2^{r+s}+2^r-2^s)}{2^r}}$-dimensional space, having redundancy $\smash{2^r\frac{2^s+2}{2^s+1}}$ and density $\smash{\frac{2^r}{2^{r+s}+2^r-2^s}}$.  As such, this family has some qualitative similarities to the familes of ETFs produced by affine and projective geometries.  However, unlike those families, the ETFs produced by Denniston designs can never be real: the requisite Hadamard matrices are of size $2^s+2$, which is never divisible by $4$.

\begin{table}
\begin{center}
\begin{scriptsize}
\begin{tabular}{llllll}
\hline
Name		&$M$										&$N$											&Redundancy						&Real?					&Restrictions\\
\hline
$2$-blocks	&$\frac{v(v-1)}{2}$							&$v^2$											&$2\frac{v}{v-1}$				&$v$					&None\medskip\\
$3$-blocks	&$\frac{v(v-1)}6$							&$\frac{v(v+1)}{2}$								&$3\frac{v+1}{v-1}$				&$\frac{v+1}{2}$		&$v\equiv 1,3\bmod 6$\medskip\\
$4$-blocks	&$\frac{v(v-1)}{12}$						&$\frac{v(v+2)}{3}$								&$4\frac{v+2}{v-1}$				&Never					&$v\equiv 1,4\bmod 12$\medskip\\
$5$-blocks	&$\frac{v(v-1)}{20}$						&$\frac{v(v+3)}{4}$								&$5\frac{v+3}{v-1}$				&$\frac{v+3}4$			&$v\equiv 1,5\bmod 20$\medskip\\
Affine		&$q^{n-1}(\frac{q^n-1}{q-1})$				&$q^n(1+\frac{q^n-1}{q-1})$						&$q(1+\frac{q-1}{q^n-1})$		&$1+\frac{q^n-1}{q-1}$	&prime power $q$, $n\geq2$\medskip\\
Projective	&$\frac{(q^n-1)(q^{n+1}-1)}{(q+1)(q-1)^2}$	&$\frac{q^{n+1}-1}{q-1}(1+\frac{q^n-1}{q-1})$	&$(q+1)(1+\frac{q-1}{q^n-1})$	&$1+\frac{q^n-1}{q-1}$	&prime power $q$, $n\geq2$\medskip\\
Unitals		&$\frac{q^2(q^3+1)}{q+1}$					&$(q^2\!+1)(q^3\!+1)$								&$(q+1)(1+\frac1{q^2})$			&Never					&prime power $q$\medskip\\
Denniston	&$\frac{(2^s+1)(2^{r+s}+2^r-2^s)}{2^r}$		&$(2^s\!+2)(2^{r+s}\!+2^r\!-2^s)$						&$2^r\frac{2^s+2}{2^s+1}$		&Never					&$2\leq r<s$\\
\hline
\end{tabular}
\end{scriptsize}
\caption{\label{table.infinite families}Eight infinite families of Steiner ETFs, each arising from a known infinite family of $(2,k,v)$-Steiner designs.  Each family permits both $M$ and $N$ to grow very large, but only a few families---affine, projective and Denniston---give one the freedom to simultaneously control the proportion between $M$ and $N$, namely the redundancy $\frac NM$ of the ETF.  The column denoted ``Real?" indicates the size for which a real Hadamard matrix must exist in order for the resulting ETF to be real; it suffices to have this size be a power of $2$; if the Hadamard conjecture is true, it would suffice for this number to be divisible by $4$.}
\end{center}
\end{table}

\begin{table}
\begin{center}
\begin{tabular}{rrrrrcl}
\hline
$M$&$N$&$k$&$v$&$r$&$\bbR/\bbC$&Construction of the Steiner system\\
\hline
  6&  16& 2& 4& 3&$\bbR$&$2$-blocks of $v=4$; Affine with $q=2$, $n=2$\\
  7&  28& 3& 7& 3&$\bbR$&$3$-blocks of $v=7$; Projective with $q=2$, $n=2$\\
 28&  64& 2& 8& 7&$\bbR$&$2$-blocks of $v=8$; Affine with $q=2$, $n=3$\\
 35& 120& 3&15& 7&$\bbR$&$3$-blocks of $v=15$; Projective with $q=2$, $n=3$\\
 66& 144& 2&12&11&$\bbR$&$2$-blocks of $v=12$\\
 99& 540& 5&45&11&$\bbR$&$5$-blocks of $v=45$\\ 
\hline
  3&   9& 2& 3& 2&$\bbC$&$2$-blocks of $v=3$\\
 10&  25& 2& 5& 4&$\bbC$&$2$-blocks of $v=5$\\
 12&  45& 3& 9& 4&$\bbC$&$3$-blocks of $v=9$; Affine with $q=3$, $n=2$\\
 13&  65& 4&13& 4&$\bbC$&$4$-blocks of $v=13$; Projective with $q=3$, $n=2$\\
 15&  36& 2& 6& 5&$\bbC$&$2$-blocks of $v=6$\\
 20&  96& 4&16& 5&$\bbC$&$4$-blocks of $v=16$; Affine with $q=4$, $n=2$\\
 21&  49& 2& 7& 6&$\bbC$&$2$-blocks of $v=7$\\
 21& 126& 5&21& 5&$\bbC$&$5$-blocks of $v=21$; Projective with $q=4$, $n=2$\\
 26&  91& 3&13& 6&$\bbC$&$3$-blocks of $v=13$\\
 30& 175& 5&25& 6&$\bbC$&$5$-blocks of $v=25$; Affine with $q=5$, $n=2$\\
 31& 217& 6&31& 6&$\bbC$&Projective with $q=5$, $n=2$\\
 36&  81& 2& 9& 8&$\bbC$&$2$-blocks of $v=9$\\
 45& 100& 2&10& 9&$\bbC$&$2$-blocks of $v=10$\\
 50& 225& 4&25& 8&$\bbC$&$4$-blocks of $v=25$\\
 55& 121& 2&11&10&$\bbC$&$2$-blocks of $v=11$\\
 56& 441& 7&49& 8&$\bbC$&Affine with $q=7$, $n=2$\\
 57& 190& 3&19& 9&$\bbC$&$3$-blocks of $v=19$\\
 57& 513& 8&57& 8&$\bbC$&Projective with $q=7$, $n=2$\\
 63& 280& 4&28& 9&$\bbC$&Unital with $q=3$; Denniston with $r=2$, $s=3$\\
 70& 231& 3&21&10&$\bbC$&$3$-blocks of $v=21$\\
 72& 640& 8&64& 9&$\bbC$&Affine with $q=8$, $n=2$\\
 73& 730& 9&73& 9&$\bbC$&Projective with $q=8$, $n=2$\\
 78& 169& 2&13&12&$\bbC$&$2$-blocks of $v=13$\\
 82& 451& 5&41&19&$\bbC$&$5$-blocks of $v=41$\\
 90& 891& 9&81&10&$\bbC$&Affine with $q=9$, $n=2$\\
 91& 196& 2&14&13&$\bbC$&$2$-blocks of $v=14$\\
 91&1001&10&91&10&$\bbC$&Projective with $q=9$, $n=2$\\
100& 325& 3&25&12&$\bbC$&$3$-blocks of $v=25$\\ 
\hline
\end{tabular}
\caption{\label{table.low-dimensional examples}The ETFs of dimension 100 or less that can be constructed by applying Theorem~\ref{theorem.steiner etfs} to the eight infinite families of Steiner systems detailed in Section~\ref{section.examples of Steiner ETFs}.  That is, these ETFs represent the first few examples of the general constructions summarized in Table~\ref{table.infinite families}.  For each ETF, we give the dimension $M$ of the underlying space, the number of frame vectors $N$, as well as the number $k$ of elements that lie in any block of a $v$-element set in the corresponding $(2,k,v)$-Steiner system.  We further give the value $r$ of the number of blocks that contain a given point; by Theorem~\ref{theorem.necessary conditions}, $\abs{\ip{f_n}{f_{n'}}}=\frac1r$ measures the angle between any two frame elements.  We also indicate whether the given frame is real or complex, and the method(s) of constructing the corresponding Steiner system.
}
\end{center}
\end{table}

\subsection{Conditions for the existence of Steiner equiangular tight frames}

$(2,k,v)$-Steiner systems have been actively studied for over a century, with many celebrated results.  Nevertheless, much about these systems is still unknown.  In this subsection, we discuss some known partial characterizations of the Steiner systems which lie outside of the eight families we have already discussed, as well as what these results tell us about the existence of certain ETFs.  To begin, recall that, for a given $k$ and $v$, if a $(2,k,v)$-Steiner system exists, then the number $r$ of blocks that contain a given point is necessarily $\smash{\frac{v-1}{k-1}}$, while the total number of blocks $b$ is $\smash{\frac{v(v-1)}{k(k-1)}}$.  As such, in order for a $(2,k,v)$-Steiner system to exist, it is necessary for $(k,v)$ to be \textit{admissible}, that is, to have the property that $\smash{\frac{v-1}{k-1}}$ and $\smash{\frac{v(v-1)}{k(k-1)}}$ are integers.

However, this property is not sufficient for existence: it is known that a $(2,6,16)$-Steiner system does not exist~\cite{AbelG:07} despite the fact that $\smash{\frac{v-1}{k-1}}=3$ and $\smash{\frac{v(v-1)}{k(k-1)}}=8$.  In fact, letting $v$ be either $16$, $21$, $36$, or $46$ results in an admissible pair with $k=6$, despite the fact that none of the corresponding Steiner systems exist; there are twenty-nine additional values of $v$ which form an admissible pair with $k=6$ and for which the existence of a corresponding Steiner system remains an open problem~\cite{AbelG:07}.  Similar nastiness arises with $k\geq 7$.  The good news is that admissibility, though not sufficient for existence, is, in fact, asymptotically sufficient: for any fixed $k$, there exists a corresponding admissible index $v_0(k)$ for which for all $v>v_0(k)$ such that $\smash{\frac{v-1}{k-1}}$ and $\smash{\frac{v(v-1)}{k(k-1)}}$ are integers, a $(2,k,v)$-Steiner system indeed exists~\cite{AbelG:07}.  Moreover, explicit values of $v_0(k)$ are known for small $k$: $v_0(6)=801$, $v_0(7)=2605$, $v_0(8)=3753$, $v_0(9)=16497$.  We now detail the ramifications of these design-theoretic results on frame theory:

\begin{thm}
\label{theorem.necessary conditions}
If an $M\times N$ Steiner equiangular tight frame exists, then letting $\smash{\alpha=(\frac{N-M}{M(N-1)})^{\frac12}}$, the corresponding block design has parameters:
\begin{equation*}
v=\tfrac{N\alpha}{1+\alpha},
\qquad
b=M,
\qquad
r=\tfrac1{\alpha},
\qquad
k=\tfrac{N}{M(1+\alpha)}.
\end{equation*}
In particular, if such a frame exists, then these expressions for $v$, $k$ and $r$ are necessarily integers.\medskip

\noindent
Conversely, for any fixed $k\geq2$, there exists an index $v_0(k)$ for which for all $v>v_0(k)$ such that $\smash{\frac{v-1}{k-1}}$ and $\smash{\frac{v(v-1)}{k(k-1)}}$ are integers, there exists a Steiner equiangular tight frame of $\smash{v(1+\frac{v-1}{k-1})}$ vectors for a space of dimension $\smash{\frac{v(v-1)}{k(k-1)}}$.\medskip

\noindent
In particular, for any fixed $k\geq2$, letting $v$ be either $jk(k-1)+1$ or $jk(k-1)+k$ for increasingly large values of $j$ results in a sequence of Steiner equiangular tight frames whose redundancy is asymptotically $k$; these frames can be real if there exist real Hadamard matrices of sizes $jk+1$ or $jk+2$, respectively.
\end{thm}

\begin{proof}
To prove the necessary conditions on $M$ and $N$, recall that Steiner ETFs, namely those ETFs produced by Theorem~\ref{theorem.steiner etfs}, have $\smash{N=v(1+\frac{v-1}{k-1})}$ and $\smash{M=\frac{v(v-1)}{k(k-1)}}$.  Together, these two equations imply $N=v+kM$.  Solving for $k$ and substituting the resulting expression into $\smash{N=v(1+\frac{v-1}{k-1})}$ yields the quadratic equation $0=(M-1)v^2+2(N-M)v-N(N-M)$.  With some algebra, the only positive root of this equation can be found to be $\smash{v=\frac{N\alpha}{1+\alpha}}$, as claimed.  Substituting this expression for $v$ into $N=v+kM$ yields $\smash{k=\tfrac{N}{M(1+\alpha)}}$.  Having $v$ and $k$, the previously mentioned relations $bk=vr$ and $v-1=r(k-1)$ imply $\smash{r=\frac{v-1}{k-1}=\frac1\alpha}$ and $\smash{b=\frac vkr=M}$, as claimed. 

The second set of conclusions is the result of applying Theorem~\ref{theorem.steiner etfs} to the aforementioned $(2,k,v)$-Steiner ETFs that are guaranteed to exist for all sufficiently large $v$, provided $\smash{\frac{v-1}{k-1}}$ and $\smash{\frac{v(v-1)}{k(k-1)}}$ are integers.  The final set of conclusions are then obtained by applying this fact in the special cases where $v$ is either $jk(k-1)+1$ or $jk(k-1)+k$.  In particular, if $v=jk(k-1)+1$ then $\smash{\frac{v-1}{k-1}=jk}$ and $M=\smash{\frac{v(v-1)}{k(k-1)}=j\big(jk(k-1)+1\big)}$ are integers, and the resulting ETF of $(jk+1)\big(jk(k-1)+1\big)$ vectors has a redundancy of $\smash{k+\frac1j}$ that tends to $k$ for large $j$; such an ETF can be real if there exists a real Hadamard matrix of size $jk+1$.  Meanwhile, if $v=jk(k-1)+k$ then $\smash{\frac{v-1}{k-1}=jk+1}$ and $M=\smash{\frac{v(v-1)}{k(k-1)}=(jk+1)\big(j(k-1)+1\big)}$ are integers, and the resulting ETF of $k(jk+2)\big(j(k-1)+1\big)$ vectors has a redundancy of $\smash{k\frac{jk+2}{jk+1}}$ that tends to $k$ for large $j$; such an ETF can be real if there exists a real Hadamard matrix of size $jk+2$.
\end{proof}
We conclude this section with a few thoughts on Theorems~\ref{theorem.steiner etfs} and~\ref{theorem.necessary conditions}.  First, we emphasize that the method of Theorem~\ref{theorem.steiner etfs} is a method for constructing some ETFs, and by no means constructs them all.  Indeed, as noted above, the redundancy of Steiner ETFs is always strictly greater than $2$; while some of those ETFs with $\frac NM<2$ will be the Naimark complements of Steiner ETFs, one must admit that the Steiner method contributes little towards the understanding of those ETFs with $\frac NM=2$, such as those arising from Paley graphs~\cite{Waldron:09}.  Moreover, Theorem~\ref{theorem.necessary conditions} implies that not even every ETF with $\frac NM>2$ arises from a Steiner system: though there exists an ETF of $76$-elements in $\bbR^{19}$~\cite{Waldron:09}, the corresponding parameters of the design would be $v=\frac{38}3$, $r=5$ and $k=\frac{10}3$, not all of which are integers.

That said, the method of Theorem~\ref{theorem.steiner etfs} is truly significant: comparing Table~\ref{table.low-dimensional examples} with a comprehensive list of all real ETFs of dimension $50$ or less~\cite{Waldron:09}, we see the Steiner method produces $4$ of the $17$ ETFs that have redundancy greater than $2$, namely $6\times 16$, $7\times 28$, $28\times 64$ and $35\times 120$ ETFs.  Interestingly, an additional $4$ of these $17$ ETFs can also be produced by the Steiner method, but only in complex form, namely those of $15\times 36$, $20\times 96$, $21\times 126$ and $45\times 100$ dimensions; it is unknown whether this is the result of a deficit in our analysis or the true non-existence of real-valued Steiner-based constructions of these sizes.  The plot further thickens when one realizes that an additional $2$ of these $17$ real ETFs satisfy the necessary conditions of Theorem~\ref{theorem.necessary conditions}, but that the corresponding $(2,k,v)$-Steiner systems are known to not exist: if a $28\times 288$ ETF was to arise as a result of Theorem~\ref{theorem.steiner etfs}, the corresponding Steiner system would have $k=6$ and $v=36$, while the $43\times 344$ ETF would have $k=7$ and $v=43$; in fact, $(2,6,36)$- and  $(2,7,43)$-Steiner systems cannot exist~\cite{AbelG:07}.  With our limited knowledge of the rich literature on Steiner systems, we were unable to resolve the existence of two remaining candidates: $23\times 276$ and $46\times 736$ ETFs could potentially arise from $(2,10,46)$- and $(2,14,92)$-Steiner systems, respectively, provided they exist.

\section{Restricted isometry and digital fingerprinting}

In the previous section, we used Theorem~\ref{theorem.steiner etfs} to construct many examples of Steiner ETFs.  In this section, we investigate the feasibility of using such frames for applications in sparse signal processing.  
Regarding restricted isometry, one of the sad consequences of the Steiner construction method in Theorem~\ref{theorem.steiner etfs} is that we now know there is a large class of ETFs for which the seemingly coarse estimate from the Gershgorin analysis~\eqref{eq.bound} is, in fact, accurate.  In particular, recall that Gershgorin guarantees that every $M\times N$ ETF is $(K,\delta)$-RIP whenever $K\leq\delta\sqrt{M}$.
Furthermore, recall from Theorem~\ref{theorem.steiner etfs} that every Steiner ETF is built by carefully overlapping $v$ regular simplices, each consisting of $r+1$ vectors in an $r$-dimensional subspace of $b$-dimensional space.  Thus, the corresponding subcollection of $r+1$ vectors that lie in a given block are linearly dependent.
Considering the value of $r$ given in Theorem~\ref{theorem.necessary conditions}, we see that Steiner ETFs $\Phi$ have
\begin{equation*}
\mathrm{Spark}(\Phi)
\leq r+1
=\sqrt{\frac{M(N-1)}{N-M}}+1
\leq\sqrt{\frac{MN}{N-N/2}}+1
=\sqrt{2M}+1,
\end{equation*}
where the last inequality uses the fact that Steiner ETFs have redundancy $\frac{N}{M}\geq2$.
Therefore, Steiner ETFs are not $(K,1-\varepsilon)$-RIP for any $K>\sqrt{2M}$, that is, they fail to break the square-root bottleneck.
This begs the open question: Are there any ETFs which are as RIP as random matrices, or does being optimal in the Gershgorin sense necessarily come at the cost of being able to support large sparsity levels?
In Chapter~3, we address this problem directly and make some interesting connections with graph theory and number theory, but we do not give a conclusive answer.

%%%%%%%   USER  DEFINED  COMMANDS
\newcommand{\mK}{\mathcal{K}}
\newcommand{\mC}{\mathcal{C}}
\newcommand{\maginnerprod}[2]{ \left | \! \left \langle #1 , #2  \right \rangle  \! \right |}
\newcommand{\innerprod}[2]{  \left \langle #1 , #2  \right \rangle }
\newcommand{\ExpVal}[1]{ \text{E} \left [ #1  \right ] }
\newcommand{\mCETF}{\mathcal{C}_{\text{ETF}}} 
\newcommand{\Kmax}{K_{\text{max}}}
\newcommand{\dmk}{\mathrm{dist}(\mathcal{G}_{K,n},\neg\mathcal{G}_{K,n})}  %distance guilty not guilty
\newcommand{\udmk}{\underline{\mathrm{dist}}(\mathcal{G}_{K,n},\neg\mathcal{G}_{K,n})}
\newcommand{\WNR}{\text{WNR}}
%%%%%%%

Despite their provably suboptimal performance as RIP matrices, we will see that Steiner ETFs are particularly well-suited for the application of digital fingerprints.
Digital media protection has become an important issue in recent years, as illegal distribution of licensed material has become increasingly prevalent.  
A number of methods have been proposed to restrict illegal distribution of media and ensure only licensed users are able to access it.  
One method involves cryptographic techniques, which encrypt the media before distribution.  
By doing this, only the users with appropriate licensed hardware or software have access; satellite TV and DVDs are two such examples.  
Unfortunately, cryptographic approaches are limited in that once the content is decrypted (legally or illegally), it can potentially be copied and distributed freely.

An alternate approach involves marking each copy of the media with a unique signature.  
The signature could be a change in the bit sequence of the digital file or some noise-like distortion of the media.  
The unique signatures are called \emph{fingerprints}, by analogy to the uniqueness of human fingerprints.  
With this approach, a licensed user could illegally distribute the file, only to be implicated by his fingerprint.
The potential for prosecution acts as a deterrent to unauthorized distribution.  
However, fingerprinting systems are vulnerable when multiple users form a \emph{collusion} by combining their copies to create a forged copy.  
This attack can reduce and distort the colluders' individual fingerprints, making identification of any particular user difficult.  
Some examples of potential attacks involve comparing the bit sequences of different copies, averaging copies in the signal space, as well as introducing noise, rotations, or cropping.  

One of the principal approaches to designing fingerprints with robustness to collusions uses what is called the \emph{distortion assumption}.
In this regime, fingerprints are noise-like distortions to the media in signal space.  
In order to preserve the overall quality of the media, limits are placed on the magnitude of this distortion. The content owner limits the power of the fingerprint he adds, and the collusion limits the power of the noise they add in their attack.  
When applying the distortion assumption, the literature typically assumes that the collusion linearly averages their individual copies to forge the host signal.
Also, while results using the distortion assumption tend to accommodate fewer users than those with other assumptions, this assumption is distinguished by its natural embedding of fingerprints, namely in the signal space.

Cox et al.~introduced one of the first robust fingerprint designs under the distortion assumption~\cite{cox1997secure}; the robustness was later analytically proven in~\cite{kilian1998resistance}.  Different fingerprint designs have since been studied, including orthogonal fingerprints~\cite{wang2005anti} and simplex fingerprints~\cite{kiyavash2009regular}.  
We propose ETFs as a fingerprint design under the distortion assumption, and we analyze their performance against the worst-case collusion~\cite{MixonQKF:11,MixonQKF:12}.
Using analysis from Ergun et al.~\cite{ergun1999note}, we will show that ETFs perform particularly well as fingerprints; as a matter of fact, Steiner ETF fingerprints perform comparably to orthogonal and simplex fingerprints on average, while accommodating several times as many users~\cite{MixonQKF:12}.  
We start by formally presenting the fingerprinting and collusion processes.

\subsection{Problem setup} 
\label{sec:problem_setup}

A content owner has a host signal that he wishes to share, but he wants to mark it with fingerprints before distributing it.
We view this host signal as a vector $s\in\mathbb{R}^M$, and the marked versions of this vector will be given to $N>M$ users.  
Specifically, the $n$th user is given 
\begin{equation*} 
\label{eq:fingerprint_assignement}
\hat{s}_n := s + \varphi_n,
\end{equation*} 
where $\varphi_n\in\mathbb{R}^M$ denotes the $n$th fingerprint; we assume the fingerprints have equal norm.
We wish to design the fingerprints $\{\varphi_n\}_{n=1}^N$ to be robust to a linear averaging attack.
In particular, let $\mathcal{K}\subseteq\{1,\ldots,N\}$ denote a collection of users who together make a different copy of the host signal.
Then their linear averaging attack produces a forgery:
\begin{equation}
\label{eq:dm:attack1}
f:=\sum_{k\in\mathcal{K}}x_k\hat{s}_k+z,
\qquad\sum_{k\in\mathcal{K}}x_k=1,
\qquad x_k\geq0~~~\forall k,
\end{equation}
where $z$ is a noise vector introduced by the colluders.  
This attack model is illustrated in Figure~\ref{fig:attack_channel}.

\begin{figure}[t]
\centering

\begin{picture}(300,260)(-30,-30)

\put(-27,98){$s$}

\put(0,-30){\dashbox{1}(66,260){}}
\put(16,-20){\tiny{fingerprint}}
\put(16,-25){\tiny{assignment}}
\put(-20,100){\line(1,0){40}}
\put(20,0){\line(0,1){200}}

% fingerprints

\put(20,200){\vector(1,0){20}}
\put(43,200){\circle{6}}
\put(43,200){\line(0,1){3}}
\put(43,200){\line(1,0){3}}
\put(43,200){\line(0,-1){3}}
\put(43,200){\line(-1,0){3}}
\put(43,213){\vector(0,-1){10}}
\put(38,217){$\varphi_1$}

\put(20,165){\vector(1,0){20}}
\put(43,165){\circle{6}}
\put(43,165){\line(0,1){3}}
\put(43,165){\line(1,0){3}}
\put(43,165){\line(0,-1){3}}
\put(43,165){\line(-1,0){3}}
\put(43,178){\vector(0,-1){10}}
\put(38,182){$\varphi_2$}

\put(20,130){\vector(1,0){20}}
\put(43,130){\circle{6}}
\put(43,130){\line(0,1){3}}
\put(43,130){\line(1,0){3}}
\put(43,130){\line(0,-1){3}}
\put(43,130){\line(-1,0){3}}
\put(43,143){\vector(0,-1){10}}
\put(38,147){$\varphi_3$}

\put(41,104){$\vdots$}

\put(20,70){\vector(1,0){20}}
\put(43,70){\circle{6}}
\put(43,70){\line(0,1){3}}
\put(43,70){\line(1,0){3}}
\put(43,70){\line(0,-1){3}}
\put(43,70){\line(-1,0){3}}
\put(43,83){\vector(0,-1){10}}
\put(32,87){$\varphi_{N-2}$}

\put(20,35){\vector(1,0){20}}
\put(43,35){\circle{6}}
\put(43,35){\line(0,1){3}}
\put(43,35){\line(1,0){3}}
\put(43,35){\line(0,-1){3}}
\put(43,35){\line(-1,0){3}}
\put(43,48){\vector(0,-1){10}}
\put(32,52){$\varphi_{N-1}$}

\put(20,0){\vector(1,0){20}}
\put(43,0){\circle{6}}
\put(43,0){\line(0,1){3}}
\put(43,0){\line(1,0){3}}
\put(43,0){\line(0,-1){3}}
\put(43,0){\line(-1,0){3}}
\put(43,13){\vector(0,-1){10}}
\put(37,17){$\varphi_N$}

% marked copies

\put(46,200){\vector(1,0){40}}
\put(88,198){$\hat{s}_1$}

\put(92,104){$\vdots$}

\put(46,70){\vector(1,0){40}}
\put(88,68){$\hat{s}_{N-2}$}

\put(46,35){\vector(1,0){40}}
\put(88,33){$\hat{s}_{N-1}$}

\put(113,-15){$\underbrace{}_{\mathcal{K}}$}

\put(150,-30){\dashbox{1}(92,260){}}
\put(159,-20){\tiny{linear-average-plus-noise}}
\put(173,-25){\tiny{forgery process}}

\put(46,165){\vector(1,0){70}}
\put(118,163){$\hat{s}_2$}
\put(130,165){\vector(1,0){40}}
\put(173,165){\circle{6}}
\put(173,165){\line(1,1){2}}
\put(173,165){\line(-1,1){2}}
\put(173,165){\line(1,-1){2}}
\put(173,165){\line(-1,-1){2}}
\put(173,178){\vector(0,-1){10}}
\put(168,182){$x_2$}
\put(176,165){\line(1,0){20}}

\put(46,130){\vector(1,0){70}}
\put(118,128){$\hat{s}_3$}
\put(130,130){\vector(1,0){40}}
\put(173,130){\circle{6}}
\put(173,130){\line(1,1){2}}
\put(173,130){\line(-1,1){2}}
\put(173,130){\line(1,-1){2}}
\put(173,130){\line(-1,-1){2}}
\put(173,143){\vector(0,-1){10}}
\put(168,147){$x_3$}
\put(176,130){\line(1,0){20}}

\put(120,104){$\vdots$}

\put(171,104){$\vdots$}

\put(46,0){\vector(1,0){70}}
\put(117,-2){$\hat{s}_N$}
\put(130,0){\vector(1,0){40}}
\put(173,0){\circle{6}}
\put(173,0){\line(1,1){2}}
\put(173,0){\line(-1,1){2}}
\put(173,0){\line(1,-1){2}}
\put(173,0){\line(-1,-1){2}}
\put(173,13){\vector(0,-1){10}}
\put(168,17){$x_N$}
\put(176,0){\line(1,0){20}}

% smash together

\put(196,0){\line(0,1){165}}
\put(196,100){\vector(1,0){20}}
\put(219,100){\circle{6}}
\put(219,100){\line(0,1){3}}
\put(219,100){\line(1,0){3}}
\put(219,100){\line(0,-1){3}}
\put(219,100){\line(-1,0){3}}
\put(219,113){\vector(0,-1){10}}
\put(216,117){$z$}
\put(222,100){\vector(1,0){40}}
\put(264,97){$f$}

\end{picture}

\caption{The fingerprint and forgery processes.
First, the content owner makes different copies of his host signal $s$ by adding fingerprints $\varphi_n$ which are unknown to the users.
Next, a subcollection $\mathcal{K}\subseteq\{1,\ldots,N\}$ of the users collude to create a forgery $f$ by picking a convex combination of their copies and adding noise $z$.
In this example, the forgery coalition $\mathcal{K}$ includes users $2$, $3$, and $N$.
}
\label{fig:attack_channel}
\end{figure}
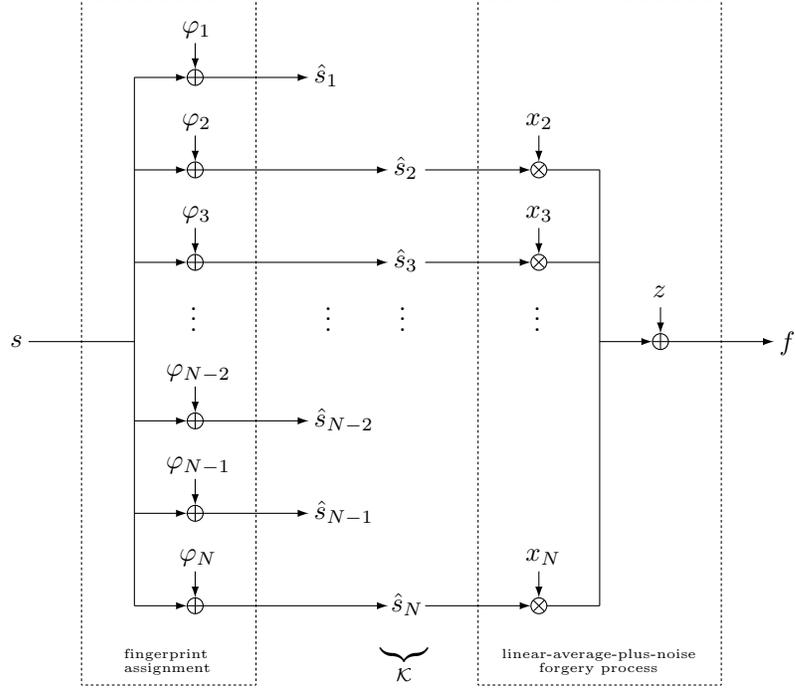

Certainly, the ultimate goal of the content owner is to detect every member of the forgery coalition.  
This can prove difficult in practice, though, particularly when some individuals contribute little to the forgery, with $x_k \ll \frac{1}{|\mathcal{K}|}$.  
However, in the real world, if at least one colluder is caught, then other members could be identified through the legal process.
As such, we consider \emph{focused} detection, where a test statistic is computed for each user, and we perform a binary hypothesis test to decide whether that particular user is guilty.  

Our detection procedure is as follows:
With the cooperation of the content owner, the host signal can be subtracted from a forgery to isolate the fingerprint combination: 
\begin{equation}
\label{eq:forgery} 
y:=f-s=\sum_{k\in\mathcal{K}}x_k\varphi_k+z.
\end{equation} 
To help the content owner discern who is guilty, we then use a normalized correlation function as a test statistic for each user $n$: 
\begin{equation*} 
\label{eq:test_stat} 
T_n(y):=\frac{\langle y,\varphi_n\rangle}{\|\varphi_n\|^2}.
\end{equation*} 
Having devised a test statistic, let $H_1(n)$ denote the guilty hypothesis ($n\in\mathcal{K}$) and $H_0(n)$ denote the innocent hypothesis ($n\not\in\mathcal{K}$).  
Then picking some correlation threshold $\tau$, we use the following detector:
\begin{equation} 
\label{eq:focused_detection}
D_\tau(n) := \left\{
   \begin{array}{ll}
     H_1(n), & T_n(y) \geq \tau,\\
     H_0(n), & T_n(y) < \tau.
   \end{array} \right.
\end{equation}  
To determine the effectiveness of our fingerprint design and focused detector, we will investigate the corresponding error probabilities, but first, we build our intuition for fingerprint design using a certain geometric figure of merit.

\subsection{A geometric figure of merit for fingerprint design} 

For each user $n$, consider the distance between forgeries deriving from two types of potential collusions: those of which $n$ is a member, and those of which $n$ is not.  
Intuitively, if every fingerprint combination involving $n$ is distant from every combination not involving $n$, then even with moderate noise, there should be little ambiguity as to whether the $n$th user was involved.
To make this precise, for each user $n$, we define the ``guilty'' and ``not guilty'' sets of noiseless fingerprint combinations:
\begin{align*}
\mathcal{G}_{K,n}&:=\bigg\{\frac{1}{|\mathcal{K}|}\sum_{k\in\mathcal{K}}\varphi_k:n\in\mathcal{K}\subseteq\{1,\ldots,N\},~|\mathcal{K}|\leq K\bigg\},\\
\neg\mathcal{G}_{K,n}&:=\bigg\{\frac{1}{|\mathcal{K}|}\sum_{k\in\mathcal{K}}\varphi_k:n\not\in\mathcal{K}\subseteq\{1,\ldots,N\},~|\mathcal{K}|\leq K\bigg\}.
\end{align*}
In words, $\mathcal{G}_{K,n}$ is the set of size-$K$ fingerprint combinations of equal weights which include $n$, while $\neg\mathcal{G}_{K,n}$ is the set of combinations which do not include $n$. 
Note that in our setup \eqref{eq:dm:attack1}, the weights $x_k$ were arbitrary values which sum to $1$.  
We will show in Theorem~\ref{thm:cq:linear_attack_against_corr_detector} that the best attack from the collusion's perspective uses equal weights so that no single colluder is particularly vulnerable.  
From this perspective, it makes sense to bound the distance between these two sets:
\begin{equation}
\label{eqn:dist_guilt_notguilt}
\dmk
:= \min\{\|y-y'\|_2:y\in\mathcal{G}_{K,n},~y'\in\neg\mathcal{G}_{K,n}\}.
\end{equation}  

Note that by taking $\Phi$ to be the $M\times N$ matrix whose columns are the fingerprints $\varphi_n$, the fingerprint combination \eqref{eq:forgery} can be rewritten as $y=\Phi x+z$, where the entries of $x$ are $x_k$ when $k\in\mathcal{K}$ and zero otherwise.
Thus, if the matrix of fingerprints $\Phi$ is $(K,\delta)$-RIP with $\delta<\sqrt{2}-1$, then we can recover the $K$-sparse vector $x$ using Theorem~\ref{thm.rip use}.
However, the error in the estimate $\tilde{x}$ of $x$ will be on the order of $10$ times the size of the noise $z$~\cite{Candes:08}.
Due to the potential legal ramifications of false accusations, this order of error is not tolerable.  
Note that the methods of compressed sensing recover the entire vector $x$, the support of which identifies the entire collusion.
By contrast, we will investigate RIP matrices for fingerprint design, but to minimize false accusations, we will use focused detection \eqref{eq:focused_detection} to identify colluders.

We now investigate how well RIP matrices perform with respect to our geometric figure of merit.
Without loss of generality, we assume the fingerprints are unit norm; since they have equal norm, the fingerprint combination can be scaled by $\frac{1}{\|\varphi_n\|}$ before the detection phase.   
With this in mind, we have the following a lower bound on the distance~\eqref{eqn:dist_guilt_notguilt} between the ``guilty'' and ``not guilty'' sets corresponding to any user $n$:

\begin{thm}
\label{thm.rip fingerprints}
Suppose fingerprints $\Phi=[\varphi_1\cdots\varphi_N]$ have restricted isometry constant $\delta_{2K}$. 
Then
\begin{equation} 
\label{lem:distbnd1}
\dmk\geq\sqrt{\frac{1-\delta_{2K}}{K(K-1)}}. 
\end{equation}
\end{thm}

\begin{proof}
Take $\mathcal{K},\mathcal{K}'\subseteq\{1,\ldots,N\}$ such that $|\mathcal{K}|,|\mathcal{K}'|\leq K$ and $n\in\mathcal{K}\setminus\mathcal{K}'$.
Then the left-hand inequality of the restricted isometry property gives
\begin{align}
\nonumber
\bigg\|\frac{1}{|\mathcal{K}|}\sum_{n\in\mathcal{K}}\varphi_n-\frac{1}{|\mathcal{K}'|}\sum_{n\in\mathcal{K}'}\varphi_n\bigg\|^2
\nonumber
&=\bigg\|\Big(\frac{1}{|\mathcal{K}|}-\frac{1}{|\mathcal{K}'|}\Big)\sum_{n\in\mathcal{K}\cap\mathcal{K}'}\varphi_n+\frac{1}{|\mathcal{K}|}\sum_{n\in\mathcal{K}\setminus\mathcal{K}'}\varphi_n-\frac{1}{|\mathcal{K}'|}\sum_{n\in\mathcal{K}'\setminus\mathcal{K}}\varphi_n\bigg\|^2\\
\nonumber
&\geq(1-\delta_{|\mathcal{K}\cup\mathcal{K}'|})\bigg(|\mathcal{K}\cap\mathcal{K}'|\Big(\frac{1}{|\mathcal{K}|}-\frac{1}{|\mathcal{K}'|}\Big)^2+\frac{|\mathcal{K}\setminus\mathcal{K}'|}{|\mathcal{K}|^2}+\frac{|\mathcal{K}'\setminus\mathcal{K}|}{|\mathcal{K}'|^2}\bigg)\\
\label{eq.distance bound}
&=\frac{1-\delta_{|\mathcal{K}\cup\mathcal{K}'|}}{|\mathcal{K}||\mathcal{K}'|}\bigg(|\mathcal{K}|+|\mathcal{K}'|-2|\mathcal{K}\cap\mathcal{K}'|\bigg).
\end{align}
For a fixed $|\mathcal{K}|$, we will find a lower bound for
\begin{equation}
\label{eq.min value} \frac{1}{|\mathcal{K}|}\bigg(|\mathcal{K}|+|\mathcal{K}'|-2|\mathcal{K}\cap\mathcal{K}'|\bigg)=1+\frac{|\mathcal{K}|-2|\mathcal{K}\cap\mathcal{K}'|}{|\mathcal{K}'|}.
\end{equation}
Since we can have $|\mathcal{K}\cap\mathcal{K}'|>\frac{|\mathcal{K}|}{2}$, 
we know $\frac{|\mathcal{K}|-2|\mathcal{K}\cap\mathcal{K}'|}{|\mathcal{K}'|}<0$ when \eqref{eq.min value} is minimized.  
That said, $|\mathcal{K}'|$ must be as small as possible, i.e., $|\mathcal{K}'|=|\mathcal{K}\cap\mathcal{K}'|$.
Thus, when \eqref{eq.min value} is minimized, we have
\begin{equation*}
\frac{1}{|\mathcal{K}|}\bigg(|\mathcal{K}|+|\mathcal{K}'|-2|\mathcal{K}\cap\mathcal{K}'|\bigg)=\frac{|\mathcal{K}|}{|\mathcal{K}\cap\mathcal{K}'|}-1,
\end{equation*}
i.e., $|\mathcal{K}\cap\mathcal{K}'|$ must be as large as possible.
Since $n\in\mathcal{K}\setminus\mathcal{K}'$, we have $|\mathcal{K}\cap\mathcal{K}'|\leq|\mathcal{K}|-1$.
Therefore,
\begin{equation}
\label{eq.min value 2}
\frac{1}{|\mathcal{K}|}\bigg(|\mathcal{K}|+|\mathcal{K}'|-2|\mathcal{K}\cap\mathcal{K}'|\bigg)\geq\frac{1}{|\mathcal{K}|-1}.
\end{equation}
Substituting \eqref{eq.min value 2} into \eqref{eq.distance bound} gives 
\begin{equation*}
\bigg\|\frac{1}{|\mathcal{K}|}\sum_{n\in\mathcal{K}}\varphi_n-\frac{1}{|\mathcal{K}'|}\sum_{n\in\mathcal{K}'}\varphi_n\bigg\|^2
\geq\frac{1-\delta_{|\mathcal{K}\cup\mathcal{K}'|}}{|\mathcal{K}|(|\mathcal{K}|-1)}
\geq\frac{1-\delta_{2K}}{K(K-1)}.
\end{equation*}
Since this bound holds for every $n$, $\mathcal{K}$ and $\mathcal{K}'$ with $n\in\mathcal{K}\setminus\mathcal{K}'$, we have~\eqref{lem:distbnd1}.
\end{proof}
  
Combining Theorem~\ref{thm.rip fingerprints} with the Gershgorin estimate $\delta_{2K}\leq(2K-1)\mu$ in terms of worst-case coherence $\mu$ yields the following:

\begin{cor}
\label{cor:etf_fingerprints}
Suppose fingerprints $\Phi=[\varphi_1\cdots\varphi_N]$ are unit-norm with worst-case coherence $\mu$.  
Then
\begin{equation}
\label{eq:dm:bound_dist_guilty_notguilty_set}
\dmk\geq\sqrt{\frac{1-(2K-1)\mu}{K(K-1)}}. 
\end{equation} 
\end{cor}

In words, Corrolary~\ref{cor:etf_fingerprints} says that less coherent fingerprints provide a greater distance between the ``guilty'' and ``not guilty'' sets.
It is therefore fitting to consider minimizers of worst-case coherence, namely equiangular tight frames.
One type of ETF has already been proposed for fingerprint design: the simplex~\cite{kiyavash2009regular}. 
The simplex is an ETF with $N=M+1$ and $\mu=\frac{1}{M}$.  
In fact, \cite{kiyavash2009regular} gives a derivation for the exact value of the distance \eqref{eqn:dist_guilt_notguilt} in this case:
\begin{equation} 
\label{eq.simplex distance} 
\dmk
=\sqrt{\frac{1}{K(K-1)}\frac{N}{N-1}}.
\end{equation}  
The bound \eqref{eq:dm:bound_dist_guilty_notguilty_set} is lower than \eqref{eq.simplex distance} by a factor of $\sqrt{1 - \frac{2K}{M+1}}$, and for practical cases in which $K \ll M$, the two are particularly close.   
Overall, ETF fingerprint design is a natural generalization of the provably optimal simplex design of~\cite{kiyavash2009regular}.

Having applied the Gershgorin analysis to illustrate how ETF fingerprints perform with respect to our geometric figure of merit, we have yet to establish any fingerprint-specific consequences of Steiner ETFs not being as RIP as random matrices.
Certainly, whether $K$ scales as $\sqrt{M}$ or $M$ is an important distinction in the compressed sensing community, but interestingly, in the context of fingerprints, this difference offers no advantage.
To be clear, Ergun et al.~\cite{ergun1999note} showed that for any fingerprinting system, there is a tradeoff between the probabilities of successful detection and false positives imposed by a linear-average-plus-noise attack from sufficiently large collusions. 
Specifically, a collusion of size $K=\mathrm{\Omega}\big(\sqrt{\frac{M}{\log M}}\big)$ is sufficient to overcome the fingerprints, as the detector will not be able to identify any attacker without incurring a false-alarm probability that is too large to be admissible in court.  
This constraint is more restrictive than the coherence-based reconstruction guarantees which require $K=\mathrm{O}(\sqrt{M})$, and so from this perspective, random RIP constructions are no better for fingerprint design than deterministic constructions.

\subsection{Error analysis} \label{sec:error_analysis}

We now investigate the errors associated with using ETF fingerprints and a focused correlation detector with linear-average-plus-noise attacks.
To do this, we assume that the noise $z$ included in the attack \eqref{eq:dm:attack1} has independent Gaussian entries of mean zero and variance $\sigma^2$.
One type of error we can expect is the false-positive error, in which an innocent user $n\notin\mathcal{K}$ is found guilty ($T_n(y) \geq \tau$).  
This could have significant ramifications in legal proceedings, so this error probability $\mathrm{Pr}\big[T_n(y)\geq\tau\big|H_0(n)\big]$ should be kept extremely low. 
To ensure this type of error is improbable, we consider the \emph{worst-case type I error probability}, which depends on the fingerprint design $\Phi$, the correlation threshold~$\tau$, and the weights $\{x_k\}_{k=1}^K$ used by the colluders in their linear average:
\begin{equation}
\label{eq:cq:worst_case_error_1}
\mathrm{P}_\mathrm{I}(\Phi,\tau,\{x_k\}_{k=1}^K)
:=\max_{\substack{\mathcal{K}\subseteq\{1,\ldots,N\}\\|\mathcal{K}|=K}}\max_{\substack{\mathcal{K}\rightarrow\{x_k\}\\\mathrm{bijective}}}\max_{n\not\in\mathcal{K}}\mathrm{Pr}\big[T_n(y)\geq\tau\big|H_0(n)\big].
\end{equation}
In words, the probability that an innocent user $n$ is found guilty is no larger than $\mathrm{P}_\mathrm{I}(\Phi,\tau,\{x_k\}_{k=1}^K)$, regardless of the coalition $\mathcal{K}$ or how the coalition members assign weights from $\{x_k\}_{k=1}^K$.
The other error type is the false-negative error, in which a guilty user $n\in\mathcal{K}$ is found innocent ($T_n(y)<\tau$).
In this case, since the goal of our detection is to catch at least one of the colluders, we define the \emph{worst-case type II error probability} as follows:
\begin{equation}
\label{eq:cq:worst_case_error_2}
\mathrm{P}_\mathrm{II}(\Phi,\tau,\{x_k\}_{k=1}^K)
:=\max_{\substack{\mathcal{K}\subseteq\{1,\ldots,N\}\\|\mathcal{K}|=K}}\max_{\substack{\mathcal{K}\rightarrow\{x_k\}\\\mathrm{bijective}}}\min_{n\in\mathcal{K}}\mathrm{Pr}\big[T_n(y)<\tau\big|H_1(n)\big].
\end{equation}
This way, regardless of who the colluders are or how they assign the weights, at least one of the colluders will have a false-negative probability less than $\mathrm{P}_\mathrm{II}(\Phi,\tau,\{x_k\}_{k=1}^K)$, meaning even in the worst-case scenario, we can correctly identify one of the colluders with probability $\geq1-\mathrm{P}_\mathrm{II}$.

\begin{thm}
\label{thm:cq:linear_attack_against_corr_detector}
Take fingerprints as the columns of an $M\times N$ matrix $\Phi=[\varphi_1\cdots\varphi_N]$, which, when normalized by the fingerprints' common norm $\gamma$, forms an equiangular tight frame.
If the noise $z$ included in the attack \eqref{eq:dm:attack1} has independent Gaussian entries of mean zero and variance $\sigma^2$, then the worst-case type I and type II error probabilities, \eqref{eq:cq:worst_case_error_1} and \eqref{eq:cq:worst_case_error_2}, satisfy
\begin{align*}
\mathrm{P}_\mathrm{I}(\Phi,\tau,\{x_k\}_{k=1}^K)
&\leq Q\bigg(\frac{\gamma}{\sigma}\big(\tau-\mu\big)\bigg),\\
\mathrm{P}_\mathrm{II}(\Phi,\tau,\{x_k\}_{k=1}^K)
&\leq Q\bigg(\frac{\gamma}{\sigma}\Big((1+\mu)\max\{x_k\}_{k=1}^K-\mu-\tau\Big)\bigg),
\end{align*}
where $Q(x) := \frac{1}{\sqrt{2\pi}}\int_{x}^{\infty} e^{-u^2/2} du$ and $\mu=\sqrt{\frac{N-M}{M(N-1)}}$.
\end{thm}

\begin{proof}
To bound $\mathrm{P}_\mathrm{I}(\Phi,\tau,\{x_k\}_{k=1}^K)$, assume a given user $n$ is innocent, i.e., $H_0(n)$.
Then the test statistic for our detector \eqref{eq:focused_detection} is given by
\begin{equation*}
T_n(y)
=\frac{1}{\gamma^2}\bigg\langle \sum_{k\in\mathcal{K}}x_k\varphi_k+z,\varphi_n\bigg\rangle
=\sum_{k\in\mathcal{K}}x_k\bigg\langle \frac{\varphi_k}{\|\varphi_k\|},\frac{\varphi_n}{\|\varphi_n\|}\bigg\rangle+\frac{1}{\gamma}\bigg\langle z,\frac{\varphi_n}{\|\varphi_n\|}\bigg\rangle.
\end{equation*}
By the symmetry of $z$'s Gaussian distribution, we know the projection $\langle z,\frac{\varphi_n}{\|\varphi_n\|}\rangle$ also has Gaussian distribution with mean zero and variance $\sigma^2$, meaning our test statistic $T_n(y)$ has Gaussian distribution with mean $\sum_{k\in\mathcal{K}}x_k\langle \frac{\varphi_k}{\|\varphi_k\|},\frac{\varphi_n}{\|\varphi_n\|}\rangle$ and variance $\frac{\sigma^2}{\gamma^2}$.
Furthermore, since the normalized fingerprints form an ETF with worst-case coherence $\mu$, we can use the triangle inequality to bound the mean of $T_n(y)$:
\begin{equation*}
\sum_{k\in\mathcal{K}}x_k\bigg\langle \frac{\varphi_k}{\|\varphi_k\|},\frac{\varphi_n}{\|\varphi_n\|}\bigg\rangle
\leq\bigg|\sum_{k\in\mathcal{K}}x_k\bigg\langle \frac{\varphi_k}{\|\varphi_k\|},\frac{\varphi_n}{\|\varphi_n\|}\bigg\rangle\bigg|
\leq\sum_{k\in\mathcal{K}}x_k\bigg|\bigg\langle \frac{\varphi_k}{\|\varphi_k\|},\frac{\varphi_n}{\|\varphi_n\|}\bigg\rangle\bigg|
=\mu.
\end{equation*}
We use this to bound the false-positive probability for user $n$:
\begin{equation*}
\mathrm{Pr}\big[T_n(y)\geq\tau\big|H_0(n)\big]
=Q\bigg(\frac{\gamma}{\sigma}\Big(\tau-\mathbb{E}\big[T_n(y)|H_0(n)\big]\Big)\bigg)
\leq Q\bigg(\frac{\gamma}{\sigma}\big(\tau-\mu\big)\bigg).
\end{equation*}
Since this bound holds for all coalitions, weight assignments and innocent users, this bound must also hold for $\mathrm{P}_\mathrm{I}(\Phi,\tau,\{x_k\}_{k=1}^K)$.

Next, to bound $\mathrm{P}_\mathrm{II}(\Phi,\tau,\{x_k\}_{k=1}^K)$, assume a given user $n$ is guilty, i.e., $H_1(n)$.
In this case, the test statistic for our detector \eqref{eq:focused_detection} is given by
\begin{equation*}
T_n(y)
=\frac{1}{\gamma^2}\bigg\langle \sum_{k\in\mathcal{K}}x_k\varphi_k+z,\varphi_n\bigg\rangle
=x_n+\sum_{\substack{k\in\mathcal{K}\\k\neq n}}x_k\bigg\langle \frac{\varphi_k}{\|\varphi_k\|},\frac{\varphi_n}{\|\varphi_n\|}\bigg\rangle+\frac{1}{\gamma}\bigg\langle z,\frac{\varphi_n}{\|\varphi_n\|}\bigg\rangle.
\end{equation*}
As before, $T_n(y)$ has Gaussian distribution with variance $\frac{\sigma^2}{\gamma^2}$, but this time, the mean is
\begin{equation*}
x_n+\sum_{\substack{k\in\mathcal{K}\\k\neq n}}x_k\bigg\langle \frac{\varphi_k}{\|\varphi_k\|},\frac{\varphi_n}{\|\varphi_n\|}\bigg\rangle
\geq x_n-\bigg|\sum_{\substack{k\in\mathcal{K}\\k\neq n}}x_k\bigg\langle \frac{\varphi_k}{\|\varphi_k\|},\frac{\varphi_n}{\|\varphi_n\|}\bigg\rangle\bigg|
\geq x_n-\mu\sum_{\substack{k\in\mathcal{K}\\k\neq n}}x_k
=(1+\mu)x_n-\mu.
\end{equation*}
As such, the false-negative probability for user $n$ is
\begin{equation*}
\mathrm{Pr}\big[T_n(y)<\tau\big|H_1(n)\big]
=Q\bigg(-\frac{\gamma}{\sigma}\Big(\tau-\mathbb{E}\big[T_n(y)|H_1(n)\big]\Big)\bigg)
\leq Q\bigg(\frac{\gamma}{\sigma}\Big((1+\mu)x_n-\mu-\tau\Big)\bigg).
\end{equation*}
Applying the definition of $\mathrm{P}_\mathrm{II}(\Phi,\tau,\{x_k\}_{k=1}^K)$ therefore gives
\begin{align*}
\mathrm{P}_\mathrm{II}(\Phi,\tau,\{x_k\}_{k=1}^K)
&=\max_{\substack{\mathcal{K}\subseteq\{1,\ldots,N\}\\|\mathcal{K}|=K}}\max_{\substack{\mathcal{K}\rightarrow\{x_k\}\\\mathrm{bijective}}}\min_{n\in\mathcal{K}}\mathrm{Pr}\big[T_n(y)<\tau\big|H_1(n)\big]\\
&\leq\max_{\substack{\mathcal{K}\subseteq\{1,\ldots,N\}\\|\mathcal{K}|=K}}\max_{\substack{\mathcal{K}\rightarrow\{x_k\}\\\mathrm{bijective}}}\min_{n\in\mathcal{K}}Q\bigg(\frac{\gamma}{\sigma}\Big((1+\mu)x_n-\mu-\tau\Big)\bigg)\\
&=Q\bigg(\frac{\gamma}{\sigma}\Big((1+\mu)\max\{x_k\}_{k=1}^K-\mu-\tau\Big)\bigg).
\qedhere
\end{align*}
\end{proof}

From Theorem~\ref{thm:cq:linear_attack_against_corr_detector}, we can glean a few interesting insights about ETF fingerprints.
First, the upper bound on $\mathrm{P}_\mathrm{I}(\Phi,\tau,\{x_k\}_{k=1}^K)$ is independent of $\{x_k\}_{k=1}^K$, indicating that the coalition cannot pick weights in a way that frames an innocent user.
Additionally, the upper bound on $\mathrm{P}_\mathrm{II}(\Phi,\tau,\{x_k\}_{k=1}^K)$ is maximized when the weights $x_k$ are equal, corresponding to our use of equal weights in the geometric figure of merit.
This confirms our intuition that the coalition has the best chance of not being caught if no member is particularly vulnerable.

\chapter{Full spark frames}

In the previous chapter, we reviewed how to use the Gershgorin circle theorem to demonstrate the restricted isometry property (RIP), and how identifying small spark disproves RIP.
We then showed that Steiner equiangular tight frames (ETFs) are optimal in the Gershgorin sense, but have particularly small spark.
Among other things, this illustrates that the ``square-root bottleneck'' with deterministic RIP matrices is not merely an artifact of the Gershgorin analysis.
That said, as an intermediate goal to constructing RIP matrices, we seek deterministic matrices with large spark, understanding that RIP matrices necessarily have this property.
To this end, one is naturally led to consider \emph{full spark} matrices, that is, $M\times N$ matrices $\Phi$ with the largest spark possible: $\mathrm{Spark}(\Phi)=M+1$.
Equivalently, $M\times N$ full spark matrices have the property that every $M\times M$ submatrix is invertible; as such, a full spark matrix is necessarily full rank, and therefore a frame.

Interestingly, in sparse signal processing, the specific application of full spark frames has already been studied for some time.
In 1997, Gorodnitsky and Rao~\cite{GorodnitskyR:97} first considered full spark frames, referring to them as matrices with the \emph{unique representation property}.
Since~\cite{GorodnitskyR:97}, the unique representation property has been explicitly used to find a variety of performance guarantees for sparse signal processing~\cite{BourguignonCI:07,MohimaniBJ:09,WipfR:04}.
Tang and Nehorai~\cite{TangN:10} also obtain performance guarantees using full spark frames, but they refer to them as \emph{non-degenerate measurement matrices}.

For another application of full spark frames, we consider the problem of reconstructing a signal from distorted frame coefficients.
Specifically, we observe a scenario in which frame coefficients $\{(\Phi^*x)[n]\}_{n=1}^N$ are transmitted over a noisy or lossy channel before reconstructing the signal:
\begin{equation}
\label{eq.reconstruction}
y=\mathcal{D}(\Phi^*x),\qquad \tilde{x}=(\Phi\Phi^*)^{-1}\Phi y,                                                                                                                                                                                                                  \end{equation}
where $\mathcal{D}(\cdot)$ represents the channel's random and not-necessarily-linear deformation process.
Using an additive white Gaussian noise model, Goyal~\cite{Goyal:phd98} established that, of all unit norm frames, unit norm tight frames minimize mean squared error in reconstruction.
For the case of a lossy channel, Holmes and Paulsen~\cite{HolmesP:04} established that, of all tight frames, unit norm tight frames minimize worst-case error in reconstruction after one erasure, and that equiangular tight frames minimize this error after two erasures.
We note that the reconstruction process in \eqref{eq.reconstruction}, namely the application of $(\Phi\Phi^*)^{-1}\Phi$, is inherently blind to the effect of the deformation process of the channel.
This contrasts with P\"{u}schel and Kova\v{c}evi\'{c}'s more recent work~\cite{PuschelK:dcc05}, which describes an adaptive process for reconstruction after multitudes of erasures.
In this context, they reconstruct the signal after first identifying which frame coefficients were not erased; with this information, the signal can be estimated provided the corresponding frame elements span.
In this sense, full spark frames are \emph{maximally robust to erasures}, as coined in~\cite{PuschelK:dcc05}.
In particular, an $M\times N$ full spark frame is robust to $N-M$ erasures since any $M$ of the frame coefficients will uniquely determine the original signal.

Yet another application of full spark frames is phaseless reconstruction, which can be viewed in terms of a channel, as in \eqref{eq.reconstruction}; in this case, $\mathcal{D}(\cdot)$ is the entrywise absolute value function.
Phaseless reconstruction has a number of real-world applications including speech processing~\cite{BalanCE:acha06}, X-ray crystallography~\cite{CandesSV:arxiv11}, and quantum state estimation~\cite{RenesBSC:04}.
As such, there has been a lot of work to reconstruct an $M$-dimensional vector (up to an overall phase factor) from the magnitudes of its frame coefficients, most of which involves frames in operator space, which inherently require $N=\Omega(M^2)$ measurements~\cite{BalanBCE:jfaa09,RenesBSC:04}.
However, Balan et al.~\cite{BalanCE:acha06} show that if an $M\times N$ real frame $\Phi$ is full spark with $N\geq 2M-1$, then $\mathcal{D}\circ \Phi^*$ is injective, meaning an inversion process is possible with only $N=\mathrm{O}(M)$ measurements.
This result prompted an ongoing search for efficient phaseless reconstruction processes~\cite{BalanBCE:spie07,CandesSV:arxiv11}, but no reconstruction process can succeed without a good family of frames, such as full spark frames.

Despite the fact that full spark frames have a multitude of applications, to date, there has not been much progress in constructing deterministic full spark frames, let alone full spark frames with additional desirable properties.
A noteworthy exception is P\"{u}schel and Kova\v{c}evi\'{c}'s work~\cite{PuschelK:dcc05}, in which real full spark tight frames are constructed using polynomial transforms.
In the present chapter, we start by investigating Vandermonde frames, harmonic frames, and modifications thereof~\cite{AlexeevCM:arxiv11}.
While the use of certain Vandermonde and harmonic frames as full spark frames is not new~\cite{BourguignonCI:07,CandesRT:06,Fuchs:05}, the fruits of our investigation are new:
For instance, we demonstrate that certain classes of ETFs are full spark, and we characterize the $M\times N$ full spark harmonic frames for which $N$ is a prime power.
Later, we prove that verifying whether a matrix is full spark is hard for $\NP$ under randomized polynomial-time reductions~\cite{AlexeevCM:arxiv11}.
In other words, assuming $\NP\not\subseteq\BPP$ (a computational complexity assumption slightly stronger than $\P\neq\NP$ and nearly as widely believed), then there is no method by which one can efficiently test whether matrices are full spark.
As such, the deterministic constructions we provide are significant in that they guarantee a property which is otherwise difficult to check.
We conclude the chapter by introducing a new technique for efficient phaseless recovery, which explicitly makes use of deterministic full spark frames to design $N=\mathrm{O}(M)$ measurements.

\section{Deterministic constructions of full spark frames}

A square matrix is invertible if and only if its determinant is nonzero, and in our quest for deterministic constructions of full spark frames, this characterization will reign supreme.
One class of matrices has a particularly simple determinant formula: Vandermonde matrices.
Specifically, Vandermonde matrices have the following form:
\begin{equation}
\label{eq.vandermonde}
V
=\begin{bmatrix}
 1 & 1 & \cdots & 1\\
 \alpha_1 & \alpha_2 & \cdots & \alpha_N\\
\vdots & \vdots & \cdots & \vdots\\
 \alpha_1^{M-1} & \alpha_2^{M-1} & \cdots & \alpha_N^{M-1}
 \end{bmatrix},
\end{equation}
and square Vandermonde matrices, i.e., with $N=M$, have the following determinant:
\begin{equation}
\label{eq.vandermonde determinant}
\mathrm{det}(V)=\prod_{1\leq i<j\leq M}(\alpha_j-\alpha_i).
\end{equation}
Consider \eqref{eq.vandermonde} in the case where $N\geq M$.
Since every $M\times M$ submatrix of $V$ is also Vandermonde, we can modify the indices in \eqref{eq.vandermonde determinant} to calculate the determinant of the submatrices.  
These determinants are nonzero precisely when the bases $\{\alpha_{n}\}_{n=1}^N$ are distinct, yielding the following result:

\begin{lem}
\label{lem.vandermonde}
A Vandermonde matrix is full spark if and only if its bases are distinct.
\end{lem}

To be clear, this result is not new.
In fact, the full spark of Vandermonde matrices was first exploited by Fuchs~\cite{Fuchs:05} for sparse signal processing.
Later, Bourguignon et al.~\cite{BourguignonCI:07} specifically used the full spark of Vandermonde matrices whose bases are sampled from the complex unit circle.
Interestingly, when viewed in terms of frame theory, Vandermonde matrices naturally point to the discrete Fourier transform:

\begin{thm}
\label{thm.vandermonde}
The only $M\times N$ Vandermonde matrices that are equal norm and tight have bases in the complex unit circle.  Among these, the frames with the smallest worst-case coherence have bases that are equally spaced in the complex unit circle, provided $N\geq 2M$.
\end{thm}

\begin{proof}
Suppose a Vandermonde matrix is equal norm and tight.  
Note that a zero base will produce the zeroth identity basis element $\delta_0$.
Letting $\mathcal{P}$ denote the indices of the nonzero bases, the fact that the matrix is full rank implies $|\mathcal{P}|\geq M-1$.
Also, equal norm gives that the frame element length
\[ \|\varphi_n\|^2=\sum_{m=0}^{M-1}|\varphi_n[m]|^2=\sum_{m=0}^{M-1}|\alpha_n^m|^2=\sum_{m=0}^{M-1}|\alpha_n|^{2m} \]
is constant over $n\in\mathcal{P}$.  
Since $\sum_{m=0}^{M-1}x^{2m}$ is strictly increasing over $0<x<\infty$, there exists $c>0$ such that $|\alpha_n|^2=c$ for all $n\in\mathcal{P}$.   
Next, tightness gives that the rows have equal norm, implying that the first two rows have equal norm, i.e., $|\mathcal{P}|c=|\mathcal{P}|c^2$.  
Thus $c=1$, and so the nonzero bases are in the complex unit circle.  
Furthermore, since the zeroth and first rows have equal norm by tightness, we have $|\mathcal{P}|=N$, and so every base is in the complex unit circle.

Now consider the inner product between Vandermonde frame elements whose bases $\{e^{2\pi i x_n}\}_{n=1}^N$ come from the complex unit circle:
\[ \langle \varphi_n,\varphi_{n'}\rangle=\sum_{m=0}^{M-1}(e^{2\pi i x_n})^m\overline{(e^{2\pi i x_{n'}})^m}=\sum_{m=0}^{M-1}e^{2\pi i (x_n-x_{n'})m}. \]
We will show that the worst-case coherence comes from the two closest bases.  
Consider the following function:
\begin{equation}
\label{eq.g function}
g(x)
:=\bigg|\sum_{m=0}^{M-1}e^{2\pi ixm}\bigg|^2.
\end{equation}
Figure~\ref{figure} gives a plot of this function in the case where $M=5$.
We will prove two things about this function:
\begin{itemize}
\item[(i)] $\tfrac{d}{dx}g(x)<0$ for every $x\in(0,\tfrac{1}{2M})$,
\item[(ii)] $g(x)\leq g(\tfrac{1}{2M})$ for every $x\in(\tfrac{1}{2M},1-\tfrac{1}{2M})$.
\end{itemize}

\begin{figure}[t]
\centering
\includegraphics[width=0.5\textwidth]{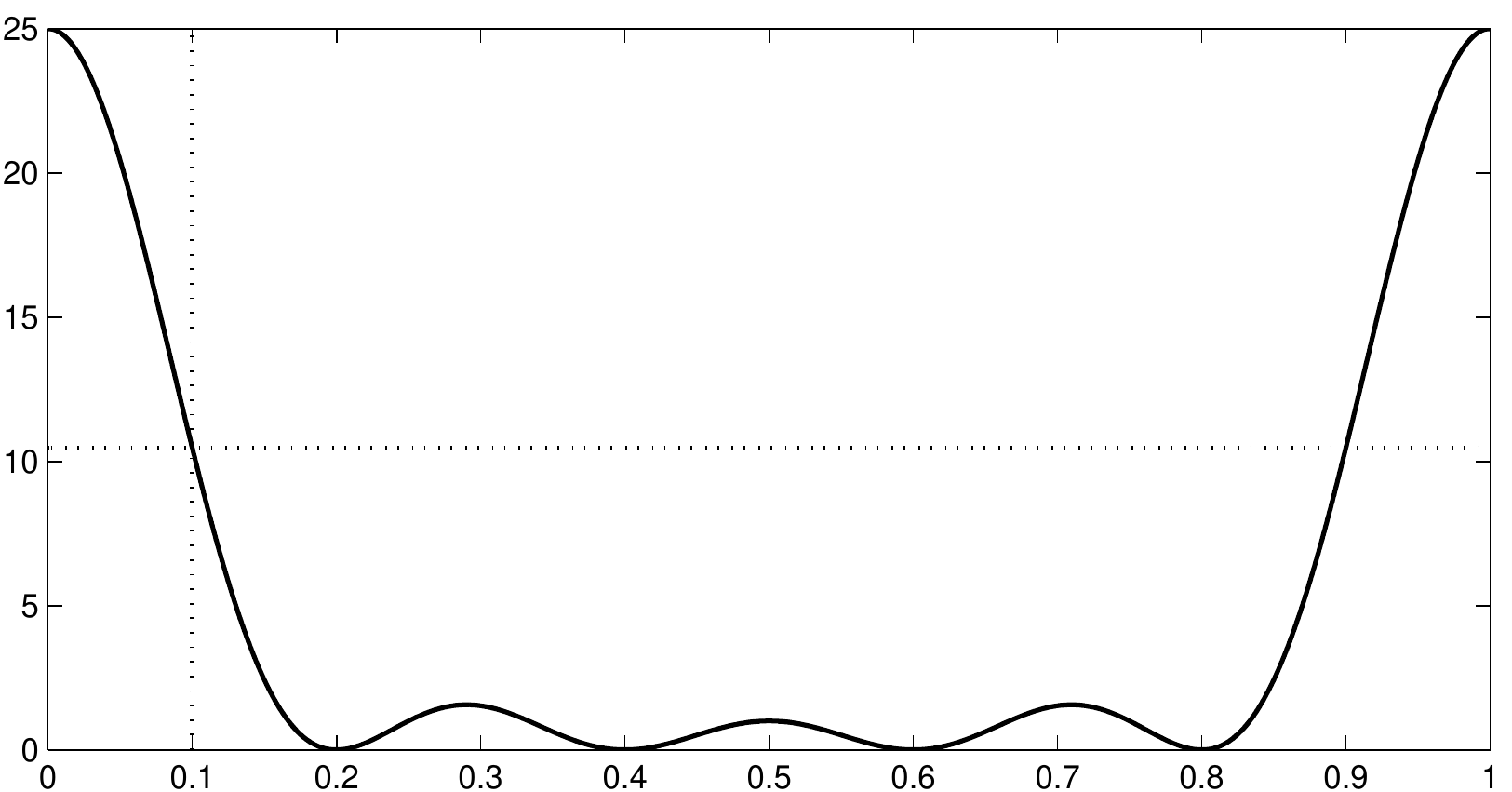}
\caption{Plot of $g$ defined by \eqref{eq.g function} in the case where $M=5$.  Observe (i) that $g$ is strictly decreasing on the interval $(0,\frac{1}{10})$, and (ii) that $g(x)\leq g(\frac{1}{10})$ for every $x\in(\frac{1}{10},\frac{9}{10})$.  As established in the proof of Theorem~\ref{thm.vandermonde}, $g$ behaves in this manner for general values of $M$.
\label{figure}}
\end{figure}

First, we claim that (i) and (ii) are sufficient to prove our result.
To establish this, we first show that the two closest bases $e^{2\pi ix_{n'}}$ and $e^{2\pi ix_{n''}}$ satisfy $|x_{n'}-x_{n''}|\leq\frac{1}{2M}$.
Without loss of generality, the $n$'s are ordered in such a way that $\{x_n\}_{n=0}^{N-1}\subseteq[0,1)$ are nondecreasing.
Define
\begin{equation*}
d(x_n,x_{n+1})
:=\left\{\begin{array}{ll}
x_{n+1}-x_n,&n=0,\ldots,N-2\\ 
x_0-(x_{N-1}-1),&n=N-1,
\end{array}
\right.
\end{equation*}
and let $n'$ be the $n$ which minimizes $d(x_n,x_{n+1})$.
Since the minimum is less than the average, we have
\begin{equation}
\label{eq.average inequality}
d(x_{n'},x_{n'+1})\leq\frac{1}{N}\bigg((x_0-(x_{N-1}-1))+\sum_{n=0}^{N-1}(x_{n+1}-x_n)\bigg)=\frac{1}{N}\leq\frac{1}{2M},
\end{equation}
provided $N\geq2M$.
Note that if we view $\{x_n\}_{n\in\mathbb{Z}_N}$ as members of $\mathbb{R}/\mathbb{Z}$, then $d(x_n,x_{n+1})=x_{n+1}-x_n$.
Since $g(x)$ is even, then (i) implies that $|\langle \varphi_{n'+1},\varphi_{n'}\rangle|^2=g(x_{n'+1}-x_{n'})$ is larger than any other $g(x_p-x_{p'})=|\langle \varphi_p,\varphi_{p'}\rangle|^2$ in which $x_p-x_{p'}\in[0,\tfrac{1}{2M}]\cup[1-\tfrac{1}{2M},1)$.
Next, \eqref{eq.average inequality} and (ii) together imply that $|\langle \varphi_{n'+1},\varphi_{n'}\rangle|^2=g(x_{n'+1}-x_{n'})\geq g(\tfrac{1}{2M})$ is larger than any other $g(x_p-x_{p'})=|\langle \varphi_p,\varphi_{p'}\rangle|^2$ in which $x_p-x_{p'}\in(\tfrac{1}{2M},1-\tfrac{1}{2M})$, provided $N\geq2M$.
Combined, (i) and (ii) give that $|\langle \varphi_{n'+1},\varphi_{n'}\rangle|$ achieves the worst-case coherence of $\{\varphi_n\}_{n\in\mathbb{Z}_N}$.
Additionally, (i) gives that the worst-case coherence $|\langle \varphi_{n'+1},\varphi_{n'}\rangle|$ is minimized when $x_{n'+1}-x_{n'}$ is maximized, i.e., when the $x_n$'s are equally spaced in the unit interval.

To prove (i), note that the geometric sum formula gives
\begin{equation}
\label{eq.g of x}
g(x)
=\bigg|\sum_{m=0}^{M-1}e^{2\pi ixm}\bigg|^2
=\bigg|\frac{e^{2M\pi ix}-1}{e^{2\pi ix}-1}\bigg|^2
=\frac{2-2\cos(2M\pi x)}{2-2\cos(2\pi x)}
=\bigg(\frac{\sin(M\pi x)}{\sin(\pi x)}\bigg)^2,
\end{equation}
where the final expression uses the identity $1-\cos(2z)=2\sin^2 z$.
To show that $g$ is decreasing over $(0,\frac{1}{2M})$, note that the base of \eqref{eq.g of x} is positive on this interval, and performing the quotient rule to calculate its derivative will produce a fraction whose denominator is nonnegative and whose numerator is given by
\begin{equation}
\label{eq.g of x 2}
M\pi\sin(\pi x)\cos(M\pi x)-\pi\sin(M\pi x)\cos(\pi x).
\end{equation}
This factor is zero at $x=0$ and has derivative:
\begin{equation*}
-(M^2-1)\pi^2\sin(\pi x)\sin(M\pi x),
\end{equation*}
which is strictly negative for all $x\in(0,\tfrac{1}{2M})$.
Hence, \eqref{eq.g of x 2} is strictly negative whenever $x\in(0,\tfrac{1}{2M})$, and so $g'(x)<0$ for every $x\in(0,\tfrac{1}{2M})$.

For (ii), note that for every $x\in(\frac{1}{2M},1-\frac{1}{2M})$, we can individually bound the numerator and denominator of what the geometric sum formula gives:
\begin{equation*}
g(x)
=\bigg|\sum_{m=0}^{M-1}e^{2\pi ixm}\bigg|^2
=\frac{|e^{2M\pi ix}-1|^2}{|e^{2\pi ix}-1|^2}
\leq\frac{|e^{\pi i}-1|^2}{|e^{\pi i/M}-1|^2}
=\bigg|\sum_{m=0}^{M-1}e^{\pi im/M}\bigg|^2
=g(\tfrac{1}{2M}).\qedhere
\end{equation*}
\end{proof}

Consider the $N\times N$ discrete Fourier transform (DFT) matrix, scaled to have entries of unit modulus:
\begin{equation*}
\begin{bmatrix}
   1 & 1 & 1 & \cdots & 1\\
   1 & \omega & \omega^2 & \cdots & \omega^{N-1}\\
   1 & \omega^2 & \omega^4 & \cdots & \omega^{2(N-1)}\\
\vdots & \vdots & \vdots & \cdots & \vdots\\
   1 & \omega^{N-1} & \omega^{2(N-1)} & \cdots & \omega^{(N-1)(N-1)}
  \end{bmatrix},
\end{equation*}
where $\omega=e^{-2\pi i/N}$.
The first $M$ rows of the DFT form a Vandermonde matrix of distinct bases $\{\omega^n\}_{n=0}^{N-1}$; as such, this matrix is full spark by Lemma~\ref{lem.vandermonde}.
In fact, the previous result says that this is in some sense an optimal Vandermonde frame, but this might not be the best way to pick rows from a DFT.
Indeed, several choices of DFT rows could produce full spark frames, some with smaller coherence or other desirable properties, and so the remainder of this section focuses on full spark DFT submatrices.
First, we note that not every DFT submatrix is full spark.
For example, consider the $4\times 4$ DFT:
\begin{equation*}
\begin{bmatrix}
1&1&1&1\\
1&-i&-1&i\\
1&-1&1&-1\\
1&i&-1&-i
\end{bmatrix}.
\end{equation*}
Certainly, the zeroth and second rows of this matrix are not full spark, since the zeroth and second columns of this submatrix form the all-ones matrix, which is not invertible.
So what can be said about the set of permissible row choices?
The following result gives some necessary conditions on this set:

\begin{thm}
\label{thm.closure}
Take an $N\times N$ discrete Fourier transform matrix, and select the rows indexed by $\mathcal{M}\subseteq\mathbb{Z}_N$ to build the matrix $\Phi$.
If $\Phi$ is full spark, then so is the matrix built from rows indexed by
\begin{itemize}
\item[(i)] any translation of $\mathcal{M}$,
\item[(ii)] any $A\mathcal{M}$ with $A$ relatively prime to $N$,
\item[(iii)] the complement of $\mathcal{M}$ in $\mathbb{Z}_N$.
\end{itemize}
\end{thm}

\begin{proof}
For (i), we first define $D$ to be the $N\times N$ diagonal matrix whose diagonal entries are $\{\omega^n\}_{n=0}^{N-1}$.
Note that, since $\omega^{(m+1)n}=\omega^n\omega^{mn}$, translating the row indices $\mathcal{M}$ by $1$ corresponds to multiplying $\Phi$ on the right by $D$.
For some set $\mathcal{K}\subseteq\mathbb{Z}_N$ of size $M:=|\mathcal{M}|$, let $\Phi_\mathcal{K}$ denote the $M\times M$ submatrix of $\Phi$ whose columns are indexed by $\mathcal{K}$, and let $D_\mathcal{K}$ denote the $M\times M$ diagonal submatrix of $D$ whose diagonal entries are indexed by $\mathcal{K}$.
Then since $D_\mathcal{K}$ is unitary, we have
\begin{equation*}
|\mathrm{det}((\Phi D)_\mathcal{K})|
=|\mathrm{det}(\Phi_\mathcal{K}D_\mathcal{K})|
=|\mathrm{det}(\Phi_\mathcal{K})||\mathrm{det}(D_\mathcal{K})|
=|\mathrm{det}(\Phi_\mathcal{K})|.
\end{equation*}
Thus, if $\Phi$ is full spark, $|\mathrm{det}((\Phi D)_\mathcal{K})|=|\mathrm{det}(\Phi_\mathcal{K})|>0$, and so $\Phi D$ is also full spark.
Using this fact inductively proves (i) for all translations of $\mathcal{M}$.

For (ii), let $\Psi$ denote the submatrix of rows indexed by $A\mathcal{M}$.
Then for any $\mathcal{K}\subseteq\mathbb{Z}_N$ of size $M$,
\begin{equation*}
\mathrm{det}(\Psi_\mathcal{K})
=\mathrm{det}(\omega^{(Am)k})_{m\in\mathcal{M},k\in\mathcal{K}}
=\mathrm{det}(\omega^{m(Ak)})_{m\in\mathcal{M},k\in\mathcal{K}}
=\mathrm{det}(\Phi_{A\mathcal{K}}).
\end{equation*}
Since $A$ is relatively prime to $N$, multiplication by $A$ permutes the elements of $\mathbb{Z}_N$, and so $A\mathcal{K}$ has exactly $M$ distinct elements.
Thus, if $\Phi$ is full spark, then $\mathrm{det}(\Psi_\mathcal{K})=\mathrm{det}(\Phi_{A\mathcal{K}})\neq0$, and so $\Psi$ is also full spark.

For (iii), we let $\Psi$ be the $(N-M)\times N$ submatrix of rows indexed by $\mathcal{M}^\mathrm{c}$, so that
\begin{equation}
\label{eq.complements}
NI_N
=\begin{bmatrix}\Phi^* & \Psi^*\end{bmatrix}\begin{bmatrix}\Phi\\\Psi\end{bmatrix}
=\Phi^*\Phi+\Psi^*\Psi.
\end{equation}
We will use contraposition to show that $\Phi$ being full spark implies that $\Psi$ is also full spark.
To this end, suppose $\Psi$ is not full spark.
Then $\Psi$ has a collection of $N-M$ linearly dependent columns $\{\psi_i\}_{i\in \mathcal{K}}$, and so there exists a nontrivial sequence $\{\alpha_i\}_{i\in \mathcal{K}}$ such that 
\begin{equation*}
\sum_{i\in \mathcal{K}}\alpha_i\psi_i
=0.
\end{equation*}
Considering $\psi_i=\Psi\delta_i$, where $\delta_i$ is the $i$th identity basis element, we can use \eqref{eq.complements} to express this linear dependence in terms of $\Phi$:
\begin{equation*}
0
=\Psi^*0
=\Psi^*\sum_{i\in \mathcal{K}}\alpha_i\psi_i
=\sum_{i\in \mathcal{K}}\alpha_i\Psi^*\Psi\delta_i
=\sum_{i\in \mathcal{K}}\alpha_i(NI_N-\Phi^*\Phi)\delta_i.
\end{equation*}
Rearranging then gives
\begin{equation}
\label{eq.x defn}
x:=N\sum_{i\in \mathcal{K}}\alpha_i\delta_i=\sum_{i\in \mathcal{K}}\alpha_i\Phi^*\Phi\delta_i.
\end{equation}
Here, we note that $x$ is nonzero since $\{\alpha_i\}_{i\in \mathcal{K}}$ is nontrivial, and that $x\in\mathrm{Range}(\Phi^*\Phi)$.
Furthermore, whenever $j\not\in \mathcal{K}$, we have from \eqref{eq.x defn} that
\begin{equation*}
\langle x,\Phi^*\Phi\delta_j\rangle
=\langle \Phi^*\Phi x,\delta_j\rangle
=N\bigg\langle \Phi^*\Phi\sum_{i\in \mathcal{K}}\alpha_i\delta_i,\delta_j\bigg\rangle
=N^2\bigg\langle\sum_{i\in \mathcal{K}}\alpha_i\delta_i,\delta_j\bigg\rangle
=0,
\end{equation*}
and so $x\perp\mathrm{Span}\{\Phi^*\Phi\delta_j\}_{j\in \mathcal{K}^\mathrm{c}}$.
Thus, the containment $\mathrm{Span}\{\Phi^*\Phi\delta_j\}_{j\in \mathcal{K}^\mathrm{c}}\subseteq\mathrm{Range}(\Phi^*\Phi)$ is proper, and so
\begin{equation*}
M
=\mathrm{Rank}(\Phi)
=\mathrm{Rank}(\Phi^*\Phi)
>\mathrm{Rank}(\Phi^*\Phi_{\mathcal{K}^\mathrm{c}})
=\mathrm{Rank}(\Phi_{\mathcal{K}^\mathrm{c}}).
\end{equation*}
Since the $M\times M$ submatrix $\Phi_{\mathcal{K}^\mathrm{c}}$ is rank-deficient, it is not invertible, and therefore $\Phi$ is not full spark.
\end{proof}

We note that our proof of (iii) above uses techniques from Cahill et al.~\cite{CahillCH:sampta11}, and can be easily generalized to prove that the Naimark complement of a full spark tight frame is also full spark.
Theorem~\ref{thm.closure} tells us quite a bit about the set of permissible choices for DFT rows.
For example, not only can we pick the first $M$ rows of the DFT to produce a full spark Vandermonde frame, but we can also pick any consecutive $M$ rows, by Theorem~\ref{thm.closure}(i).
We would like to completely characterize the choices that produce full spark harmonic frames.
The following classical result does this in the case where $N$ is prime:

\begin{thm}[Chebotar\"{e}v, see~\cite{StevenhagenL:mi96}]
Let $N$ be prime.
Then every square submatrix of the $N\times N$ discrete Fourier transform matrix is invertible.
\end{thm}

As an immediate consequence of Chebotar\"{e}v's theorem, every choice of rows from the DFT produces a full spark harmonic frame, provided $N$ is prime.
This application of Chebotar\"{e}v's theorem was first used by Cand\`{e}s et al.~\cite{CandesRT:06} for sparse signal processing.
Note that each of these frames are equal-norm and tight by construction.
Harmonic frames can also be designed to have minimal coherence; Xia et al.~\cite{XiaZG:05} produces harmonic equiangular tight frames by selecting row indices which form a difference set in $\mathbb{Z}_N$.
Interestingly, most known families of difference sets in $\mathbb{Z}_N$ require $N$ to be prime~\cite{JungnickelPS:hcd07}, and so the corresponding harmonic equiangular tight frames are guaranteed to be full spark by Chebotar\"{e}v's theorem.
In the following, we use Chebotar\"{e}v's theorem to demonstrate full spark for a class of frames which contains harmonic frames, namely, frames which arise from concatenating harmonic frames with any number of identity basis elements:

\begin{thm}[{cf.~\cite[Theorem 1.1]{Tao:mrl05}}]
\label{thm.harmonic plus identity}
Let $N$ be prime, and pick any $M\leq N$ rows of the $N\times N$ discrete Fourier transform matrix to form the harmonic frame $H$.
Next, pick any $K\leq M$, and take $D$ to be the $M\times M$ diagonal matrix whose first $K$ diagonal entries are $\sqrt{\frac{N+K-M}{MN}}$, and whose remaining $M-K$ entries are $\sqrt{\frac{N+K}{MN}}$.
Then concatenating $DH$ with the first $K$ identity basis elements produces an $M\times(N+K)$ full spark unit norm tight frame.
\end{thm}

As an example, when $N=5$ and $K=1$, we can pick $M=3$ rows of the $5\times 5$ DFT which are indexed by $\{0,1,4\}$.  
In this case, $D$ makes the entries of the first DFT row have size $\sqrt{\frac{1}{5}}$ and the entries of the remaining rows have size $\sqrt{\frac{2}{5}}$.
Concatenating with the first identity basis element then produces an equiangular tight frame which is full spark:
\begin{equation}
\label{eq.full spark etf example}
\Phi
=\left[\begin{array}{llllll} 
\sqrt{\frac{1}{5}}\quad\quad\quad&\sqrt{\frac{1}{5}}\quad\quad\quad&\sqrt{\frac{1}{5}}&\sqrt{\frac{1}{5}}\quad\quad\quad&\sqrt{\frac{1}{5}}\quad\quad\quad&1\\
\sqrt{\frac{2}{5}}&\sqrt{\frac{2}{5}}e^{-2\pi\mathrm{i}/5}&\sqrt{\frac{2}{5}}e^{-2\pi\mathrm{i}2/5}&\sqrt{\frac{2}{5}}e^{-2\pi\mathrm{i}3/5}&\sqrt{\frac{2}{5}}e^{-2\pi\mathrm{i}4/5}&0\\
\sqrt{\frac{2}{5}}&\sqrt{\frac{2}{5}}e^{-2\pi\mathrm{i}4/5}&\sqrt{\frac{2}{5}}e^{-2\pi\mathrm{i}3/5}&\sqrt{\frac{2}{5}}e^{-2\pi\mathrm{i}2/5}&\sqrt{\frac{2}{5}}e^{-2\pi\mathrm{i}/5}&0\\     
       \end{array}\right].
\end{equation}

\begin{proof}[Proof of Theorem~\ref{thm.harmonic plus identity}]
Let $\Phi$ denote the resulting $M\times(N+K)$ frame.
We start by verifying that $\Phi$ is unit norm.
Certainly, the identity basis elements have unit norm.
For the remaining frame elements, the modulus of each entry is determined by $D$, and so the norm squared of each frame element is
\begin{equation*}
K(\tfrac{N+K-M}{MN})+(M-K)(\tfrac{N+K}{MN})
=1.
\end{equation*}
To demonstrate that $\Phi$ is tight, it suffices to show that $\Phi\Phi^*=\frac{N+K}{M}I_M$.
The rows of $DH$ are orthogonal since they are scaled rows of the DFT, while the rows of the identity portion are orthogonal because they have disjoint support.
Thus, $\Phi\Phi^*$ is diagonal.
Moreover, the norm squared of each of the first $K$ rows is $N(\frac{N+K-M}{MN})+1=\frac{N+K}{M}$, while the norm squared of each of the remaining rows is $N(\frac{N+K}{MN})=\frac{N+K}{M}$, and so $\Phi\Phi^*=\frac{N+K}{M}I_M$.

To show that $\Phi$ is full spark, note that every $M\times M$ submatrix of $DH$ is invertible since
\begin{equation*}
|\mathrm{det}((DH)_\mathcal{K})|
=|\mathrm{det}(DH_\mathcal{K})|
=|\mathrm{det}(D)||\mathrm{det}(H_\mathcal{K})|
>0,
\end{equation*}
by Chebotar\"{e}v's theorem.
Also, in the case where $K=M$, we note that the $M\times M$ submatrix of $\Phi$ composed solely of identity basis elements is trivially invertible.
The only remaining case to check is when identity basis elements and columns of $DH$ appear in the same $M\times M$ submatrix $\Phi_\mathcal{K}$.
In this case, we may shuffle the rows of $\Phi_\mathcal{K}$ to have the form 
\begin{equation*}
\begin{bmatrix} A&0\\B&I_K\end{bmatrix}.
\end{equation*}
Since shuffling rows has no impact on the size of the determinant, we may further use a determinant identity on block matrices to get
\begin{equation*}
|\mathrm{det}(\Phi_\mathcal{K})|
=\left|\mathrm{det}\begin{bmatrix} A&0\\B&I_K\end{bmatrix}\right|
=|\mathrm{det}(A)\mathrm{det}(I_K)|
=|\mathrm{det}(A)|.
\end{equation*}
Since $A$ is a multiple of a square submatrix of the $N\times N$ DFT, we are done by Chebotar\"{e}v's theorem.
\end{proof}

As an example of Theorem~\ref{thm.harmonic plus identity}, pick $N$ to be a prime congruent to $1\bmod 4$, and select $\frac{N+1}{2}$ rows of the $N\times N$ DFT according to the index set $\mathcal{M}:=\{k^2:k\in\mathbb{Z}_N\}$.
If we take $K=1$, the process in Theorem~\ref{thm.harmonic plus identity} produces an equiangular tight frame of redundancy $2$, which we will verify in the next chapter using quadratic Gauss sums; in the case where $N=5$, this construction produces \eqref{eq.full spark etf example}.
Note that this corresponds to a special case of a construction in Zauner's thesis~\cite{Zauner:phd99}, which was later studied by Renes~\cite{Renes:07} and Strohmer~\cite{Strohmer:08}.
Theorem~\ref{thm.harmonic plus identity} says that this construction is full spark.

Maximally sparse frames have recently become a subject of active research~\cite{CasazzaHKK:10,FickusMT:10}.
We note that when $K=M$, Theorem~\ref{thm.harmonic plus identity} produces a maximally sparse $M\times (N+K)$ full spark frame, having a total of $M(M-1)$ zero entries.
To see that this sparsity level is maximal, we note that if the frame had any more zero entries, then at least one of the rows would have $M$ zero entries, meaning the corresponding $M\times M$ submatrix would have a row of all zeros and hence a zero determinant.
Similar ideas were studied previously by Nakamura and Masson~\cite{NakamuraM:ieeetc82}.

Another interesting case is where $K=M=N$, i.e., when the frame constructed in Theorem~\ref{thm.harmonic plus identity} is a union of the unitary DFT and identity bases.
Unions of orthonormal bases have received considerable attention in the context of sparse approximation~\cite{DonohoE:03,Tropp:acha08}.
In fact, when $N$ is a perfect square, concatenating the DFT with an identity basis forms the canonical example $\Phi$ of a dictionary with small spark~\cite{DonohoE:03}, and we used this example in the previous chapter.
Recall the Dirac comb of $\sqrt{N}$ spikes is an eigenvector of the DFT, and so concatenating this comb with the negative of its Fourier transform produces a $2\sqrt{N}$-sparse vector in the nullspace of $\Phi$.
In stark contrast, when $N$ is prime, Theorem~\ref{thm.harmonic plus identity} shows that $\Phi$ is full spark.

The vast implications of Chebotar\"{e}v's theorem leads one to wonder whether the result admits any interesting generalization.
In this direction, Cand\`{e}s et al.~\cite{CandesRT:06} note that any such generalization must somehow account for the nontrivial subgroups of $\mathbb{Z}_N$ which are not present when $N$ is prime.
Certainly, if one could characterize the full spark submatrices of a general DFT, this would provide ample freedom to optimize full spark frames for additional considerations.
While we do not have a characterization for the general case, we do have one for the case where $N$ is a prime power.
Before stating the result, we require a definition:

\begin{defn}
We say a subset $\mathcal{M}\subseteq\mathbb{Z}_N$ is \emph{uniformly distributed over the divisors of $N$} if, for every divisor $d$ of $N$, the $d$ cosets of $\langle d\rangle$ partition $\mathcal{M}$ into subsets, each of size $\lfloor\frac{|\mathcal{M}|}{d}\rfloor$ or $\lceil\frac{|\mathcal{M}|}{d}\rceil$.
\end{defn}

At first glance, this definition may seem rather unnatural, but we will discover some important properties of uniformly distributed rows from the DFT.
As an example, we briefly consider uniform distribution in the context of the restricted isometry property (RIP).
Recall that a matrix of random rows from a DFT and normalized columns is RIP with high probability~\cite{RudelsonV:08}. 
We will show that harmonic frames satisfy RIP only if the selected row indices are nearly uniformly distributed over sufficiently small divisors of $N$.

To this end, recall that for any divisor $d$ of $N$, the Fourier transform of the $d$-sparse normalized Dirac comb $\frac{1}{\sqrt{d}}\chi_{\langle \frac{N}{d}\rangle}$ is the $\frac{N}{d}$-sparse normalized Dirac comb $\sqrt{\frac{d}{N}}\chi_{\langle d\rangle}$.
Let $F$ be the $N\times N$ unitary DFT, and let $\Phi$ be the harmonic frame which arises from selecting rows of $F$ indexed by $\mathcal{M}$ and then normalizing the columns.
In order for $\Phi$ to be $(K,\delta)$-RIP, $\mathcal{M}$ must contain at least one member of $\langle d\rangle$ for every divisor $d$ of $N$ which is $\leq K$, since otherwise
\begin{equation*}
\Phi\tfrac{1}{\sqrt{d}}\chi_{\langle\frac{N}{d}\rangle}
=\sqrt{\tfrac{N}{|\mathcal{M}|}}(F\tfrac{1}{\sqrt{d}}\chi_{\langle\frac{N}{d}\rangle})_\mathcal{M}
=\sqrt{\tfrac{N}{|\mathcal{M}|}}\Big(\sqrt{\tfrac{d}{N}}\chi_{\langle d\rangle}\Big)_\mathcal{M}
=\sqrt{\tfrac{d}{|\mathcal{M}|}}\chi_{\mathcal{M}\cap\langle d\rangle}
=0, 
\end{equation*}
which violates the lower RIP bound at $x=\frac{1}{\sqrt{d}}\chi_{\langle\frac{N}{d}\rangle}$.
In fact, the RIP bounds indicate that 
\begin{equation*}
\|\Phi x\|^2
=\|\Phi\tfrac{1}{\sqrt{d}}\chi_{\langle\frac{N}{d}\rangle}\|^2
=\Big\|\sqrt{\tfrac{d}{|\mathcal{M}|}}\chi_{\mathcal{M}\cap\langle d\rangle}\Big\|^2
=\tfrac{d}{|\mathcal{M}|}|\mathcal{M}\cap\langle d\rangle|
\end{equation*}
cannot be more than $\delta$ away from $\|x\|^2=1$.
Similarly, taking $x$ to be $\frac{1}{\sqrt{d}}\chi_{\langle\frac{N}{d}\rangle}$ modulated by $a$, i.e., $x[n]:=\frac{1}{\sqrt{d}}\chi_{\langle\frac{N}{d}\rangle}[n]e^{2\pi ian/N}$ for every $n\in\mathbb{Z}_N$, gives that $\|\Phi x\|^2=\frac{d}{|\mathcal{M}|}|\mathcal{M}\cap(a+\langle d\rangle)|$ is also no more than $\delta$ away from $1$.
This observation gives the following result:

\begin{thm}
Select rows indexed by $\mathcal{M}\subseteq\mathbb{Z}_N$ from the $N\times N$ discrete Fourier transform matrix and then normalize the columns to produce the harmonic frame $\Phi$.
Then $\Phi$ satisfies the $(K,\delta)$-restricted isometry property only if
\begin{equation*}
\Big|\big|\mathcal{M}\cap(a+\langle d\rangle)\big|-\tfrac{|\mathcal{M}|}{d}\Big|\leq\tfrac{|\mathcal{M}|}{d}\delta
\end{equation*}
for every divisor $d$ of $N$ with $d\leq K$ and every $a=0,\ldots,d-1$.
\end{thm}

Now that we have an intuition for uniform distribution in terms of modulated Dirac combs and RIP, we take this condition to the extreme by considering uniform distribution over all divisors.
Doing so produces a complete characterization of full spark harmonic frames when $N$ is a prime power:

\begin{thm}
\label{thm.uniformly distributed}
Let $N$ be a prime power, and select rows indexed by $\mathcal{M}\subseteq\mathbb{Z}_N$ from the $N\times N$ discrete Fourier transform matrix to build the submatrix $\Phi$.
Then $\Phi$ is full spark if and only if $\mathcal{M}$ is uniformly distributed over the divisors of $N$.
\end{thm}

Note that, perhaps surprisingly, an index set $\mathcal{M}$ can be uniformly distributed over $p$ but not over $p^2$, and vice versa.
For example, $\mathcal{M}=\{0,1,4\}$ is uniformly distributed over $2$ but not $4$, while $\mathcal{M}=\{0,2\}$ is uniformly distributed over $4$ but not $2$.

Since the first $M$ rows of a DFT form a full spark Vandermonde matrix, let's check that this index set is uniformly distributed over the divisors of $N$.
For each divisor $d$ of $N$, we partition the first $M$ indices into the $d$ cosets of $\langle d\rangle$.
Write $M=qd+r$ with $0\leq r<d$.
The first $qd$ of the $M$ indices are distributed equally amongst all $d$ cosets, and then the remaining $r$ indices are distributed equally amongst the first $r$ cosets.
Overall, the first $r$ cosets contain $q+1=\lfloor\frac{M}{d}\rfloor+1$ indices, while the remaining $d-r$ cosets have $q=\lfloor\frac{M}{d}\rfloor$ indices; thus, the first $M$ indices are indeed uniformly distributed over the divisors of $N$.
Also, when $N$ is prime, \emph{every} subset of $\mathbb{Z}_N$ is uniformly distributed over the divisors of $N$ in a trivial sense.
In fact, Chebotar\"{e}v's theorem follows immediately from Theorem~\ref{thm.uniformly distributed}.
In some ways, portions of our proof of Theorem~\ref{thm.uniformly distributed} mirror recurring ideas in the existing proofs of Chebotar\"{e}v's theorem~\cite{DelvauxV:laa08,EvansI:ams77,StevenhagenL:mi96,Tao:mrl05}.
For the sake of completeness, we provide the full argument and save the reader from having to parse portions of proofs from multiple references.
We start with the following lemmas, whose proofs are based on the proofs of Lemmas~1.2 and~1.3 in~\cite{Tao:mrl05}.

\begin{lem}
\label{lem.multiple of p}
Let $N$ be a power of some prime $p$, and let $P(z_1,\ldots,z_M)$ be a polynomial with integer coefficients.
Suppose there exists $N$th roots of unity $\{\omega_m\}_{m=1}^M$ such that $P(\omega_1,\ldots,\omega_M)=0$.
Then $P(1,\ldots,1)$ is a multiple of $p$.
\end{lem}

\begin{proof}
Denoting $\omega:=e^{-2\pi i/N}$, then for every $m=1,\ldots,M$, we have $\omega_m=\omega^{k_m}$ for some $0\leq k_m<N$.
Defining the polynomial $Q(z):=P(z^{k_1},\ldots,z^{k_M})$, we have $Q(\omega)=0$ by assumption.
Also, $Q(z)$ is a polynomial with integer coefficients, and so it must be divisible by the minimal polynomial of $\omega$, namely, the cyclotomic polynomial $\Phi_N(z)$.
Evaluating both polynomials at $z=1$ then gives that $p=\Phi_N(1)$ divides $Q(1)=P(1,\ldots,1)$.
\end{proof}

\begin{lem}
\label{lem.big fraction}
Let $N$ be a power of some prime $p$, and pick $\mathcal{M}=\{m_i\}_{i=1}^M\subseteq\mathbb{Z}_N$ such that
\begin{equation}
\label{eq.big fraction}
\frac{\displaystyle{\prod_{1\leq i<j\leq M}(m_j-m_i)}}{\displaystyle{\prod_{m=0}^{M-1}m!}}
\end{equation}
is not a multiple of $p$.
Then the rows indexed by $\mathcal{M}$ in the $N\times N$ discrete Fourier transform form a full spark frame. 
\end{lem}

\begin{proof}
We wish to show that $\mathrm{det}(\omega_n^m)_{m\in\mathcal{M},1\leq n\leq M}\neq0$ for all $M$-tuples of distinct $N$th roots of unity $\{\omega_n\}_{n=1}^M$.
Define the polynomial $D(z_1,\ldots,z_M):=\mathrm{det}(z_n^m)_{m\in\mathcal{M},1\leq n\leq M}$.
Since columns $i$ and $j$ of $(z_n^m)_{m\in\mathcal{M},1\leq n\leq M}$ are identical whenever $z_i=z_j$, we know that $D$ vanishes in each of these instances, and so we can factor:
\begin{equation*}
D(z_1,\ldots,z_M)=P(z_1,\ldots,z_M)\prod_{1\leq i<j\leq M}(z_j-z_i)
\end{equation*}
for some polynomial $P(z_1,\ldots,z_M)$ with integer coefficients.
By Lemma~\ref{lem.multiple of p}, it suffices to show that $P(1,\ldots,1)$ is not a multiple of $p$, since this implies $D(\omega_1,\ldots,\omega_M)$ is nonzero for all $M$-tuples of distinct $N$th roots of unity $\{\omega_n\}_{n=1}^M$.

To this end, we proceed by considering
\begin{equation}
\label{eq.A defn}
A:=\bigg(z_1\frac{\partial}{\partial z_1}\bigg)^0\bigg(z_2\frac{\partial}{\partial z_2}\bigg)^1\cdots\bigg(z_M\frac{\partial}{\partial z_M}\bigg)^{M-1}D(z_1,\ldots,z_M)\bigg|_{z_1=\cdots=z_M=1}.
\end{equation}
To compute $A$, we note that each application of $z_j\frac{\partial}{\partial z_j}$ produces terms according to the product rule.
For some terms, a linear factor of the form $z_j-z_i$ or $z_i-z_j$ is replaced by $z_j$ or $-z_j$, respectively.
For each the other terms, these linear factors are untouched, while another factor, such as $P(z_1,\ldots,z_M)$, is differentiated and multiplied by $z_j$.
Note that there are a total of $M(M-1)/2$ linear factors, and only $M(M-1)/2$ differentiation operators to apply.
Thus, after expanding every product rule, there will be two types of terms: terms in which every differentiation operator was applied to a linear factor, and terms which have at least one linear factor remaining untouched.
When we evaluate at $z_1=\cdots=z_M=1$, the terms with linear factors vanish, and so the only terms which remain came from applying every differentiation operator to a linear factor.
Furthermore, each of these terms before the evaluation is of the form $P(z_1,\ldots,z_M)\prod_{1\leq i<j\leq M}z_j$, and so evaluation at $z_1=\cdots=z_M=1$ produces a sum of terms of the form $P(1,\ldots,1)$; to determine the value of $A$, it remains to count these terms.
The $M-1$ copies of $z_M\frac{\partial}{\partial z_M}$ can only be applied to linear factors of the form $z_M-z_i$, of which there are $M-1$, and so there are a total of $(M-1)!$ ways to distribute these operators.
Similarly, there are $(M-2)!$ ways to distribute the $M-2$ copies of $z_{M-1}\frac{\partial}{\partial z_{M-1}}$ amongst the $M-2$ linear factors of the form $z_{M-1}-z_i$.
Continuing in this manner produces an expression for $A$:
\begin{equation}
\label{eq.first A}
A=(M-1)!(M-2)!\cdots 1!0!~P(1,\ldots,1).
\end{equation}

For an alternate expression of $A$, we substitute the definition of $D(z_1,\ldots,z_M)$ into $\eqref{eq.A defn}$.
Here, we exploit the multilinearity of the determinant and the fact that $(z_n\frac{\partial}{\partial z_n})z_n^m=mz_n^m$ to get
\begin{equation}
\label{eq.second A}
A=\mathrm{det}(m^{n-1})_{m\in\mathcal{M},1\leq n\leq M}=\prod_{1\leq i<j\leq M}(m_j-m_i),
\end{equation}
where the final equality uses the fact that $(m^{n-1})_{m\in\mathcal{M},1\leq n\leq M}$ is the transpose of a Vandermonde matrix.
Equating \eqref{eq.first A} to \eqref{eq.second A} reveals that \eqref{eq.big fraction} is an expression for $P(1,\ldots,1)$.  
Thus, by assumption, $P(1,\ldots,1)$ is not a multiple of $p$, and so we are done.
\end{proof}

\begin{proof}[Proof of Theorem~\ref{thm.uniformly distributed}]
($\Leftarrow$) We will use Lemma~\ref{lem.big fraction} to demonstrate that $\Phi$ is full spark.
To apply this lemma, we need to establish that \eqref{eq.big fraction} is not a multiple of $p$, and to do this, we will show that there are as many $p$-divisors in the numerator of \eqref{eq.big fraction} as there are in the denominator.
We start by counting the $p$-divisors of the denominator:
\begin{equation}
\label{eq.denominator 1}
\prod_{m=0}^{M-1}m!=\prod_{m=1}^{M-1}\prod_{\ell=1}^m \ell=\prod_{\ell=1}^{M-1}\prod_{m=1}^{M-l} \ell.
\end{equation}
For each pair of integers $k,a\geq 1$, there are $\max\{M-ap^k,~0\}$ factors in \eqref{eq.denominator 1} of the form $\ell=ap^k$.
By adding these, we count each factor $\ell$ as many times as it can be expressed as a multiple of a power of $p$, which equals the number of $p$-divisors in $\ell$.
Thus, the number of $p$-divisors of \eqref{eq.denominator 1} is
\begin{equation}
\label{eq.denominator 2}
\sum_{k=1}^{\lfloor \log_p M\rfloor}\sum_{a=1}^{\lfloor \frac{M}{p^k}\rfloor}(M-ap^k).
\end{equation}

Next, we count the $p$-divisors of the numerator of \eqref{eq.big fraction}.
To do this, we use the fact that $\mathcal{M}$ is uniformly distributed over the divisors of $N$. 
Since $N$ is a power of $p$, the only divisors of $N$ are smaller powers of $p$.
Also, the cosets of $\langle p^k\rangle$ partition $\mathcal{M}$ into subsets $S_{k,b}:=\{m_i\equiv b\mod p^k\}$.
We note that $m_j-m_i$ is a multiple of $p^k$ precisely when $m_i$ and $m_j$ belong to the same subset $S_{k,b}$ for some $0\leq b<p^k$.
To count $p$-divisors, we again count each factor $m_j-m_i$ as many times as it can be expressed as a multiple of a prime power:
\begin{equation}
\label{eq.numerator 1}
\sum_{k=1}^{\lfloor \log_p M\rfloor}\sum_{b=0}^{p^k-1}\binom{|S_{k,b}|}{2}.
\end{equation}
Write $M=qp^k+r$ with $0\leq r<p^k$.
Then $q=\lfloor\frac{M}{p^k}\rfloor$.
Since $\mathcal{M}$ is uniformly distributed over $p^k$, there are $r$ subsets $S_{k,b}$ with $q+1$ elements and $p^k-r$ subsets with $q$ elements.
We use this to get
\begin{equation*}
\sum_{b=0}^{p^k-1}\binom{|S_{k,b}|}{2}
=\binom{q+1}{2}r+\binom{q}{2}(p^k-r)
=\frac{q}{2}\Big((q-1)p^k+2r+(qp^k-qp^k)\Big).
\end{equation*}
Rearranging and substituting $M=qp^k+r$ then gives
\begin{equation*}
\sum_{b=0}^{p^k-1}\binom{|S_{k,b}|}{2}
=\frac{q}{2}\Big(2M-(q+1)p^k\Big)
=Mq-\binom{q+1}{2}p^k
=\sum_{a=1}^{\lfloor \frac{M}{p^k}\rfloor}(M-ap^k).
\end{equation*}
Thus, there are as many $p$-divisors in the numerator \eqref{eq.numerator 1} as there are in the denominator \eqref{eq.denominator 2}, and so \eqref{eq.big fraction} is not divisible by $p$.
Lemma~\ref{lem.big fraction} therefore gives that $\Phi$ is full spark.

($\Rightarrow$)
We will prove that this direction holds regardless of whether $N$ is a prime power.
Suppose $\mathcal{M}\subseteq\mathbb{Z}_N$ is not uniformly distributed over the divisors of $N$.
Then there exists a divisor $d$ of $N$ such that one of the cosets of $\langle d\rangle$ intersects $\mathcal{M}$ with $\leq\lfloor \frac{M}{d}\rfloor-1$ or $\geq\lceil\frac{M}{d}\rceil+1$ indices.
Notice that if a coset of $\langle d\rangle$ intersects $\mathcal{M}$ with $\leq\lfloor \frac{M}{d}\rfloor-1$ indices, then the complement $\mathcal{M}^\mathrm{c}$ intersects the same coset with $\geq\lceil\frac{N-M}{d}\rceil+1=\lceil\frac{|\mathcal{M}^\mathrm{c}|}{d}\rceil+1$ indices.
By Theorem~\ref{thm.closure}(iii), $\mathcal{M}$ produces a full spark harmonic frame precisely when $\mathcal{M}^\mathrm{c}$ produces a full spark harmonic frame, and so we may assume without loss of generality that there exists a coset of $\langle d\rangle$ which intersects $\mathcal{M}$ with $\geq\lceil\frac{M}{d}\rceil+1$ indices.

To prove that the rows with indices in $\mathcal{M}$ are not full spark, we find column entries which produce a singular submatrix.
Writing $M=qd+r$ with $0\leq r<d$, let $\mathcal{K}$ contain $q=\lfloor\frac{M}{d}\rfloor$ cosets of $\langle\frac{N}{d}\rangle$ along with $r$ elements from an additional coset.
We claim that the DFT submatrix with row entries $\mathcal{M}$ and column entries $\mathcal{K}$ is singular.
To see this, shuffle the rows and columns to form a matrix $A$ in which the row entries are grouped into common cosets of $\langle d\rangle$ and the column entries are grouped into common cosets of $\langle\frac{N}{d}\rangle$.
This breaks $A$ into rank-1 submatrices: each pair of cosets $a+\langle d\rangle$ and $b+\langle\frac{N}{d}\rangle$ produces a submatrix
\begin{equation*}
(\omega^{(a+id)(b+j\frac{N}{d})})_{i\in\mathcal{I},j\in\mathcal{J}}
=\omega^{ab}(\omega^{bdi}\omega^{a\frac{N}{d}j})_{i\in\mathcal{I},j\in\mathcal{J}}
\end{equation*}
for some index sets $\mathcal{I}$ and $\mathcal{J}$; this is a rank-1 outer product.
Let $\mathcal{L}$ be the largest intersection between $\mathcal{M}$ and a coset of $\langle d\rangle$.
Then $|\mathcal{L}|\geq\lceil\frac{M}{d}\rceil+1$ is the number of rows in the tallest of these rank-1 submatrices.
Define $A_\mathcal{L}$ to be the $M\times M$ matrix with entries $A_\mathcal{L}[i,j]=A[i,j]$ whenever $i\in\mathcal{L}$ and zero otherwise.
Then 
\begin{equation}
\label{eq.rank 1}
\mathrm{Rank}(A)
=\mathrm{Rank}(A_\mathcal{L}+A-A_\mathcal{L})
\leq\mathrm{Rank}(A_\mathcal{L})+\mathrm{Rank}(A-A_\mathcal{L}).
\end{equation}
Since $A-A_\mathcal{L}$ has $|\mathcal{L}|$ rows of zero entries, we also have
\begin{equation}
\label{eq.rank 2}
\mathrm{Rank}(A-A_\mathcal{L})
\leq M-|\mathcal{L}|\leq M-(\lceil\tfrac{M}{d}\rceil+1).
\end{equation}
Moreover, since we can decompose $A_\mathcal{L}$ into a sum of $\lceil\frac{M}{d}\rceil$ zero-padded rank-1 submatrices, we have $\mathrm{Rank}(A_\mathcal{L})
\leq \lceil\frac{M}{d}\rceil$.
Combining this with \eqref{eq.rank 1} and \eqref{eq.rank 2} then gives that $\mathrm{Rank}(A)\leq M-1$, and so the DFT submatrix is not invertible.
\end{proof}

Note that our proof of Theorem~\ref{thm.uniformly distributed} establishes the necessity of having row indices uniformly distributed over the divisors of $N$ in the general case.
This leaves some hope for completely characterizing full spark harmonic frames.
Naturally, one might suspect that the uniform distribution condition is sufficient in general, but this suspicion fails when $N=10$.
Indeed, the following DFT submatrix is singular despite the row indices being uniformly distributed over the divisors of $10$:
\begin{equation*}
(e^{-2\pi i mn/10})_{m\in\{0,1,3,4\},n\in\{0,1,2,6\}}.
\end{equation*}

Just as we used Chebotar\"{e}v's theorem to analyze the harmonic equiangular tight frames from Xia et al.~\cite{XiaZG:05}, we can also use Theorem~\ref{thm.uniformly distributed} to determine whether harmonic equiangular tight frames with a prime power number of frame elements are full spark.
Unfortunately, none of the infinite families in~\cite{XiaZG:05} have the number of frame elements in the form of a prime power (other than primes).
Luckily, there is at least one instance in which the number of frame elements happens to be a prime power: the harmonic frames that arise from Singer difference sets have 
$M=\frac{q^d-1}{q-1}$ and $N=\frac{q^{d+1}-1}{q-1}$ for a prime power $q$ and an integer $d\geq2$; when $q=3$ and $d=4$, the number of frame elements $N=11^2$ is a prime power.
In this case, the row indices we select are
\begin{align*}
\mathcal{M}
=&\{1, 2, 3, 6, 7, 9, 11, 18, 20, 21, 25, 27, 33, 34, 38, 41, 44, 47, 53, 54, 55, 56, \\
& ~~~  58, 59, 60, 63, 64, 68, 70, 71, 75, 81, 83, 89, 92, 99, 100, 102, 104, 114\},
\end{align*}
but these are not uniformly distributed over 11, and so the corresponding harmonic frame is not full spark by Theorem~\ref{thm.uniformly distributed}.

\section{The computational complexity of verifying full spark}

In the previous section, we constructed a large collection of deterministic full spark frames.
To see how special these constructions are, we consider the following question:  
How much computation is required to check whether any given frame is full spark?
At the heart of the matter is computational complexity theory, which provides a rigorous playing field for expressing how hard certain problems are.
In this section, we consider the complexity of the following problem:

\begin{prob}[\textsc{Full Spark}]
\label{prob.full spark}
Given a matrix, is it full spark?
\end{prob}

For the lay mathematician, \textsc{Full Spark} is ``obviously'' $\NP$-hard because the easiest way he can think to solve it for a given $M\times N$ matrix is by determining whether each of the $M\times M$ submatrices is invertible; computing $\binom{N}{M}$ determinants would do, but this would take a lot of time, and so \textsc{Full Spark} must be $\NP$-hard.
However, computing $\binom{N}{M}$ determinants may not necessarily be the fastest way to test whether a matrix is full spark.
For example, perhaps there is an easy-to-calculate expression for the product of the determinants; after all, this product is nonzero precisely when the matrix is full spark.
Recall that Theorem~\ref{thm.uniformly distributed} gives a very straightforward litmus test for \textsc{Full Spark} in the special case where the matrix is formed by rows of a DFT of prime-power order---who's to say that a version of this test does not exist for the general case?
If such a test exists, then it would suffice to find it, but how might one disprove the existence of any such test?
Indeed, since we are concerned with the necessary amount of computation, as opposed to a sufficient amount, the lay mathematician's intuition is a bit misguided.

To discern how much computation is necessary, the main feature of interest is a problem's \emph{complexity}.
We use complexity to compare problems and determine whether one is harder than the other.
As an example of complexity, intuitively, doubling an integer is no harder than adding integers, since one can use addition to multiply by $2$; put another way, the complexity of doubling is somehow ``encoded'' in the complexity of adding, and so it must be lesser (or equal).
To make this more precise, complexity theorists use what is called a \emph{polynomial-time reduction}, that is, a polynomial-time algorithm that solves problem $A$ by exploiting an oracle which solves problem $B$; the reduction indicates that solving problem $A$ is no harder than solving problem $B$ (up to polynomial factors in time), and we say ``$A$ reduces to $B$,'' or $A\leq B$.
Since we can use the polynomial-time routine $x+x$ to produce $2x$, we conclude that doubling an integer reduces to adding integers, as expected.

In complexity theory, problems are categorized into complexity classes according to the amount of resources required to solve them.
For example, the complexity class $\P$ contains all problems which can be solved in polynomial time, while problems in $\EXP$ may require as much as exponential time.
Problems in $\NP$ have the defining quality that solutions can be verified in polynomial time given a certificate for the answer.
As an example, the graph isomorphism problem is in $\NP$ because, given an isomorphism between graphs (a certificate), one can verify that the isomorphism is legit in polynomial time.
Clearly, $\P\subseteq\NP$, since we can ignore the certificate and still solve the problem in polynomial time.
Finally, a problem $B$ is called $\NP$-\emph{hard} if every problem $A$ in $\NP$ reduces to $B$, and a problem is called $\NP$-\emph{complete} if it is both $\NP$-hard and in $\NP$.
In plain speak, $\NP$-hard problems are harder than every problem in $\NP$, while $\NP$-complete problems are the hardest of problems in $\NP$.

At this point, it should be clear that $\NP$-hard problems are not merely problems that seem to require a lot of computation to solve.
Certainly, $\NP$-hard problems have this quality, as an $\NP$-hard problem can be solved in polynomial time only if $\P=\NP$; this is an open problem, but it is 
widely believed that $\P\neq\NP$.
However, there are other problems which seem hard but are not known to be $\NP$-hard (e.g., the graph isomorphism problem).
Rather, to determine whether a problem is $\NP$-hard, one must find a polynomial-time reduction that compares the problem to all problems in $\NP$.
To this end, notice that $A\leq B$ and $B\leq C$ together imply $A\leq C$, and so to demonstrate that a problem $C$ is $\NP$-hard, it suffices to show that $B\leq C$ for some $\NP$-hard problem $B$.

Unfortunately, it can sometimes be difficult to find a deterministic reduction from one problem to another.
One example is reducing the satisfiability problem (\textsc{SAT}) to the unique satisfiability problem (\textsc{Unique~SAT}).
To be clear, \textsc{SAT} is an $\NP$-hard problem~\cite{Karp:ccc72} that asks whether there exists an input for which a given Boolean function returns ``true,'' while \textsc{Unique~SAT} asks the same question with an additional promise: that the given Boolean function is satisfiable only if there is a \emph{unique} input for which it returns ``true.''
Intuitively, \textsc{Unique~SAT} is easier than \textsc{SAT} because we might be able to exploit the additional structure of uniquely satisfiable Boolean functions; thus, it could be difficult to find a reduction from \textsc{SAT} to \textsc{Unique~SAT}.
Despite this intuition, there is a \emph{randomized} polynomial-time reduction from \textsc{SAT} to \textsc{Unique~SAT}~\cite{ValiantV:tcs86}.
Defined over all Boolean functions of $n$ variables, the reduction maps functions that are not satisfiable to other functions that are not satisfiable, and with probability $\geq\frac{1}{8n}$, it maps satisfiable functions to uniquely satisfiable functions.
After applying this reduction to a given Boolean function, if a \textsc{Unique~SAT} oracle declares ``uniquely satisfiable,'' then we know for certain that the original Boolean function was satisfiable.
But the reduction will only map a satisfiable problem to a uniquely satisfiable problem with probability $\geq\frac{1}{8n}$, so what good is this reduction?
The answer lies in something called \emph{amplification}; since the success probability is, at worst, polynomially small in $n$ (i.e., $\geq\frac{1}{p(n)}$), we can repeat our oracle-based randomized algorithm a polynomial number of times $np(n)$ and achieve an error probability $\leq(1-\frac{1}{p(n)})^{np(n)}\sim e^{-n}$ which is exponentially small.

In this section, we give a randomized polynomial-time reduction from a problem in matroid theory.
Before stating the problem, we first briefly review some definitions.
To each bipartite graph with bipartition $(E,E')$, we associate a \emph{transversal matroid} $(E,\mathcal{I})$, where $\mathcal{I}$ is the collection of subsets of $E$ whose vertices form the ends of a matching in the bipartite graph; subsets in $\mathcal{I}$ are called $\emph{independent}$.
Next, just as spark is the size of the smallest linearly dependent set, the \emph{girth} of a matroid is the size of the smallest subset of $E$ that is not in $\mathcal{I}$.
In fact, this analogy goes deeper: 
A matroid is \emph{representable over a field} $\mathbb{F}$ if, for some $M$, there exists a mapping $\varphi\colon E\rightarrow\mathbb{F}^M$ such that $\varphi(A)$ is linearly independent if and only if $A\in\mathcal{I}$; as such, the girth of $(E,\mathcal{I})$ is the spark of $\varphi(E)$.
In our reduction, we make use of the fact that every transversal matroid is representable over $\mathbb{R}$~\cite{PiffW:jlms70}.
We are now ready to state the problem from which we will reduce \textsc{Full Spark}:

\begin{prob}
\label{prob.girth}
Given a bipartite graph, what is the girth of its transversal matroid?
\end{prob}

Before giving the reduction, we note that Problem~\ref{prob.girth} is $\NP$-hard.
This is demonstrated in McCormick's thesis~\cite{McCormick:phd83}, which credits the proof to Stockmeyer; since~\cite{McCormick:phd83} is difficult to access, we refer the reader to~\cite{AlexeevCM:arxiv11}.
We now turn to the main result of this section; note that our proof is specifically geared toward the case where the matrix in question has integer entries---this is stronger than manipulating real (complex) numbers exactly as well as with truncations and tolerances.

\begin{thm}
\label{thm.full spark hard}
\textsc{Full Spark} is hard for $\NP$ under randomized polynomial-time reductions.
\end{thm}

\begin{proof}
We will give a randomized polynomial-time reduction from Problem~\ref{prob.girth} to \textsc{Full Spark}.
As such, suppose we are given a bipartite graph $G$, in which every edge is between the disjoint sets $A$ and $B$.
Take $M:=|B|$ and $N:=|A|$.
Using this graph, we randomly draw an $M\times N$ matrix $\Phi$ using the following process: for each $i\in B$ and $j\in A$, pick the entry $\Phi_{ij}$ randomly from $\{1,\ldots,N2^{N+1}\}$ if $i\leftrightarrow j$ in $G$; otherwise set $\Phi_{ij}=0$.  
In Proposition 3.11 of~\cite{Marx:tcs09}, it is shown that the columns of $\Phi$ form a representation of the transversal matroid of $G$ with probability $\geq\frac{1}{2}$.
For the moment, we assume that $\Phi$ succeeds in representing the matroid.

Since the girth of the original matroid equals the spark of its representation, for each $K=1,\ldots,M$, we test whether $\mathrm{Spark}(\Phi)>K$.
To do this, take $H$ to be some $M\times P$ full spark frame.
We will determine an appropriate value for $P$ later, but for simplicity, we can take $H$ to be the Vandermonde matrix formed from bases $\{1,\ldots,P\}$; see Lemma~\ref{lem.vandermonde}.
We claim we can randomly select $K$ indices $\mathcal{K}\subseteq\{1,\ldots,P\}$ and test whether $H_\mathcal{K}^*\Phi$ is full spark to determine whether $\mathrm{Spark}(\Phi)>K$.
Moreover, after performing this test for each $K=1,\ldots,M$, the probability of incorrectly determining $\mathrm{Spark}(\Phi)$ is $\leq\frac{1}{2}$, provided $P$ is sufficiently large.

We want to test whether $H_\mathcal{K}^*\Phi$ is full spark and use the result as a proxy for whether $\mathrm{Spark}(\Phi)>K$.
For this to work, we need to have $\mathrm{Rank}(H_\mathcal{K}^*\Phi_{\mathcal{K}'})=K$ precisely when $\mathrm{Rank}(\Phi_{\mathcal{K}'})=K$ for every $\mathcal{K}'\subseteq\{1,\ldots,N\}$ of size $K$.
To this end, it suffices to have the nullspace $\mathcal{N}(H_\mathcal{K}^*)$ of $H_\mathcal{K}^*$ intersect trivially with the column space of $\Phi_{\mathcal{K}'}$ for every $\mathcal{K}'$.
To be clear, it is always the case that $\mathrm{Rank}(H_\mathcal{K}^*\Phi_{\mathcal{K}'})\leq\mathrm{Rank}(\Phi_{\mathcal{K}'})$, and so $\mathrm{Rank}(\Phi_{\mathcal{K}'})<K$ implies $\mathrm{Rank}(H_\mathcal{K}^*\Phi_{\mathcal{K}'})<K$.
If we further assume that $\mathcal{N}(H_\mathcal{K}^*)\cap\mathrm{Span}(\Phi_{\mathcal{K}'})=\{0\}$, then the converse also holds.
To see this, suppose $\mathrm{Rank}(H_\mathcal{K}^*\Phi_{\mathcal{K}'})<K$.
Then by the rank-nullity theorem, there is a nontrivial $x\in\mathcal{N}(H_\mathcal{K}^*\Phi_{\mathcal{K}'})$.
Since $H_\mathcal{K}^*\Phi_{\mathcal{K}'}x=0$, we must have $\Phi_{\mathcal{K}'}x\in\mathcal{N}(H_\mathcal{K}^*)$, which in turn implies $x\in\mathcal{N}(\Phi_{\mathcal{K}'})$ since $\mathcal{N}(H_\mathcal{K}^*)\cap\mathrm{Span}(\Phi_{\mathcal{K}'})=\{0\}$ by assumption.
Thus, $\mathrm{Rank}(\Phi_{\mathcal{K}'})<K$ by the rank-nullity theorem.

Now fix $\mathcal{K}'\subseteq\{1,\ldots,N\}$ of size $K$ such that $\mathrm{Rank}(\Phi_{\mathcal{K}'})=K$.
We will show that the vast majority of choices $\mathcal{K}\subseteq\{1,\ldots,P\}$ of size $K$ satisfy $\mathcal{N}(H_\mathcal{K}^*)\cap\mathrm{Span}(\Phi_{\mathcal{K}'})=\{0\}$.
To do this, we consider the columns $\{h_k\}_{k\in\mathcal{K}}$ of $H_\mathcal{K}$ one at a time, and we make use of the fact that $\mathcal{N}(H_\mathcal{K}^*)=\bigcap_{k\in\mathcal{K}}\mathcal{N}(h_k^*)$.
In particular, since $H$ is full spark, there are at most $M-K$ columns of $H$ in the orthogonal complement of $\mathrm{Span}(\Phi_{\mathcal{K}'})$, and so there are at least $P-(M-K)$ choices of $h_{k_1}$ for which $\mathcal{N}(h_{k_1}^*)$ does not contain $\mathrm{Span}(\Phi_{\mathcal{K}'})$, i.e.,
\begin{equation*}
\mathrm{dim}\Big(\mathcal{N}(h_{k_1}^*)\cap\mathrm{Span}(\Phi_{\mathcal{K}'})\Big)=K-1.
\end{equation*}
Similarly, after selecting the first $J$ $h_k$'s, we have $\mathrm{dim}(S)=K-J$, where
\begin{equation*}
S:=\bigcap_{j=1}^J\mathcal{N}(h_{k_j}^*)\cap\mathrm{Span}(\Phi_{\mathcal{K}'}).
\end{equation*}
Again, since $H$ is full spark, there are at most $M-(K-J)$ columns of $H$ in the orthogonal complement of $S$, and so the remaining $P-(M-(K-J))$ columns are candidates for $h_{k_{J+1}}$ that give
\begin{equation*}
\mathrm{dim}\bigg(\bigcap_{j=1}^{J+1}\mathcal{N}(h_{k_j}^*)\cap\mathrm{Span}(\Phi_{\mathcal{K}'})\bigg)
=\mathrm{dim}\Big(\mathcal{N}(h_{k_{J+1}}^*)\cap S\Big)
=K-(J+1).
\end{equation*}
Overall, if we randomly pick $\mathcal{K}\subseteq\{1,\ldots,P\}$ of size $K$, then
\begin{align*}
\mathrm{Pr}\Big(\mathcal{N}(H_\mathcal{K}^*)\cap\mathrm{Span}(\Phi_{\mathcal{K}'})=\{0\}\Big)
&\geq(1-\tfrac{M-K}{P})(1-\tfrac{M-(K-1)}{P})\cdots(1-\tfrac{M-1}{P})\\
&\geq(1-\tfrac{M}{P})^K\\
&\geq 1-\tfrac{MK}{P},
\end{align*}
where the final step is by Bernoulli's inequality.
Taking a union bound over all choices of $\mathcal{K}'\subseteq\{1,\ldots,N\}$ and all values of $K=1,\ldots,M$ then gives
\begin{align*}
\mathrm{Pr}\bigg(\mbox{fail to determine $\mathrm{Spark}(\Phi)$}\bigg)
&\leq\sum_{K=1}^M\binom{N}{K}\mathrm{Pr}\Big(\mathcal{N}(H_\mathcal{K}^*)\cap\mathrm{Span}(\Phi_{\mathcal{K}'})\neq\{0\}\Big)\\
&\leq\sum_{K=1}^M\binom{N}{K}\frac{MK}{P}\\
&\leq\frac{M^32^N}{P}.
\end{align*}
Thus, to make the probability of failure $\leq\frac{1}{2}$, it suffices to have $P=M^32^{N+1}$.

In summary, we succeed in representing the original matroid with probability $\geq\frac{1}{2}$, and then we succeed in determining the spark of its representation with probability $\geq\frac{1}{2}$.
The probability of overall success is therefore $\geq\frac{1}{4}$.
Since our success probability is, at worst, polynomially small, we can apply amplification to achieve an exponentially small error probability.
\end{proof}

Our use of random linear projections in the above reduction to \textsc{Full Spark} is similar in spirit to Valiant and Vazirani's use of random hash functions in their reduction to \textsc{Unique~SAT}~\cite{ValiantV:tcs86}.
Since their randomized reduction is the canonical example thereof, we find our reduction to be particularly natural.

To conclude this section, we clarify that Theorem~\ref{thm.full spark hard} is a statement about the amount of computation necessary in the \emph{worst case}.
Indeed, the hardness of \textsc{Full Spark} does not rule out the existence of smaller classes of matrices for which full spark is easily determined.
As an example, Theorem~\ref{thm.uniformly distributed} determines \textsc{Full Spark} in the special case where the matrix is formed by rows of a DFT of prime-power order.
This illustrates the utility of applying additional structure to efficiently solve the \textsc{Full Spark} problem, and indeed, such classes of matrices are rather special for this reason.

\section{Phaseless recovery with polarization}

In the previous sections, we constructed deterministic full spark frames and showed that checking for full spark in general is computationally hard.
In this section, we provide a new technique for phaseless recovery which makes use of full spark frames in the measurement design.
We are particularly interested in using the fewest measurements necessary for recovery, namely $N=\mathrm{O}(M)$, where $M$ is the dimension of the signal~\cite{BalanCE:acha06}.

Take a finite set $V$, and suppose we take phaseless measurements of $x\in\mathbb{C}^M$ with a frame $\Phi_V:=\{\varphi_i\}_{i\in V}\subseteq\mathbb{C}^M$ with the task of recovering $x$ up to a global phase factor.
For notational convenience, we take $\sim$ to be the equivalence relation of being identical up to a global phase factor, and we say $y$ is a member of the equivalence class $[x]\in\mathbb{C}^M/\!\!\sim$ if $y\sim x$.
Having $|\langle x,\varphi_i\rangle|$ for every $i\in V$, we claim it suffices to determine the relative phase between all pairs of frame coefficients. 
If we had this information, we could arbitrarily assign some nonzero frame coefficient $c_i=|\langle x,\varphi_i\rangle|$ to have positive phase.
If $\langle x,\varphi_j\rangle$ is also nonzero, then it has well-defined relative phase
\begin{equation*}
\omega_{ij}:=\big(\tfrac{\langle x,\varphi_i\rangle}{|\langle x,\varphi_i\rangle|}\big)^{-1}\tfrac{\langle x,\varphi_j\rangle}{|\langle x,\varphi_j\rangle|},
\end{equation*}
which determines the frame coefficent by multiplication: $c_j=\omega_{ij}|\langle x,\varphi_j\rangle|$.
Otherwise when $\langle x,\varphi_j\rangle=0$, we naturally take $c_j=0$, and for notational convenience, we arbitrarily take $\omega_{ij}=1$.
From here, $[x]\in\mathbb{C}^M/\!\!\sim$ can be identified by applying the canonical dual frame $\{\tilde{\varphi}_j\}_{j\in V}$ of $\Phi_V$:
\begin{equation*}
\sum_{j\in V}c_j\tilde{\varphi}_j
=\sum_{j\in V}\omega_{ij}|\langle x,\varphi_j\rangle|\tilde{\varphi}_j
=\big(\tfrac{\langle x,\varphi_i\rangle}{|\langle x,\varphi_i\rangle|}\big)^{-1}\sum_{j\in V}\langle x,\varphi_j\rangle\tilde{\varphi}_j
=\big(\tfrac{\langle x,\varphi_i\rangle}{|\langle x,\varphi_i\rangle|}\big)^{-1}x
\in[x].
\end{equation*}
To find the relative phase between frame coefficients, we turn to the polarization identity:
\begin{equation*}
\overline{\langle x,\varphi_i\rangle}\langle x,\varphi_j\rangle
=\frac{1}{4}\sum_{k=0}^3\mathrm{i}^{k}\big|\langle x,\varphi_i\rangle+\mathrm{i}^{-k}\langle x,\varphi_j\rangle\big|^2
=\frac{1}{4}\sum_{k=0}^3\mathrm{i}^{k}\big|\langle x,\varphi_i+\mathrm{i}^{k}\varphi_j\rangle\big|^2.
\end{equation*}
Thus, if in addition to $\Phi_V$, we measure with $\{\varphi_i+\mathrm{i}^{k}\varphi_j\}_{k=0}^3$, we can use the above calculation to determine $\overline{\langle x,\varphi_i\rangle}\langle x,\varphi_j\rangle$ and then normalize to get the relative phase $\omega_{ij}$, provided both $\langle x,\varphi_i\rangle$ and $\langle x,\varphi_j\rangle$ are nonzero.
To summarize, if we measure with $\Phi_V$ and $\{\varphi_i+\mathrm{i}^{k}\varphi_j\}_{k=0}^3$ for every pair $i,j\in V$, then we can recover $[x]$.
However, such a method uses $|V|+4\binom{|V|}{2}$ measurements, and since $\Phi_V$ is a frame, we necessarily have $|V|\geq M$ and thus a total of $\Omega(M^2)$ measurements.

In pursuit of $\mathrm{O}(M)$ measurements, take some simple graph $G=(V,E)$, and only take measurements with $\Phi_V$ and $\Phi_E:=\bigcup_{(i,j)\in E}\{\varphi_i+\mathrm{i}^{k}\varphi_j\}_{k=0}^3$.
To recover $[x]$, we again arbitrarily assign some nonzero vertex measurement to have positive phase, and then we propagate relative phase information along the edges by multiplication to determine the phase of the other vertex measurements relative to the original vertex measurement.
However, if $x$ is orthogonal to a given vertex vector, then that measurement is zero, and so relative phase information cannot propagate through the corresponding vertex; indeed, such orthogonality has the effect of removing the vertex from the graph, and for some graphs, this will prevent recovery.
For example, if $G$ is a star, then $x$ could be orthogonal to the vector corresponding to the internal vertex, whose removal would render the remaining graph edgeless.
That said, we should select $\Phi_V$ and $G$ so as to minimize the impact of orthogonality with vertex vectors.

First, we can take $\Phi_V$ to be full spark so that every subcollection of $M$ frame elements spans.
This implies that $x$ is orthogonal to at most $M-1$ members of $\Phi_V$, thereby limiting the extent of $x$'s damage to our graph.
Additionally, $\Phi_V$ being full spark frees us from requiring the graph to be connected after the removal of vertices; indeed, any remaining component of size $M$ or more will correspond to a subframe of $\Phi_V$ that necessarily has a dual frame to reconstruct with.
It remains to find a graph of $\mathrm{O}(M)$ vertices and edges that maintains a size-$M$ component after the removal of any $M-1$ vertices.

To this end, we consider a well-studied family of sparse graphs known as \emph{expander graphs}.
We choose these graphs for their notably strong connectivity properties.
There is a combinatorial definition of expander graphs, but we will focus on the spectral definition.
Given a $d$-regular graph $G$ of $n$ vertices, consider the eigenvalues of its adjacency matrix: $\lambda_1\geq\lambda_2\geq\cdots\geq\lambda_n$.
We say $G$ has \emph{expansion} $\lambda(G):=\frac{1}{d}\max\{|\lambda_2|,|\lambda_n|\}$.
Furthermore, a family of $d$-regular graphs $\{G_i\}_{i=1}^\infty$ is a \emph{spectral expander family} if there exists $c<1$ such that every $G_i$ has expansion $\lambda(G_i)\leq c$.
Since $d$ is constant over an expander family, we see that expanders with many vertices are particularly sparse.
There are many results which describe the connectivity of expanders, but the following is particularly relevant to our application:

\begin{lem}[\cite{HarshaB:05}]
Consider a $d$-regular graph $G$ of $n$ vertices with spectral expansion $\leq\lambda$.
For all $\varepsilon\leq\frac{1-\lambda}{6}$, removing any $\varepsilon dn$ edges from $G$ results in a connected component of size $\geq(1-\frac{2\varepsilon}{1-\lambda})n$.
\end{lem}

For our application, removing $\varepsilon n$ vertices from a $d$-regular graph necessarily removes $\leq\varepsilon dn$ edges, and so this lemma directly applies.
Also,
\begin{equation*}
\varepsilon
\leq\frac{1-\lambda}{6}
<\frac{1}{6}
<\frac{2}{3}
\leq 1-\frac{2\varepsilon}{1-\lambda},
\end{equation*}
where the last inequality is a rearrangement of $\varepsilon\leq\frac{1-\lambda}{6}$.
Since we want to guarantee that the removal of any $M-1$ vertices maintains a size-$M$ component, we must therefore take $M\leq\varepsilon n+1$.
Overall, we use the following criteria to pick our expander graph:
Given the signal dimension $M$, use a $d$-regular graph $G=(V,E)$ of $n$ vertices with spectral expansion $\lambda$ such that $M\leq(\frac{1-\lambda}{6})n+1$.
Then by the previous discussion, the total number of measurements is $N=|V|+4|E|=(2d+1)n$.
We wish to find choices of graphs which yield only $N=\mathrm{O}(M)$ measurements.

To minimize the redundancy $\frac{N}{M}$, we see that for a fixed degree $d$, we would like minimal spectral expansion $\lambda$.
Spectral graph families known as \emph{Ramanujan graphs} are asymptotically optimal in this sense; taking $\mathcal{G}_n^d$ to be the set of connected $d$-regular graphs with $\geq n$ vertices, Alon and Boppana (see~\cite{Alon:86}) showed that for any fixed $d$,
\begin{equation*} \lim_{n\rightarrow\infty}\inf_{G\in\mathcal{G}_n^d}\lambda(G)\geq\frac{2\sqrt{d-1}}{d},
\end{equation*}
while Ramanujan graphs are defined to have spectral expansion $\leq\frac{2\sqrt{d-1}}{d}$.
To date, Ramanujan graphs have only been constructed for certain values of $d$.
One important construction was given by Lubotzky et al.~\cite{LubotzkyPS:88}, which produces a Ramanujan family whenever $d-1\equiv 1\bmod4$ is prime.
Among these graphs, we get the smallest redundancy $\frac{N}{M}$ when $d=6$ and $M=\lfloor(\frac{1-\lambda}{6})n+1\rfloor$:
\begin{equation*}
\frac{N}{M}
\leq\frac{(2d+1)n}{(1-\lambda)n/6}
\leq\frac{6d(2d+1)}{d-2\sqrt{d-1}}
=\frac{234}{3-\sqrt{5}}
\approx 306.31.
\end{equation*}
Thus, in such cases, we may perform phaseless recovery with only $N\leq 307M$ measurements.
However, the number of vertices in each Ramanujan graph from~\cite{LubotzkyPS:88} is of the form $q(q^2-1)$ or $\frac{q(q^2-1)}{2}$, where $q\equiv1\bmod4$ is prime, and so any bound on redundancy $\frac{N}{M}$ using graphs from~\cite{LubotzkyPS:88} will only be valid for particular values of $M$.

In order to get $N=\mathrm{O}(M)$ in general, we use the fact that random graphs are nearly Ramanujan with high probability.
In particular, for every $\varepsilon>0$ and even $d$, a random $d$-regular graph has spectral expansion $\lambda\leq\frac{1}{d}(2\sqrt{d-1}+\varepsilon)$ with high probability as $n\rightarrow\infty$~\cite{Friedman:03}.
Thus, picking $\varepsilon$ and $d$ to satisfy $\frac{1}{d}(2\sqrt{d-1}+\varepsilon)<1$, we may again take $M=\lfloor(\frac{1-\lambda}{6})n+1\rfloor$ to get
\begin{equation*}
\frac{N}{M}
\leq\frac{6(2d+1)}{1-\lambda}
\leq\frac{6d(2d+1)}{d-(2\sqrt{d-1}+\varepsilon)}
\end{equation*}
with high probability.
Note that in this case, $n$ can be any sufficiently large integer, and so the above bound is valid for all sufficiently large $M$, i.e., our procedure can perform phaseless recovery with $N=\mathrm{O}(M)$ measurements in general.

Note that this section has only considered the case in which the phaseless measurements were not corrupted by noise.
For the noisy case, Cand\`{e}s et al.~\cite{CandesSV:arxiv11} used semidefinite programming to stably reconstruct from $N=\mathrm{O}(M \log M)$ measurements.
Our technique also appears to be stable, and we expect positive results in this vein using synchronization-type analysis~\cite{Singer:11}; we leave this for future work.

\chapter{Deterministic matrices with the restricted isometry property}

In Chapter~1, we observed how to use the Gershgorin circle theorem to demonstrate that certain $M\times N$ matrices have the restricted isometry property (RIP) for sparsity levels $K=\mathrm{O}(\sqrt{M})$.
In this chapter, we consider better demonstration techniques which promise to break this ``square-root bottleneck''~\cite{BandeiraFMW:12}.
To date, the only deterministic construction that manages to go beyond the bottleneck is given by Bourgain et al.~\cite{BourgainDFKK:11}; in the following section, we discuss what they call \emph{flat RIP}, which is the technique they use to demonstrate RIP.
We will see that their technique can be used to demonstrate RIP for sparsity levels much larger than $\sqrt{M}$, meaning one could very well demonstrate random-like performance given the proper construction.
Later, we introduce an alternate technique, which can also demonstrate RIP for large sparsity levels.

After considering the efficacy of these techniques to demonstrate RIP, it remains to find a deterministic construction that is amenable to analysis.
To this end, we discuss various properties of certain equiangular tight frames (ETFs).
Specifically, real ETFs can be characterized in terms of their Gram matrices using strongly regular graphs~\cite{Waldron:09}.
By applying our demonstration techniques to real ETFs, we derive equivalent combinatorial statements in graph theory.
By focussing on the ETFs which correspond to Paley graphs of prime order, we are able to make important statements about their clique numbers and provide some intuition for an open problem in number theory.
We conclude by conjecturing that the Paley ETFs are RIP in a manner similar to random matrices.

\section{Flat restricted orthogonality}

In~\cite{BourgainDFKK:11}, Bourgain et al.~provided a deterministic construction of $M\times N$ RIP matrices that support sparsity levels $K$ on the order of $M^{1/2+\varepsilon}$ for some small value of $\varepsilon$.
To date, this is the only known deterministic RIP construction that breaks the square-root bottleneck.
In this section, we analyze their technique for demonstrating RIP, but first, we provide some historical context.
We begin with a definition:

\begin{defn}
\label{defn.ro}
The matrix $\Phi$ has \emph{$(K,\theta)$-restricted orthogonality (RO)} if 
\begin{equation*}
|\langle \Phi x, \Phi y\rangle|
\leq\theta\|x\|\|y\|
\end{equation*}
for every pair of $K$-sparse vectors $x,y$ with disjoint support.
The smallest $\theta$ for which $\Phi$ has $(K,\theta)$-RO is the \emph{restricted orthogonality constant (ROC)} $\theta_K$.
\end{defn}

In the past, restricted orthogonality was studied to produce reconstruction performance guarantees for both $\ell_1$-minimization and the Dantzig selector~\cite{CandesT:05,CandesT:07}.
Intuitively, restricted orthogonality is important to compressed sensing because any stable inversion process for \eqref{eq.cs eqn} would require $\Phi$ to map vectors of disjoint support to particularly dissimilar measurements.
For the present chapter, we are interested in upper bounds on RICs; in this spirit, the following result illustrates some sort of equivalence between RICs and ROCs:

\begin{lem}[Lemma 1.2 in \cite{CandesT:05}]
$\theta_K
\leq \delta_{2K}
\leq \theta_K+\delta_K$.
\end{lem}

To be fair, the above upper bound on $\delta_{2K}$ does not immediately help in estimating $\delta_{2K}$, as it requires one to estimate $\delta_K$.
Certainly, we may iteratively apply this bound to get
\begin{equation}
\label{eq.current bound}
\delta_{2K}
\leq\theta_K+\theta_{\lceil K/2\rceil}+\theta_{\lceil K/4\rceil}+\cdots+\theta_1+\delta_1
\leq(1+\lceil\log_2 K\rceil)\theta_K+\delta_1.
\end{equation}
Note that $\delta_1$ is particularly easy to calculate:
\begin{equation*}
\delta_1=\max_{n\in\{1,\ldots,N\}}\Big|\|\varphi_n\|^2-1\Big|,
\end{equation*}
which is zero when the columns of $\Phi$ have unit norm.
In pursuit of a better upper bound on~$\delta_{2K}$, we use techniques from \cite{BourgainDFKK:11} to remove the log factor from \eqref{eq.current bound}: 

\begin{lem}
\label{lem.ro to ri}
$\delta_{2K}
\leq 2\theta_K+\delta_1$.
\end{lem}

\begin{proof}
Given a matrix $\Phi=[\varphi_1\cdots\varphi_N]$, we want to upper-bound the smallest $\delta$ for which $(1-\delta)\|x\|^2\leq\|\Phi x\|^2\leq(1+\delta)\|x\|^2$, or equivalently:
\begin{equation}
\label{eq.smallest delta}
\delta\geq\Big|\|\Phi\tfrac{x}{\|x\|}\|^2-1\Big|
\end{equation}
for every nonzero $2K$-sparse vector $x$.
We observe from \eqref{eq.smallest delta} that we may take $x$ to have unit norm without loss of generality.
Letting $\mathcal{K}$ denote a size-$2K$ set that contains the support of $x$, and letting $\{x_k\}_{k\in\mathcal{K}}$ denote the corresponding entries of $x$, the triangle inequality gives
\begin{align}
\nonumber
\Big|\|\Phi x\|^2-1\Big|
&=\bigg|\bigg\langle\sum_{i\in\mathcal{K}}x_i\varphi_i,\sum_{j\in\mathcal{K}}x_j\varphi_j\bigg\rangle-1\bigg|\\
\nonumber
&=\bigg|\sum_{i\in\mathcal{K}}\sum_{\substack{j\in\mathcal{K}\\j\neq i}}\langle x_i\varphi_i,x_j\varphi_j\rangle+\sum_{i\in\mathcal{K}}\|x_i\varphi_i\|^2-1\bigg|\\
\label{eq.smallest delta 2}
&\leq\bigg|\sum_{i\in\mathcal{K}}\sum_{\substack{j\in\mathcal{K}\\j\neq i}}\langle x_i\varphi_i,x_j\varphi_j\rangle\bigg|+\bigg|\sum_{i\in\mathcal{K}}\|x_i\varphi_i\|^2-1\bigg|.
\end{align}
Since $\sum_{i\in\mathcal{K}}|x_i|^2=1$, the second term of \eqref{eq.smallest delta 2} satisfies
\begin{equation}
\label{eq.smallest delta 3}
\bigg|\sum_{i\in\mathcal{K}}\|x_i\varphi_i\|^2-1\bigg|
\leq\sum_{i\in\mathcal{K}}|x_i|^2\Big|\|\varphi_i\|^2-1\Big|
\leq\sum_{i\in\mathcal{K}}|x_i|^2\delta_1
=\delta_1,
\end{equation}
and so it remains to bound the first term of \eqref{eq.smallest delta 2}.
To this end, we note that for each $i,j\in\mathcal{K}$ with $j\neq i$, the term $\langle x_i\varphi_i,x_j\varphi_j\rangle$ appears in
\begin{equation*}
\sum_{\substack{\mathcal{I}\subseteq\mathcal{K}\\|\mathcal{I}|=K}}\sum_{i\in\mathcal{I}}\sum_{j\in\mathcal{K}\setminus\mathcal{I}}\langle x_i\varphi_i,x_j\varphi_j\rangle
\end{equation*}
as many times as there are size-$K$ subsets of $\mathcal{K}$ which contain $i$ but not $j$, i.e., $\binom{2K-2}{K-1}$ times.
Thus, we use the triangle inequality and the definition of restricted orthogonality to get
\begin{align*}
\bigg|\sum_{i\in\mathcal{K}}\sum_{\substack{j\in\mathcal{K}\\j\neq i}}\langle x_i\varphi_i,x_j\varphi_j\rangle\bigg|
&=\bigg|\frac{1}{\binom{2K-2}{K-1}}\sum_{\substack{\mathcal{I}\subseteq\mathcal{K}\\|\mathcal{I}|=K}}\sum_{i\in\mathcal{I}}\sum_{j\in\mathcal{K}\setminus\mathcal{I}}\langle x_i\varphi_i,x_j\varphi_j\rangle\bigg|\\
&\leq\frac{1}{\binom{2K-2}{K-1}}\sum_{\substack{\mathcal{I}\subseteq\mathcal{K}\\|\mathcal{I}|=K}}\bigg|\bigg\langle \sum_{i\in\mathcal{I}}x_i\varphi_i,\sum_{j\in\mathcal{K}\setminus\mathcal{I}}x_j\varphi_j\bigg\rangle\bigg|\\
&\leq\frac{1}{\binom{2K-2}{K-1}}\sum_{\substack{\mathcal{I}\subseteq\mathcal{K}\\|\mathcal{I}|=K}}
\theta_K\bigg(\sum_{i\in\mathcal{I}}|x_i|^2\bigg)^{1/2}\bigg(\sum_{j\in\mathcal{K}\setminus\mathcal{I}}|x_j|^2\bigg)^{1/2}.
\end{align*}
At this point, $x$ having unit norm implies $(\sum_{i\in\mathcal{I}}|x_i|^2)^{1/2}(\sum_{j\in\mathcal{K}\setminus\mathcal{I}}|x_j|^2)^{1/2}\leq\frac{1}{2}$, and so
\begin{equation*}
\bigg|\sum_{i\in\mathcal{K}}\sum_{\substack{j\in\mathcal{K}\\j\neq i}}\langle x_i\varphi_i,x_j\varphi_j\rangle\bigg|
\leq\frac{1}{\binom{2K-2}{K-1}}\sum_{\substack{\mathcal{I}\subseteq\mathcal{K}\\|\mathcal{I}|=K}}
\frac{\theta_K}{2}
=\frac{\binom{2K}{K}}{\binom{2K-2}{K-1}}\frac{\theta_K}{2}
=\bigg(4-\frac{2}{K}\bigg)\frac{\theta_K}{2}.
\end{equation*}
Applying both this and \eqref{eq.smallest delta 3} to \eqref{eq.smallest delta 2} gives the result.
\end{proof}

Having discussed the relationship between restricted isometry and restricted orthogonality, we are now ready to introduce the property used in \cite{BourgainDFKK:11} to demonstrate RIP:

\begin{defn}
The matrix $\Phi=[\varphi_1\cdots\varphi_N]$ has \emph{$(K,\hat\theta)$-flat restricted orthogonality} if 
\begin{equation*}
\bigg|\bigg\langle \sum_{i\in\mathcal{I}}\varphi_i,\sum_{j\in\mathcal{J}}\varphi_j \bigg\rangle\bigg|
\leq\hat\theta(|\mathcal{I}||\mathcal{J}|)^{1/2}
\end{equation*}
for every disjoint pair of subsets $\mathcal{I},\mathcal{J}\subseteq\{1,\ldots,N\}$ with $|\mathcal{I}|,|\mathcal{J}|\leq K$.
\end{defn}

Note that $\Phi$ has $(K,\theta_K)$-flat restricted orthogonality (FRO) by taking $x$ and $y$ in Definition~\ref{defn.ro} to be the characteristic functions $\chi_\mathcal{I}$ and $\chi_\mathcal{J}$, respectively.
Also to be clear, \emph{flat restricted orthogonality} is called \emph{flat RIP} in~\cite{BourgainDFKK:11}; we feel the name change is appropriate considering the preceeding literature.
Moreover, the definition of flat RIP in \cite{BourgainDFKK:11} required $\Phi$ to have unit-norm columns, whereas we strengthen the corresponding results so as to make no such requirement.
Interestingly, FRO bears some resemblence to the cut-norm of the Gram matrix $\Phi^*\Phi$, defined as the maximum value of $|\sum_{i\in\mathcal{I}}\sum_{j\in\mathcal{J}}\langle\varphi_i,\varphi_j\rangle|$ over \emph{all} subsets $\mathcal{I},\mathcal{J}\subseteq\{1,\ldots,N\}$; the cut-norm has received some attention recently for the hardness of its approximation~\cite{AlonN:06}.
The following theorem illustrates the utility of flat restricted orthogonality as an estimate of the RIC:

\begin{thm}
\label{thm.fro}
A matrix with $(K,\hat\theta)$-flat restricted orthogonality has a restricted orthogonality constant $\theta_K$ which is $\leq C\hat\theta\log K$, and we may take $C=75$.
\end{thm}

Indeed, when combined with Lemma~\ref{lem.ro to ri}, this result gives an upper bound on the RIC: $\delta_{2K}\leq 2C\hat\theta\log K + \delta_1$.
The noteworthy benefit of this upper bound is that the problem of estimating singular values of submatrices is reduced to a combinatorial problem of bounding the coherence of disjoint sums of columns.
Furthermore, this reduction comes at the price of a mere log factor in the estimate.
In~\cite{BourgainDFKK:11}, Bourgain et al.~managed to satisfy this combinatorial coherence property using techniques from additive combinatorics.
While we will not discuss their construction, we find the proof of Theorem~\ref{thm.fro} to be instructive; our proof is valid for all values of $K$ (as opposed to sufficiently large $K$ in the original~\cite{BourgainDFKK:11}), and it has near-optimal constants where appropriate.
The proof can be found in the Appendix.

To reiterate, Bourgain et al.~\cite{BourgainDFKK:11} used flat restricted orthogonality to build the only known deterministic construction of $M\times N$ RIP matrices that support sparsity levels $K$ on the order of $M^{1/2+\varepsilon}$ for some small value of $\varepsilon$.
We are particularly interested in the efficacy of FRO as a technique to demonstrate RIP in general.
Certainly, \cite{BourgainDFKK:11} shows that FRO can produce at least an $\varepsilon$ improvement over the Gershgorin technique discussed in the previous section, but it remains to be seen whether FRO can do better.

In the remainder of this section, we will show that flat restricted orthogonality is actually capable of demonstrating RIP with much higher sparsity levels than indicated by~\cite{BourgainDFKK:11}.
Hopefully, this realization will spur further research in deterministic constructions which satisfy FRO.
To evaluate FRO, we investigate how well it performs with random matrices; in doing so, we give an alternative proof that certain random matrices satisfy RIP with high probability:

\begin{thm}
\label{thm.fro to rip}
Construct an $M\times N$ matrix $\Phi$ by drawing each of its entries independently from a Gaussian distribution with mean zero and variance $\frac{1}{M}$, take $C$ to be the constant from Theorem~\ref{thm.fro}, and set $\alpha=0.01$.
Then $\Phi$ has $(K,\frac{(1-\alpha)\delta}{2C\log K})$-flat restricted orthogonality and $\delta_1\leq \alpha\delta$, and therefore the $(2K,\delta)$-restricted isometry property, with high probability provided $M\geq\frac{33C^2}{\delta^2}K\log^2 K\log N$.
\end{thm}

In proving this result, we will make use of the following Bernstein inequality:
\begin{thm}[see \cite{Bernstein:46,Yurinskii:76}]
\label{thm.bernstein}
Let $\{Z_m\}_{m=1}^M$ be independent random variables of mean zero with bounded moments, and suppose there exists $L>0$ such that
\begin{equation}
\label{eq.bernstein requirement}
\mathbb{E}|Z_m|^k
\leq\frac{\mathbb{E}|Z_m|^2}{2}L^{k-2}k!
\end{equation}
for every $k\geq2$.
Then 
\begin{equation}
\label{eq.bernstein conclusion}
\mathrm{Pr}\bigg[\sum_{m=1}^M Z_m\geq2t\bigg(\sum_{m=1}^M\mathbb{E}|Z_m|^2\bigg)^{1/2}\bigg]
\leq e^{-t^2}
\end{equation}
provided $\displaystyle{t\leq\frac{1}{2L}\bigg(\sum_{m=1}^M\mathbb{E}|Z_m|^2\bigg)^{1/2}}$.
\end{thm}

\begin{proof}[Proof of Theorem~\ref{thm.fro to rip}]
Considering Lemma~\ref{lem.ro to ri}, it suffices to show that $\Phi$ has restricted orthogonality and that $\delta_1$ is sufficiently small.
First, to demonstrate restricted orthogonality, it suffices to demonstrate FRO by Theorem~\ref{thm.fro}, and so we will ensure that the following quantity is small:
\begin{equation}
\label{eq.we want small}
\bigg\langle\sum_{i\in\mathcal{I}}\varphi_i,\sum_{j\in\mathcal{J}}\varphi_j\bigg\rangle
=\sum_{m=1}^M\bigg(\sum_{i\in\mathcal{I}}\varphi_i[m]\bigg)\bigg(\sum_{j\in\mathcal{J}}\varphi_j[m]\bigg).
\end{equation}
Notice that $X_m:=\sum_{i\in\mathcal{I}}\varphi_i[m]$ and $Y_m:=\sum_{j\in\mathcal{J}}\varphi_j[m]$ are mutually independent over all $m=1,\ldots,M$ since $\mathcal{I}$ and $\mathcal{J}$ are disjoint.
Also, $X_m$ is Gaussian with mean zero and variance $\frac{|\mathcal{I}|}{M}$, while $Y_m$ similarly has mean zero and variance $\frac{|\mathcal{J}|}{M}$.
Viewed this way, \eqref{eq.we want small} being small corresponds to the sum of independent random variables $Z_m:=X_mY_m$ having its probability measure concentrated at zero.
To this end, Theorem~\ref{thm.bernstein} is naturally applicable, as the absolute central moments of a Gaussian random variable $X$ with mean zero and variance $\sigma^2$ are well known:
\begin{equation*}
\mathbb{E}|X|^k
=\left\{\begin{array}{rl}\sqrt{\frac{2}{\pi}}\sigma^k(k-1)!!&\mbox{ if $k$ odd},\\\sigma^k(k-1)!!&\mbox{ if $k$ even}.\end{array}\right.
\end{equation*}
Since $Z_m=X_mY_m$ is a product of independent Gaussian random variables, this gives
\begin{equation*}
\mathbb{E}|Z_m|^k
=\mathbb{E}|X_m|^k~\mathbb{E}|Y_m|^k
\leq \Big(\frac{|\mathcal{I}|}{M}\Big)^{k/2}\Big(\frac{|\mathcal{J}|}{M}\Big)^{k/2}\Big((k-1)!!\Big)^2
\leq \bigg(\frac{(|\mathcal{I}||\mathcal{J}|)^{1/2}}{M}\bigg)^kk!.
\end{equation*}
Further since $\mathbb{E}|Z_m|^2=\frac{|\mathcal{I}||\mathcal{J}|}{M^2}$, we may define $L:=2\frac{(|\mathcal{I}||\mathcal{J}|)^{1/2}}{M}$ to get \eqref{eq.bernstein requirement}.
Later, we will take $\hat\theta<\delta<\sqrt{2}-1<\frac{1}{2}$.
Considering
\begin{equation*}
t
:=\frac{\hat{\theta}\sqrt{M}}{2}
<\frac{\sqrt{M}}{4}
=\frac{1}{2L}\Big(M\frac{|\mathcal{I}||\mathcal{J}|}{M^2}\Big)^{1/2}
=\frac{1}{2L}\bigg(\sum_{m=1}^M\mathbb{E}|Z_m|^2\bigg)^{1/2},
\end{equation*}
we therefore have \eqref{eq.bernstein conclusion}, which in this case has the form
\begin{equation*}
\mathrm{Pr}\Bigg[\bigg|\bigg\langle\sum_{i\in\mathcal{I}}\varphi_i,\sum_{j\in\mathcal{J}}\varphi_j\bigg\rangle\bigg|\geq\hat{\theta}(|\mathcal{I}||\mathcal{J}|)^{1/2}\Bigg]
\leq 2e^{-M\hat{\theta}^2/4},
\end{equation*}
where the probability is doubled due to the symmetric distribution of $\sum_{m=1}^M Z_m$.
Since we need to account for all possible choices of $\mathcal{I}$ and $\mathcal{J}$, we will perform a union bound.
The total number of choices is given by
\begin{equation*}
\sum_{|\mathcal{I}|=1}^K\sum_{|\mathcal{J}|=1}^K\binom{N}{|\mathcal{I}|}\binom{N-|\mathcal{I}|}{|\mathcal{J}|}
\leq K^2\binom{N}{K}^2
\leq N^{2K},
\end{equation*}
and so the union bound gives
\begin{equation}
\label{eq.probability bound 1}
\mathrm{Pr}\Big[\mbox{$\Phi$ does not have $(K,\hat{\theta})$-FRO}\Big]
\leq 2e^{-M\hat{\theta}^2/4}~N^{2K}
=2\exp\Big(-\frac{M\hat{\theta}^2}{4}+2K\log N\Big).
\end{equation}
Thus, Gaussian matrices tend to have FRO, and hence restricted orthogonality by Theorem~\ref{thm.fro}; this is made more precise below.

Again by Lemma~\ref{lem.ro to ri}, it remains to show that $\delta_1$ is sufficiently small.
To this end, we note that $M\|\varphi_n\|^2$ has chi-squared distribution with $M$ degrees of freedom, and so we can use another (simpler) concentration-of-measure result; see Lemma 1 of~\cite{LaurentM:00}:
\begin{equation*}
\mathrm{Pr}\bigg[\Big|\|\varphi_n\|^2-1\Big|\geq 2\Big(\sqrt{\frac{t}{M}}+\frac{t}{M}\Big)\bigg]\leq 2e^{-t}
\end{equation*}
for any $t>0$.
Specifically, we pick
\begin{equation*}
\delta'
:=2\Big(\sqrt{\frac{t}{M}}+\frac{t}{M}\Big)
\leq\frac{4t}{M},
\end{equation*}
and we perform a union bound over the $N$ choices for $\varphi_n$:
\begin{equation}
\label{eq.probability bound 2}
\mathrm{Pr}\Big[\delta_1>\delta'\Big]
\leq 2\exp\Big(-\frac{M\delta'}{4}+\log N\Big).
\end{equation}
To summarize, Lemma~\ref{lem.ro to ri}, the union bound, Theorem~\ref{thm.fro}, and \eqref{eq.probability bound 1} and \eqref{eq.probability bound 2} give
\begin{align*}
\mathrm{Pr}\Big[\delta_{2K}>\delta\Big]
&\leq\mathrm{Pr}\Big[\theta_K>\frac{(1-\alpha)\delta}{2}\mbox{ or }\delta_1>\alpha\delta\Big]\\
&\leq\mathrm{Pr}\Big[\theta_K>\frac{(1-\alpha)\delta}{2}\Big]+\mathrm{Pr}\Big[\delta_1>\alpha\delta\Big]\\
&\leq\mathrm{Pr}\Big[\mbox{$\Phi$ does not have $\displaystyle{\Big(K,\frac{(1-\alpha)\delta}{2C\log K}\Big)}$-FRO}\Big]+\mathrm{Pr}\Big[\delta_1>\alpha\delta\Big]\\
&\leq2\exp\Big(-\frac{M}{4}\Big(\frac{(1-\alpha)\delta}{2C\log K}\Big)^2+2K\log N\Big)+2\exp\Big(-\frac{M\alpha\delta}{4}+\log N\Big),
\end{align*}
and so $M\geq\frac{33C^2}{\delta^2}K\log^2 K\log N$ gives that $\Phi$ has $(2K,\delta)$-RIP with high probability.
\end{proof}

We note that a version of Theorem~\ref{thm.fro to rip} also holds for matrices whose entries are independent Bernoulli random variables taking values $\pm\frac{1}{\sqrt{M}}$ with equal probability.
In this case, one can again apply Theorem~\ref{thm.bernstein} by comparing moments with those of the Gaussian distribution; also, a union bound with $\delta_1$ will not be necessary since the columns have unit norm, meaning $\delta_1=0$.

\section{Restricted isometry by the power method}

In the previous section, we established the efficacy of flat restricted orthogonality as a technique to demonstrate RIP.
While flat restricted orthogonality has proven useful in the past~\cite{BourgainDFKK:11}, future deterministic RIP constructions might not use this technique.
Indeed, it would be helpful to have other techniques available that demonstrate RIP beyond the square-root bottleneck.
In pursuit of such techniques, we recall that the smallest $\delta$ for which $\Phi$ is $(K,\delta)$-RIP is given in terms of operator norms in \eqref{eq.delta min}.
In addition, we notice that for any self-adjoint matrix $A$,
\[ \|A\|_2=\|\lambda(A)\|_\infty\leq\|\lambda(A)\|_p, \]
where $\lambda(A)$ denotes the spectrum of $A$ with multiplicities.
Let $A=UDU^*$ be the eigenvalue decomposition of~$A$.
When $p$ is even, we can express $\|\lambda(A)\|_p$ in terms of an easy-to-calculate trace:
\[ \|\lambda(A)\|_{p}^{p}=\mathrm{Tr}[D^{p}]=\mathrm{Tr}[(UDU^*)^{p}]=\mathrm{Tr}[A^{p}]. \]
Combining these ideas with the fact that $\|\cdot\|_p\rightarrow\|\cdot\|_\infty$ pointwise leads to the following:

\begin{thm}
\label{thm.power method defn}
Given an $M\times N$ matrix $\Phi$, define
\[ \delta_{K;q}:=\max_{\substack{\mathcal{K}\subseteq\{1,\ldots,N\}\\|\mathcal{K}|=K}}\mathrm{Tr}[(\Phi_\mathcal{K}^*\Phi_\mathcal{K}^{}-I_K)^{2q}]^\frac{1}{2q}. \]
Then $\Phi$ has the $(K,\delta_{K;q})$-restricted isometry property for every $q\geq1$.  Moreover, the restricted isometry constant of $\Phi$ is approached by these estimates: $\lim_{q\rightarrow\infty}\delta_{K;q}=\delta_K$.
\end{thm}

Similar to flat restricted orthogonality, this \emph{power method} has a combinatorial aspect that prompts one to check every sub-Gram matrix of size $K$; one could argue that the power method is slightly \emph{less} combinatorial, as flat restricted orthogonality is a statement about all pairs of disjoint subsets of size $\leq K$.
Regardless, the work of Bourgain et al.~\cite{BourgainDFKK:11} illustrates that combinatorial properties can be useful, and there may exist constructions to which the power method would be naturally applied.
Moreover, we note that since $\delta_{K;q}$ approaches $\delta_K$, a sufficiently large choice of $q$ should deliver better-than-$\varepsilon$ improvement over the Gershgorin analysis.
How large should $q$ be?
If we assume $\Phi$ has unit-norm columns, taking $q=1$ gives
\begin{equation}
\label{eq.delta 1}
\delta_{K;1}^2
=\max_{\substack{\mathcal{K}\subseteq\{1,\ldots,N\}\\|\mathcal{K}|=K}}\mathrm{Tr}[(\Phi_\mathcal{K}^*\Phi_\mathcal{K}^{}-I_K)^{2}]
=\max_{\substack{\mathcal{K}\subseteq\{1,\ldots,N\}\\|\mathcal{K}|=K}}
\sum_{i\in\mathcal{K}}\sum_{\substack{j\in\mathcal{K}\\ j\neq i}}|\langle \varphi_i,\varphi_j\rangle|^2
\leq K(K-1)\mu^2,
\end{equation}
where $\mu$ is the worst-case coherence of $\Phi$.
Equality is achieved above whenever $\Phi$ is an ETF, in which case \eqref{eq.delta 1} along with reasoning similar to \eqref{eq.square root bottleneck} demonstrates that $\Phi$ is RIP with sparsity levels on the order of $\sqrt{M}$, as the Gershgorin analysis established.
It remains to be shown how $\delta_{K;2}$ compares.
To make this comparison, we apply the power method to random matrices:

\begin{thm}
\label{thm.power to rip}
Construct an $M\times N$ matrix $\Phi$ by drawing each of its entries independently from a Gaussian distribution with mean zero and variance $\frac{1}{M}$, and take $\delta_{K;q}$ to be as defined in Theorem~\ref{thm.power method defn}.
Then $\delta_{K;q}\leq\delta$, and therefore $\Phi$ has the $(K,\delta)$-restricted isometry property, with high probability provided $M\geq\frac{81}{\delta^2}K^{1+1/q}\log\frac{eN}{K}$.
\end{thm}

While flat restricted orthogonality comes with a negligible penalty of $\log^2K$ in the number of measurements, the power method has a penalty of $K^{1/q}$.
As such, the case $q=1$ uses the order of $K^2$ measurements, which matches our calculation in \eqref{eq.delta 1}.
Moreover, the power method with $q=2$ can demonstrate RIP with $K^{3/2}$ measurements, i.e., $K\sim M^{1/2+1/6}$, which is considerably better than an $\varepsilon$ improvement over the Gershgorin technique.

\begin{proof}[Proof of Theorem~\ref{thm.power to rip}]
Take $t:=\frac{\delta}{3K^{1/2q}}-(\frac{K}{M})^{1/2}$ and pick $\mathcal{K}\subseteq\{1,\ldots,N\}$.
Then Theorem~II.13 of~\cite{DavidsonS:01} states
\begin{equation*}
\mathrm{Pr}\bigg[1-\bigg(\sqrt{\frac{K}{M}}+t\bigg)\leq\sigma_{\min}(\Phi_\mathcal{K})\leq\sigma_{\max}(\Phi_\mathcal{K})\leq1+\bigg(\sqrt{\frac{K}{M}}+t\bigg)\bigg]
\geq 1-2e^{-Mt^2/2}.
\end{equation*}
Continuing, we use the fact that $\lambda(\Phi_\mathcal{K}^*\Phi_\mathcal{K}^{})=\sigma(\Phi_\mathcal{K})^2$ to get
\begin{align}
\nonumber
&1-2e^{-Mt^2/2}\\
\nonumber
&\leq \mathrm{Pr}\bigg[\bigg(1-\bigg(\sqrt{\frac{K}{M}}+t\bigg)\bigg)^2\leq\lambda_{\min}(\Phi_\mathcal{K}^*\Phi_\mathcal{K}^{})\leq\lambda_{\max}(\Phi_\mathcal{K}^*\Phi_\mathcal{K}^{})\leq\bigg(1+\bigg(\sqrt{\frac{K}{M}}+t\bigg)\bigg)^2\bigg]\\
\label{eq.eigenvalue interval}
&\leq \mathrm{Pr}\bigg[1-3\bigg(\sqrt{\frac{K}{M}}+t\bigg)\leq\lambda_{\min}(\Phi_\mathcal{K}^*\Phi_\mathcal{K}^{})\leq\lambda_{\max}(\Phi_\mathcal{K}^*\Phi_\mathcal{K}^{})\leq1+3\bigg(\sqrt{\frac{K}{M}}+t\bigg)\bigg],
\end{align}
where the last inequality follows from the fact that $(\frac{K}{M})^{1/2}+t<1$.
Since $\Phi_\mathcal{K}^*\Phi_\mathcal{K}^{}$ and $I_K$ are simultaneously diagonalizable, the spectrum of $\Phi_\mathcal{K}^*\Phi_\mathcal{K}^{}-I_K$ is given by $\lambda(\Phi_\mathcal{K}^*\Phi_\mathcal{K}^{}-I_K)=\lambda(\Phi_\mathcal{K}^*\Phi_\mathcal{K}^{})-1$.
Combining this with \eqref{eq.eigenvalue interval} then gives
\begin{equation*}
\mathrm{Pr}\bigg[\Big\|\lambda(\Phi_\mathcal{K}^*\Phi_\mathcal{K}^{}-I_K)\Big\|_\infty\leq 3\bigg(\sqrt{\frac{K}{M}}+t\bigg)\bigg]
\geq 1-2e^{-Mt^2/2}.
\end{equation*}
Considering $\mathrm{Tr}[A^{2q}]^\frac{1}{2q}=\|\lambda(A)\|_{2q}\leq K^\frac{1}{2q}\|\lambda(A)\|_\infty$, we continue:
\begin{equation*}
\mathrm{Pr}\bigg[\mathrm{Tr}[(\Phi_\mathcal{K}^*\Phi_\mathcal{K}^{}-I_K)^{2q}]^\frac{1}{2q}\leq \delta \bigg]
\geq \mathrm{Pr}\bigg[ K^\frac{1}{2q}\Big\|\lambda(\Phi_\mathcal{K}^*\Phi_\mathcal{K}^{}-I_K)\Big\|_\infty \leq \delta \bigg]
\geq 1-2e^{-Mt^2/2}.
\end{equation*}
From here, we perform a union bound over all possible choices of $\mathcal{K}$:
\begin{align}
\nonumber
\mathrm{Pr}\bigg[\exists\mathcal{K}\mbox{ s.t. }\mathrm{Tr}[(\Phi_\mathcal{K}^*\Phi_\mathcal{K}^{}-I_K)^{2q}]^\frac{1}{2q}> \delta\bigg]
&\leq\binom{N}{K}\mathrm{Pr}\bigg[\mathrm{Tr}[(\Phi_\mathcal{K}^*\Phi_\mathcal{K}^{}-I_K)^{2q}]^\frac{1}{2q}> \delta \bigg]\\
\label{eq.whp 1}
&\leq2\exp\Big(-\frac{Mt^2}{2}+K\log \frac{eN}{K}\Big).
\end{align}
Rearranging $M\geq\frac{81}{\delta^2}K^{1+1/q}\log\frac{eN}{K}$ gives $K^{1/2}\leq\frac{\delta M^{1/2}}{9K^{1/2q}\log^{1/2}(eN/K)}\leq\frac{\delta M^{1/2}}{9K^{1/2q}}$, and so
\begin{equation}
\label{eq.whp 2}
\frac{Mt^2}{2}
=\frac{1}{2}\bigg(\frac{\delta M^{1/2}}{3K^{1/2q}}-K^{1/2}\bigg)^2
\geq\frac{1}{2}\bigg(\frac{2\delta M^{1/2}}{9K^{1/2q}}\bigg)^2
\geq 2K\log\frac{eN}{K}.
\end{equation}
Combining \eqref{eq.whp 1} and \eqref{eq.whp 2} gives the result.
\end{proof}

\section{Equiangular tight frames as RIP candidates}

In Chapter~1, we observed that equiangular tight frames (ETFs) are optimal RIP matrices under the Gershgorin analysis.
In the present section, we reexamine ETFs as prospective RIP matrices.
Specifically, we consider the possibility that certain classes of $M\times N$ ETFs support sparsity levels $K$ larger than the order of $\sqrt{M}$.
Before analyzing RIP, let's first observe some important features of ETFs.
Recall that Section~\ref{section.intro to frame theory} characterized ETFs in terms of their rows and columns.
Interestingly, \emph{real} ETFs have a natural alternative characterization.

Let $\Phi$ be a real $M\times N$ ETF, and consider the corresponding Gram matrix $\Phi^*\Phi$.
Observing Section~\ref{section.intro to frame theory}, we have from (ii) that the diagonal entries of $\Phi^*\Phi$ are 1's.
Also, (iii) indicates that the off-diagonal entries are equal in absolute value (to the Welch bound); since $\Phi$ has real entries, the phase of each off-diagonal entry of $\Phi^*\Phi$ is either positive or negative.
Letting $\mu$ denote the absolute value of the off-diagonal entries, we can decompose the Gram matrix as $\Phi^*\Phi=I_N+\mu S$, where $S$ is a matrix of zeros on the diagonal and $\pm1$'s on the off-diagonal.
Here, $S$ is referred to as a \emph{Seidel adjacency matrix}, as $S$ encodes the adjacency rule of a simple graph with $i\leftrightarrow j$ whenever $S[i,j]=-1$; this correspondence originated in \cite{vanLintS:66}.

There is an important equivalence class amongst ETFs: given an ETF $\Phi$, one can negate any of the columns to form another ETF $\Phi'$.
Indeed, the ETF properties in Section~\ref{section.intro to frame theory} are easily verified to hold for this new matrix.
For obvious reasons, $\Phi$ and $\Phi'$ are called \emph{flipping equivalent}.
This equivalence plays a key role in the following result, which characterizes real ETFs in terms of a particular class of strongly regular graphs:

\begin{defn}
We say a simple graph $G$ is \emph{strongly regular} of the form $\mathrm{srg}(v,k,\lambda,\mu)$ if
\begin{itemize}
\item[(i)] $G$ has $v$ vertices,
\item[(ii)] every vertex has $k$ neighbors (i.e., $G$ is $k$\emph{-regular}),
\item[(iii)] every two adjacent vertices have $\lambda$ common neighbors, and
\item[(iv)] every two non-adjacent vertices have $\mu$ common neighbors.
\end{itemize}

\end{defn}

\begin{thm}[Corollary 5.6 in \cite{Waldron:09}]
\label{thm.etf to srg}
Every real $M\times N$ equiangular tight frame with $N>M+1$ is flipping equivalent to a frame whose Seidel adjacency matrix corresponds to the join of a vertex with a strongly regular graph of the form
\begin{equation*}
\mathrm{srg}\bigg(N-1,L,\frac{3L-N}{2},\frac{L}{2}\bigg),
\qquad L:=\frac{N}{2}-1+\bigg(1-\frac{N}{2M}\bigg)\sqrt{\frac{M(N-1)}{N-M}}.
\end{equation*}
Conversely, every such graph corresponds to flipping equivalence classes of equiangular tight frames in the same manner.
\end{thm}

The first chapter illustrated the main issue with the Gershgorin analysis: it ignores important cancellations in the sub-Gram matrices.
We suspect that such cancellations would be more easily observed in a real ETF, since Theorem~\ref{thm.etf to srg} neatly represents the Gram matrix's off-diagonal oscillations in terms of adjacencies in a strongly regular graph.
The following result gives a taste of how useful this graph representation can be:

\begin{thm}
\label{thm.etf clique}
Take a real equiangular tight frame $\Phi$ with worst-case coherence $\mu$, and let $G$ denote the corresponding strongly regular graph in Theorem~\ref{thm.etf to srg}.
Then the restricted isometry constant of $\Phi$ is given by $\delta_K=(K-1)\mu$ for every $K\leq\omega(G)+1$, where $\omega(G)$ denotes the size of the largest clique in $G$.
\end{thm}

\begin{proof}
The Gershgorin analysis \eqref{eq.bound} gives the bound $\delta_K\leq(K-1)\mu$, and so it suffices to prove $\delta_K\geq(K-1)\mu$.
Since $K\leq\omega(G)+1$, there exists a clique of size $K$ in the join of $G$ with a vertex.
Let $\mathcal{K}$ denote the vertices of this clique, and take $S_\mathcal{K}$ to be the corresponding Seidel adjacency submatrix. 
In this case, $S_\mathcal{K}=I_K-J_K$, where $J_K$ is the $K\times K$ matrix of all~1's.
Observing the decomposition $\Phi_\mathcal{K}^*\Phi_\mathcal{K}^{}=I_K+\mu S_\mathcal{K}$, it follows from \eqref{eq.delta min} that
\begin{equation*}
\delta_K
\geq\|\Phi_\mathcal{K}^*\Phi_\mathcal{K}^{}-I_K\|_2
=\|\mu S_\mathcal{K}\|_2
=\mu\|I_K-J_K\|_2
=(K-1)\mu,
\end{equation*}
which concludes the proof.
\end{proof}

This result indicates that the Gershgoin analysis is tight for all real ETFs, at least for sufficiently small values of $K$.
In particular, in order for a real ETF to be RIP beyond the square-root bottleneck, its graph must have a small clique number.
As an example, note that the first four columns of the Steiner ETF in \eqref{eq.steiner etf example} have negative inner products with each other, and thus the corresponding subgraph is a clique.
In general, each block of an $M\times N$ Steiner ETF, whose size is guaranteed to be $\mathrm{O}(\sqrt{M})$, is a lower-dimensional simplex and therefore has this property; this is an alternative proof that the Gershgorin analysis of Steiner ETFs is tight for $K=\mathrm{O}(\sqrt{M})$.

\subsection{Equiangular tight frames with flat restricted orthogonality}

To find ETFs that are RIP beyond the square-root bottleneck, we must apply better techniques than Gershgorin.
We first consider what it means for an ETF to have $(K,\hat\theta)$-flat restricted orthogonality.
Take a real ETF $\Phi=[\varphi_1\cdots\varphi_N]$ with worst-case coherence $\mu$, and note that the corresponding Seidel adjacency matrix $S$ can be expressed in terms of the usual $\{0,1\}$-adjacency matrix~$A$ of the same graph: $S[i,j]=1-2A[i,j]$ whenever $i\neq j$.
Therefore, for every disjoint $\mathcal{I},\mathcal{J}\subseteq\{1,\ldots,N\}$ with $|\mathcal{I}|,|\mathcal{J}|\leq K$, we want
\begin{align}
\nonumber
\hat\theta(|\mathcal{I}||\mathcal{J}|)^{1/2}
&\geq \bigg|\bigg\langle\sum_{i\in\mathcal{I}}\varphi_i,\sum_{j\in\mathcal{J}}\varphi_j\bigg\rangle\bigg|
= \bigg|\sum_{i\in\mathcal{I}}\sum_{j\in\mathcal{J}}\mu S[i,j]\bigg|\\
\label{eq.edge distribution}
&\qquad = \mu\bigg||\mathcal{I}||\mathcal{J}|-2\sum_{i\in\mathcal{I}}\sum_{j\in\mathcal{J}}A[i,j]\bigg|
= 2\mu\bigg|E(\mathcal{I},\mathcal{J})-\frac{1}{2}|\mathcal{I}||\mathcal{J}|\bigg|,
\end{align}
where $E(\mathcal{I},\mathcal{J})$ denotes the number of edges between $\mathcal{I}$ and $\mathcal{J}$ in the graph.
This condition bears a striking resemblence to the following well-known result in graph theory:

\begin{lem}[Expander mixing lemma \cite{HooryLW:06}]
Given a $d$-regular graph of $n$ vertices, the second largest eigenvalue $\lambda$ of its adjacency matrix satisfies
\begin{equation*}
\bigg|E(\mathcal{I},\mathcal{J})-\frac{d}{n}|\mathcal{I}||\mathcal{J}|\bigg|
\leq \lambda(|\mathcal{I}||\mathcal{J}|)^{1/2}
\end{equation*}
for every pair of vertex subsets $\mathcal{I}, \mathcal{J}$.
\end{lem}

In words, the expander mixing lemma says that the number of edges between vertex subsets of a regular graph is roughly what you would expect in a \emph{random} regular graph.
For this lemma to be applicable to \eqref{eq.edge distribution}, we need the strongly regular graph of Theorem~\ref{thm.etf to srg} to satisfy $\frac{L}{N-1}=\frac{d}{n}\approx\frac{1}{2}$.
Using the formula for $L$, it is not difficult to show that $|\frac{L}{N-1}-\frac{1}{2}|=\mathrm{O}(M^{-1/2})$ provided $N=\mathrm{O}(M)$ and $N\geq 2M$.
Furthermore, the second largest eigenvalue of the strongly regular graph will be $\lambda\approx\frac{1}{2}N^{1/2}$, and so the expander mixing lemma says the optimal $\hat\theta$ is $\leq 2\mu\lambda\approx(\frac{N-M}{M})^{1/2}$ since $\mu=(\frac{N-M}{M(N-1)})^{1/2}$.
This is a rather weak estimate for $\hat\theta$ because the expander mixing lemma does not account for the sizes of $\mathcal{I}$ and $\mathcal{J}$ being $\leq K$.
Put in this light, a real ETF that has flat restricted orthogonality corresponds to a strongly regular graph that satisfies a particularly strong version of the expander mixing lemma.

\subsection{Equiangular tight frames and the power method}
Next, we try applying the power method to ETFs.
Given a real ETF $\Phi=[\varphi_1\cdots\varphi_N]$, let $H:=\Phi^*\Phi-I_N$ denote the ``hollow'' Gram matrix.
Also, take $E_\mathcal{K}$ to be the $N\times K$ matrix built from the columns of $I_N$ that are indexed by $\mathcal{K}$.
Then
\begin{equation*}
\mathrm{Tr}[(\Phi_\mathcal{K}^*\Phi_\mathcal{K}^{}-I_K)^{2q}]
=\mathrm{Tr}[(E_\mathcal{K}^*\Phi^*\Phi E_\mathcal{K}^{}-I_K)^{2q}]
=\mathrm{Tr}[(E_\mathcal{K}^*H E_\mathcal{K}^{})^{2q}]
=\mathrm{Tr}[(H E_\mathcal{K}^{}E_\mathcal{K}^*)^{2q}].
\end{equation*}
Since $E_\mathcal{K}^{}E_\mathcal{K}^*=\sum_{k\in\mathcal{K}}\delta_k^{}\delta_k^*$, where $\delta_k$ is the $k$th identity basis element, we continue:
\begin{align}
\nonumber
\mathrm{Tr}[(\Phi_\mathcal{K}^*\Phi_\mathcal{K}^{}-I_K)^{2q}]
&=\mathrm{Tr}\bigg[\bigg(H \sum_{k\in\mathcal{K}}\delta_k^{}\delta_k^*\bigg)^{2q}\bigg]\\
\nonumber
&=\sum_{k_0\in\mathcal{K}}\cdots\sum_{k_{2q-1}\in\mathcal{K}}\mathrm{Tr}[H \delta_{k_0}^{}\delta_{k_0}^*\cdots H \delta_{k_{2q-1}}^{}\delta_{k_{2q-1}}^*]\\
\label{eq.zero terms}
&=\sum_{k_0\in\mathcal{K}}\cdots\sum_{k_{2q-1}\in\mathcal{K}}\delta_{k_0}^*H \delta_{k_{1}}^{}\cdots \delta_{k_{2q-1}}^*H \delta_{k_0}^{},
\end{align}
where the last step used the cyclic property of the trace.
From here, note that $H$ has a zero diagonal, meaning several of the terms in \eqref{eq.zero terms} are zero, namely, those for which $k_{\ell+1}=k_\ell$ for some $\ell\in\mathbb{Z}_{2q}$.
To simplify \eqref{eq.zero terms}, take $\mathcal{K}^{(2q)}$ to be the set of $2q$-tuples satisfying $k_{\ell+1}\neq k_\ell$ for every $\ell\in\mathbb{Z}_{2q}$:
\begin{equation}
\label{eq.power method combinatorics}
\mathrm{Tr}[(\Phi_\mathcal{K}^*\Phi_\mathcal{K}^{}-I_K)^{2q}]
=\sum_{\{k_\ell\}\in\mathcal{K}^{(2q)}} \prod_{\ell\in\mathbb{Z}_{2q}}\langle\varphi_{k_\ell},\varphi_{k_{\ell+1}}\rangle=\mu^{2q}\sum_{\{k_\ell\}\in\mathcal{K}^{(2q)}} \prod_{\ell\in\mathbb{Z}_{2q}}S[k_{\ell},k_{\ell+1}],
\end{equation}
where $\mu$ is the wost-case coherence of $\Phi$, and $S$ is the corresponding Seidel adjacency matrix.
Note that the left-hand side is necessarily nonnegative, while it is not immediate why the right-hand side should be.
This indicates that more simplification can be done, but for the sake of clarity, we will perform this simplification in the special case where $q=2$; the general case is very similar.
When $q=2$, we are concerned with 4-tuples $\{k_0,k_1,k_2,k_3\}\in\mathcal{K}^{(4)}$.
Let's partition these 4-tuples according to the value taken by $k_0$ and $k_q=k_2$.
Note, for a fixed $k_0$ and $k_2$, that $k_1$ can be any value other than $k_0$ or $k_2$, as can $k_3$.
This leads to the following simplification:
\begin{align*}
\sum_{\{k_\ell\}\in\mathcal{K}^{(4)}} \prod_{\ell\in\mathbb{Z}_{4}}S[k_{\ell},k_{\ell+1}]
&=\sum_{k_0\in\mathcal{K}}\sum_{k_2\in\mathcal{K}}\bigg(\sum_{\substack{k_1\in\mathcal{K}\\k_0\neq k_1\neq k_2}}S[k_0,k_1]S[k_1,k_2]\bigg)\bigg(\sum_{\substack{k_3\in\mathcal{K}\\k_2\neq k_3\neq k_0}}S[k_2,k_3]S[k_3,k_0]\bigg)\\
&=\sum_{k_0\in\mathcal{K}}\sum_{k_2\in\mathcal{K}}~~~~\bigg|\!\!\!\!\sum_{\substack{k\in\mathcal{K}\\k_0\neq k\neq k_2}}S[k_0,k]S[k,k_2]\bigg|^2\\
&=\sum_{k_0\in\mathcal{K}}\bigg|\sum_{\substack{k\in\mathcal{K}\\k\neq k_0}}S[k_0,k]S[k,k_0]\bigg|^2+\sum_{k_0\in\mathcal{K}}\sum_{\substack{k_2\in\mathcal{K}\\k_2\neq k_0}}~~~~\bigg|\!\!\!\!\sum_{\substack{k\in\mathcal{K}\\k_0\neq k\neq k_2}}S[k_0,k]S[k,k_2]\bigg|^2.
\end{align*}
The first term above is $K(K-1)^2$, while the other term is not as easy to analyze, as we expect a certain degree of cancellation.
Substituting this simplification into \eqref{eq.power method combinatorics} gives
\begin{equation*}
\mathrm{Tr}[(\Phi_\mathcal{K}^*\Phi_\mathcal{K}^{}-I_K)^4]
=\mu^4\bigg(K(K-1)^2+\sum_{k_0\in\mathcal{K}}\sum_{\substack{k_2\in\mathcal{K}\\k_2\neq k_1}}~~~~\bigg|\!\!\!\!\sum_{\substack{k\in\mathcal{K}\\k_0\neq k\neq k_2}}S[k_0,k]S[k,k_2]\bigg|^2\bigg).
\end{equation*}
If there were no cancellations in the second term, then it would equal $K(K-1)(K-2)^2$, thereby dominating the expression.
However, if oscillations occured as a $\pm1$ Bernoulli random variable, we could expect this term to be on the order of $K^3$, matching the order of the first term.
In this hypothetical case, since $\mu\leq M^{-1/2}$, the parameter $\delta_{K;2}^4$ defined in Theorem~\ref{thm.power method defn} scales as
$\frac{K^3}{M^2}$, and so $M\sim K^{3/2}$; this corresponds to the behavior exhibited in Theorem~\ref{thm.power to rip}.
To summarize, much like flat restricted orthogonality, applying the power method to ETFs leads to interesting combinatorial questions regarding subgraphs, even when $q=2$.

\subsection{The Paley equiangular tight frame as an RIP candidate}

Pick some prime $p\equiv 1\bmod 4$, and build an $M\times p$ matrix $H$ by selecting the $M:=\frac{p+1}{2}$ rows of the $p\times p$ discrete Fourier transform matrix which are indexed by $Q$, the quadratic residues modulo~$p$ (including zero).
To be clear, the entries of $H$ are scaled to have unit modulus.
Next, take $D$ to be an $M\times M$ diagonal matrix whose zeroth diagonal entry is $\sqrt{\frac{1}{p}}$, and whose remaining $M-1$ entries are $\sqrt{\frac{2}{p}}$.
Now build the matrix $\Phi$ by concatenating $DH$ with the zeroth identity basis element; for example, when $p=5$, we have a $3\times 6$ matrix:
\begin{equation*}
\Phi
=\left[\begin{array}{llllll} 
\sqrt{\frac{1}{5}}\quad\quad\quad&\sqrt{\frac{1}{5}}\quad\quad\quad&\sqrt{\frac{1}{5}}&\sqrt{\frac{1}{5}}\quad\quad\quad&\sqrt{\frac{1}{5}}\quad\quad\quad&1\\
\sqrt{\frac{2}{5}}&\sqrt{\frac{2}{5}}e^{-2\pi\mathrm{i}/5}&\sqrt{\frac{2}{5}}e^{-2\pi\mathrm{i}2/5}&\sqrt{\frac{2}{5}}e^{-2\pi\mathrm{i}3/5}&\sqrt{\frac{2}{5}}e^{-2\pi\mathrm{i}4/5}&0\\
\sqrt{\frac{2}{5}}&\sqrt{\frac{2}{5}}e^{-2\pi\mathrm{i}4/5}&\sqrt{\frac{2}{5}}e^{-2\pi\mathrm{i}3/5}&\sqrt{\frac{2}{5}}e^{-2\pi\mathrm{i}2/5}&\sqrt{\frac{2}{5}}e^{-2\pi\mathrm{i}/5}&0\\     
       \end{array}\right].
\end{equation*}
We claim that in general, this process produces an $M\times 2M$ equiangular tight frame, which we call the \emph{Paley ETF}~\cite{Renes:07}.
Presuming for the moment that this claim is true, we have the following result which lends hope for the Paley ETF as an RIP matrix:

\begin{lem}
\label{lem.paley ric}
An $M\times 2M$ Paley equiangular tight frame has restricted isometry constant $\delta_K<1$ for all $K\leq M$.
\end{lem}

\begin{proof}
First, we note that Theorem~\ref{thm.harmonic plus identity} used Chebotar\"{e}v's theorem~\cite{StevenhagenL:mi96} to prove that the spark of the $M\times 2M$ Paley ETF $\Phi$ is $M+1$, that is, every size-$M$ subcollection of columns of $\Phi$ forms a spanning set.
Thus, for every $\mathcal{K}\subseteq\{1,\ldots,2M\}$ of size $\leq M$, the smallest singular value of $\Phi_\mathcal{K}$ is positive.
It remains to show that the square of the largest singular value is strictly less than 2.
Let $x$ be a unit vector for which $\|\Phi_\mathcal{K}^*x\|=\|\Phi_\mathcal{K}^*\|_2$.
Then since the spark of $\Phi$ is $M+1$, the columns of $\Phi_{\mathcal{K}^\mathrm{c}}$ span, and so
\begin{equation*}
\|\Phi_\mathcal{K}\|_2^2
=\|\Phi_\mathcal{K}^*\|_2^2
=\|\Phi_\mathcal{K}^*x\|^2
<\|\Phi_\mathcal{K}^*x\|^2+\|\Phi_{\mathcal{K}^\mathrm{c}}^*x\|^2
=\|\Phi^* x\|^2
\leq\|\Phi^*\|_2^2
=\|\Phi\Phi^*\|_2
= 2,
\end{equation*}
where the final step follows from (i) and (ii) of Section~\ref{section.intro to frame theory}, which imply $\Phi\Phi^*=2I_M$.
\end{proof}

Now that we have an interest in the Paley ETF $\Phi$, we wish to verify that it is, in fact, an ETF.
It suffices to show that the columns of $\Phi$ have unit norm, and that the inner products between distinct columns equal the Welch bound in absolute value.
Certainly, the zeroth identity basis element is unit-norm, while the squared norm of each of the other columns is given by $\frac{1}{p}+(M-1)\frac{2}{p}=\frac{2M-1}{p}=1$.
Also, the inner product between the zeroth identity basis element and any other column equals the zeroth entry of that column: $p^{-1/2}=
(\frac{N-M}{M(N-1)})^{1/2}$.
It remains to calculate the inner product between distinct columns which are not identity basis elements.
To this end, note that since $a^2=b^2$ if and only if $a=\pm b$, the sequence $\{k^2\}_{k=1}^{p-1}\subseteq\mathbb{Z}_p$ doubly covers $Q\setminus\{0\}$, and so
\begin{equation*}
\langle\varphi_n,\varphi_{n'}\rangle
=\frac{1}{p}+\sum_{m\in Q\setminus\{0\}}\bigg(\sqrt{\frac{2}{p}}e^{-2\pi\mathrm{i}mn/p}\bigg)\bigg(\sqrt{\frac{2}{p}}e^{2\pi\mathrm{i}mn'/p}\bigg)
=\frac{1}{p}\sum_{k=0}^{p-1}e^{2\pi\mathrm{i}(n'-n)k^2/p}.
\end{equation*}
This well-known expression is called a quadratic Gauss sum, and since $p\equiv 1\bmod 4$, its value is determined by the Legendre symbol in the following way: $\langle\varphi_n,\varphi_{n'}\rangle=\frac{1}{\sqrt{p}}(\frac{n'-n}{p})$ for every $n,n'\in\mathbb{Z}_p$ with $n\neq n'$, where
\begin{equation*}
\bigg(\frac{k}{p}\bigg)
:=\left\{\begin{array}{rl}+1&\mbox{ if $k$ is a nonzero quadratic residue modulo $p$,}\\0&\mbox{ if $k=0$,}\\-1&\mbox{ otherwise.} \end{array}\right.
\end{equation*}

Having established that $\Phi$ is an ETF, we notice that the inner products between distinct columns of $\Phi$ are real.
This implies that the columns of $\Phi$ can be unitarily rotated to form a real ETF $\Psi$; indeed, one may take $\Psi$ to be the $M\times2M$ matrix formed by taking the nonzero rows of $L^\mathrm{T}$ in the Cholesky factorization $\Phi^*\Phi=LL^\mathrm{T}$.
As such, we consider the Paley ETF to be real.
From here, Theorem~\ref{thm.etf to srg} prompts us to find the corresponding strongly regular graph.
First, we can flip the identity basis element so that its inner products with the other columns of $\Phi$ are all negative.
As such, the corresponding vertex in the graph will be adjacent to each of the other vertices; naturally, this will be the vertex to which the strongly regular graph is joined.
For the remaining vertices, $n\leftrightarrow n'$ precisely when $(\frac{n'-n}{p})=-1$, that is, when $n'-n$ is not a quadratic residue.
The corresponding subgraph is therefore the complement of the Paley graph, namely, the Paley graph \cite{Sachs:62}.
In general, Paley graphs of order $p$ necessarily have $p\equiv 1\bmod 4$, and so this correspondence is particularly natural.

One interesting thing about the Paley ETF's restricted isometry is that it lends insight into important properties of the Paley graph.
The following is the best known upper bound for the clique number of the Paley graph of prime order (see Theorem 13.14 of \cite{Bollobas:01} and discussion thereafter), and we give a new proof of this bound using restricted isometry:

\begin{thm}
\label{thm.clique upper bound}
Let $G$ denote the Paley graph of prime order $p$.
Then the size of the largest clique is $\omega(G)<\sqrt{p}$. 
\end{thm}

\begin{proof}
We start by showing $\omega(G)+1\leq M$.
Suppose otherwise: that there exists a clique $\mathcal{K}$ of size $M+1$ in the join of a vertex with $G$.
Then the corresponding sub-Gram matrix of the Paley ETF has the form $\Phi_\mathcal{K}^*\Phi_\mathcal{K}^{}=(1+\mu)I_{M+1}-\mu J_{M+1}$, where $\mu=p^{-1/2}$ is the worst-case coherence and $J_{M+1}$ is the $(M+1)\times (M+1)$ matrix of 1's.
Since the largest eigenvalue of $J_{M+1}$ is $M+1$, the smallest eigenvalue of $\Phi_\mathcal{K}^*\Phi_\mathcal{K}^{}$ is $1+p^{-1/2}-(M+1)p^{-1/2}=1-\frac{1}{2}(p+1)p^{-1/2}$, which is negative when $p\geq 5$, contradicting the fact that $\Phi_\mathcal{K}^*\Phi_\mathcal{K}^{}$ is positive semidefinite.

Since $\omega(G)+1\leq M$, we can apply Lemma~\ref{lem.paley ric} and Theorem~\ref{thm.etf clique} to get
\begin{equation}
\label{eq.clique vs 1}
1>\delta_{\omega(G)+1}=\Big(\omega(G)+1-1\Big)\mu=\frac{\omega(G)}{\sqrt{p}},
\end{equation}
and rearranging gives the result.
\end{proof}

It is common to apply probabilistic and heuristic reasoning to gain intuition in number theory. 
For example, consecutive entries of the Legendre symbol are known to mimic certain properties of a $\pm1$ Bernoulli random variable \cite{Peralta:92}.
Moreover, Paley graphs enjoy a certain quasi-random property that was studied in \cite{ChungGW:89}.
On the other hand, Graham and Ringrose~\cite{GrahamR:90} showed that, while random graphs of size $p$ have an expected clique number of $(1+o(1))2\log p/\log 2$, Paley graphs of prime order deviate from this random behavior, having a clique number $\geq c\log p\log\log\log p$ infinitely often.
The best known universal lower bound, $(1/2+o(1))\log p/\log 2$, is given in \cite{Cohen:88}, which indicates that the random graph analysis is at least tight in some sense.
Regardless, this has a significant difference from the upper bound $\sqrt{p}$ in Theorem~\ref{thm.clique upper bound}, and it would be nice if probabilistic arguments could be leveraged to improve this bound, or at least provide some intuition.

Note that our proof \eqref{eq.clique vs 1} hinged on the fact that $\delta_{\omega(G)+1}<1$, courtesy of Lemma~\ref{lem.paley ric}.
Hence, any improvement to our estimate for $\delta_{\omega(G)+1}$ would directly lead to the best known upper bound on the Paley graph's clique number.
To approach such an improvement, note that for large $p$, the Fourier portion of the Paley ETF $DH$ is not significatly different from the normalized partial Fourier matrix $(\frac{2}{p+1})^{1/2}H$; indeed, $\|H_\mathcal{K}^*D^2H_\mathcal{K}^{}-\frac{2}{p+1}H_\mathcal{K}^*H_\mathcal{K}^{}\|_2\leq\frac{2}{p}$ for every $\mathcal{K}\subseteq\mathbb{Z}_p$ of size $\leq\frac{p+1}{2}$, and so the difference vanishes.
If we view the quadratic residues modulo $p$ (the row indices of $H$) as random, then a random partial Fourier matrix serves as a proxy for the Fourier portion of the Paley ETF.
This in mind, we appeal to the following:

\begin{thm}[Theorem 3.2 in \cite{Rauhut:08}]
Draw rows from the $N\times N$ discrete Fourier transform matrix uniformly at random with replacement to construct an $M\times N$ matrix, and then normalize the columns to form $\Phi$.
Then $\Phi$ has restricted isometry constant $\delta_K\leq\delta$ with probability $1-\varepsilon$ provided $\frac{M}{\log M}\geq \frac{C}{\delta^2}K\log^2 K\log N\log\varepsilon^{-1}$, where $C$ is a universal constant.
\end{thm}

In our case, both $M$ and $N$ scale as $p$, and so picking $\delta$ to achieve equality above gives
\begin{equation*}
\delta^2
=\frac{C'}{p}K\log^2K\log^2 p\log\varepsilon^{-1}.
\end{equation*}
Continuing as in \eqref{eq.clique vs 1}, denote $\omega=\omega(G)$ and take $K=\omega$ to get
\begin{equation*}
\frac{C'}{p}\omega\log^2\omega\log^2 p\log\varepsilon^{-1}
\geq\delta_{\omega}^2
=\frac{(\omega-1)^2}{p}
\geq\frac{\omega^2}{2p},
\end{equation*}
and then rearranging gives  $\omega/\log^2\omega\leq C''\log^2p\log\varepsilon^{-1}$ with probability $1-\varepsilon$.
Interestingly, having $\omega/\log^2\omega=\mathrm{O}(\log^3p)$ with high probability (again, under the model that quadratic residues are random) agrees with the results of Graham and Ringrose~\cite{GrahamR:90}.
This gives some intuition for what we can expect the size of the Paley graph's clique number to be, while at the same time demonstrating the power of Paley ETFs as RIP candidates. 
We conclude with the following, which can be reformulated in terms of both flat restricted orthogonality and the power method:

\begin{conj}
The Paley equiangular tight frame has the $(K,\delta)$-restricted isometry property with some $\delta<\sqrt{2}-1$ whenever $K\leq\frac{Cp}{\log^\alpha p}$, for some universal constants $C$ and $\alpha$.
\end{conj}

\section{Appendix}

In this section, we prove Theorem~\ref{thm.fro}, which states that a matrix with $(K,\hat\theta)$-flat restricted orthogonality has $\theta_K\leq C\hat\theta\log K$, that is, it has restricted orthogonality.
The proof below is adapted from the proof of Lemma 3 in~\cite{BourgainDFKK:11}. 
Our proof has the benefit of being valid for all values of $K$ (as opposed to sufficiently large $K$ in the original~\cite{BourgainDFKK:11}), and it has near-optimal constants where appropriate.
Moreover in this version, the columns of the matrix are not required to have unit norm.

\begin{proof}[Proof of Theorem~\ref{thm.fro}]
Given arbitrary disjoint subsets $\mathcal{I},\mathcal{J}\subseteq\{1,\ldots,N\}$ with $|\mathcal{I}|,|\mathcal{J}|\leq K$, we will bound the following quantity three times, each time with different constraints on $\{x_i\}_{i\in\mathcal{I}}$ and $\{y_j\}_{j\in\mathcal{J}}$:
\begin{equation}
\label{eq.to bound}
\bigg|\bigg\langle \sum_{i\in\mathcal{I}}x_i\varphi_i,\sum_{j\in\mathcal{J}}y_j\varphi_j \bigg\rangle\bigg|.
\end{equation}
To be clear, our third bound will have no constraints on $\{x_i\}_{i\in\mathcal{I}}$ and $\{y_j\}_{j\in\mathcal{J}}$, thereby demonstrating restricted orthogonality.
Note that by assumption, \eqref{eq.to bound} is $\leq\hat\theta(|\mathcal{I}||\mathcal{J}|)^{1/2}$ whenever the $x_i$'s and $y_j$'s are in $\{0,1\}$.
We first show that this bound is preserved when we relax the $x_i$'s and $y_j$'s to lie in the interval $[0,1]$.

Pick a disjoint pair of subsets $\mathcal{I}',\mathcal{J}'\subseteq\{1,\ldots,N\}$ with $|\mathcal{I}'|,|\mathcal{J}'|\leq K$.
Starting with some $k\in\mathcal{I}'$, note that flat restricted orthogonality gives that
\begin{align*}
\bigg|\bigg\langle \sum_{i\in\mathcal{I}}\varphi_i,\sum_{j\in\mathcal{J}}\varphi_j \bigg\rangle\bigg|
&\leq\hat\theta(|\mathcal{I}||\mathcal{J}|)^{1/2},\\
\bigg|\bigg\langle \sum_{i\in\mathcal{I}\setminus\{k\}}\varphi_i,\sum_{j\in\mathcal{J}}\varphi_j \bigg\rangle\bigg|
&\leq\hat\theta(|\mathcal{I}\setminus\{k\}||\mathcal{J}|)^{1/2}
\leq\hat\theta(|\mathcal{I}||\mathcal{J}|)^{1/2}
\end{align*}
for every disjoint $\mathcal{I},\mathcal{J}\subseteq\{1,\ldots,N\}$ with $|\mathcal{I}|,|\mathcal{J}|\leq K$ and $k\in\mathcal{I}$.
Thus, we may take any $x_k\in[0,1]$ to form a convex combination of these two expressions, and then the triangle inequality gives
\begin{align}
\nonumber
\hat\theta(|\mathcal{I}||\mathcal{J}|)^{1/2}
&\geq x_k\bigg|\bigg\langle \sum_{i\in\mathcal{I}}\varphi_i,\sum_{j\in\mathcal{J}}\varphi_j \bigg\rangle\bigg|+(1-x_k)\bigg|\bigg\langle \sum_{i\in\mathcal{I}\setminus\{k\}}\varphi_i,\sum_{j\in\mathcal{J}}\varphi_j \bigg\rangle\bigg|\\
\nonumber
&\geq\bigg|x_k\bigg\langle \sum_{i\in\mathcal{I}}\varphi_i,\sum_{j\in\mathcal{J}}\varphi_j \bigg\rangle+(1-x_k)\bigg\langle \sum_{i\in\mathcal{I}\setminus\{k\}}\varphi_i,\sum_{j\in\mathcal{J}}\varphi_j \bigg\rangle\bigg|\\
\label{eq.convex trick}
&=\bigg|\bigg\langle \sum_{i\in\mathcal{I}}\bigg\{\begin{array}{cc}x_k,&i=k\\1,&i\neq k\end{array}\bigg\}\varphi_i,\sum_{j\in\mathcal{J}}\varphi_j \bigg\rangle\bigg|.
\end{align}
Since \eqref{eq.convex trick} holds for every disjoint $\mathcal{I},\mathcal{J}\subseteq\{1,\ldots,N\}$ with $|\mathcal{I}|,|\mathcal{J}|\leq K$ and $k\in\mathcal{I}$, we can do the same thing with an additional index $i\in\mathcal{I}'$ or $j\in\mathcal{J}'$, and replace the corresponding unit coefficient with some $x_i$ or $y_j$ in $[0,1]$.
Continuing in this way proves the claim that \eqref{eq.to bound} is $\leq\hat\theta(|\mathcal{I}||\mathcal{J}|)^{1/2}$ whenever the $x_i$'s and $y_j$'s lie in the interval $[0,1]$.

For the second bound, we assume the $x_i$'s and $y_j$'s are nonnegative with unit norm: $\sum_{i\in\mathcal{I}}x_i^2=\sum_{j\in\mathcal{J}}y_j^2=1$.
To bound \eqref{eq.to bound} in this case, we partition $\mathcal{I}$ and $\mathcal{J}$ according to the size of the corresponding coefficients:
\begin{equation*}
\mathcal{I}_k:=\{i\in\mathcal{I}:2^{-(k+1)}<x_i\leq 2^{-k}\},
\qquad
\mathcal{J}_k:=\{j\in\mathcal{J}:2^{-(k+1)}<y_j\leq 2^{-k}\}.
\end{equation*}
Note the unit-norm constraints ensure that $\mathcal{I}=\bigcup_{k=0}^\infty\mathcal{I}_k$ and $\mathcal{J}=\bigcup_{k=0}^\infty\mathcal{J}_k$.
The triangle inequality thus gives
\begin{align}
\nonumber
\bigg|\bigg\langle \sum_{i\in\mathcal{I}}x_i\varphi_i,\sum_{j\in\mathcal{J}}y_j\varphi_j \bigg\rangle\bigg|
&=\bigg|\bigg\langle \sum_{k_1=0}^\infty\sum_{i\in\mathcal{I}_{k_1}}x_i\varphi_i,\sum_{k_2=0}^\infty\sum_{j\in\mathcal{J}_{k_2}}y_j\varphi_j \bigg\rangle\bigg|\\
\label{eq.partition across sizes}
&\leq\sum_{k_1=0}^\infty\sum_{k_2=0}^\infty 2^{-(k_1+k_2)} \bigg|\bigg\langle \sum_{i\in\mathcal{I}_{k_1}}\frac{x_i}{2^{-k_1}}\varphi_i,\sum_{j\in\mathcal{J}_{k_2}}\frac{y_j}{2^{-k_2}}\varphi_j \bigg\rangle\bigg|.
\end{align}
By the definitions of $\mathcal{I}_{k_1}$ and $\mathcal{J}_{k_2}$, the coefficients of $\varphi_i$ and $\varphi_j$ in \eqref{eq.partition across sizes} all lie in $[0,1]$.
As such, we continue by applying our first bound:
\begin{align}
\nonumber
\bigg|\bigg\langle \sum_{i\in\mathcal{I}}x_i\varphi_i,\sum_{j\in\mathcal{J}}y_j\varphi_j \bigg\rangle\bigg|
&\leq\sum_{k_1=0}^\infty\sum_{k_2=0}^\infty 2^{-(k_1+k_2)} \hat\theta (|\mathcal{I}_{k_1}||\mathcal{J}_{k_2}|)^{1/2}\\
\label{eq.partition across sizes 2}
&=\hat\theta\bigg(\sum_{k=0}^\infty2^{-k}|\mathcal{I}_{k}|^{1/2}\bigg)\bigg(\sum_{k=0}^\infty2^{-k}|\mathcal{J}_{k}|^{1/2}\bigg).
\end{align}
We now observe from the definition of $\mathcal{I}_k$ that
\begin{equation*}
1
=\sum_{i\in\mathcal{I}}x_i^2
=\sum_{k=0}^\infty\sum_{i\in\mathcal{I}_k}x_i^2
>\sum_{k=0}^\infty 4^{-(k+1)}|\mathcal{I}_k|.
\end{equation*}
Thus for any positive integer $t$, the Cauchy-Schwarz inequality gives
\begin{align}
\nonumber
\sum_{k=0}^\infty 2^{-k}|\mathcal{I}_k|^{1/2}
&=\sum_{k=0}^{t-1} 2^{-k}|\mathcal{I}_k|^{1/2}+\sum_{k=t}^\infty 2^{-k}|\mathcal{I}_k|^{1/2}\\
\nonumber
&\leq t^{1/2}\bigg(\sum_{k=0}^{t-1} 4^{-k}|\mathcal{I}_k|\bigg)^{1/2}+\sum_{k=t}^\infty 2^{-k}K^{1/2}\\
\label{eq.partition across sizes 3}
&< 2(t^{1/2}+K^{1/2}2^{-t}),
\end{align}
and similarly for the $\mathcal{J}_k$'s.
For a fixed $K$, we note that \eqref{eq.partition across sizes 3} is minimized when $K^{1/2}2^{-t}=\frac{t^{-1/2}}{2\log 2}$, and so we pick $t$ to be the smallest positive integer such that $K^{1/2}2^{-t}\leq\frac{t^{-1/2}}{2\log 2}$.
With this, we continue \eqref{eq.partition across sizes 2}:
\begin{align}
\nonumber
\bigg|\bigg\langle \sum_{i\in\mathcal{I}}x_i\varphi_i,\sum_{j\in\mathcal{J}}y_j\varphi_j \bigg\rangle\bigg|
&<\hat\theta\Big(2(t^{1/2}+K^{1/2}2^{-t})\Big)^2\\
\label{eq.partition across sizes 4}
&\leq 4\hat\theta \bigg(t^{1/2}+\frac{t^{-1/2}}{2\log 2}\bigg)^2
=4\hat\theta\bigg(t+\frac{1}{\log 2}+\frac{1}{(2\log 2)^2t}\bigg).
\end{align}
From here, we claim that $t\leq\lceil\frac{\log K}{\log 2}\rceil$.
Considering the definition of $t$, this is easily verified for $K=2,3,\ldots,7$ by showing $K^{1/2}2^{-s}\leq\frac{s^{-1/2}}{2\log 2}$ for $s=\lceil\frac{\log K}{\log 2}\rceil$.
For $K\geq8$, one can use calculus to verify the second inequality of the following:
\begin{equation*}
K^{1/2}2^{-\lceil\frac{\log K}{\log 2}\rceil}
\leq K^{1/2}2^{-\frac{\log K}{\log 2}}
\leq \frac{1}{2\log 2}\bigg(\frac{\log K}{\log 2}+1\bigg)^{-1/2}
\leq \frac{1}{2\log 2}\bigg\lceil\frac{\log K}{\log 2}\bigg\rceil^{-1/2},
\end{equation*}
meaning $t\leq\lceil\frac{\log K}{\log 2}\rceil$.
Substituting $t\leq\frac{\log K}{\log 2}+1$ and $t\geq1$ into \eqref{eq.partition across sizes 4} then gives
\begin{equation*}
\bigg|\bigg\langle \sum_{i\in\mathcal{I}}x_i\varphi_i,\sum_{j\in\mathcal{J}}y_j\varphi_j \bigg\rangle\bigg|
<4\hat\theta\bigg(\frac{\log K}{\log2}+1+\frac{1}{\log 2}+\frac{1}{(2\log 2)^2}\bigg)\\
=\hat\theta(C_0\log K+C_1),
\end{equation*}
with $C_0\approx 5.77$, $C_1\approx 11.85$.
As such, \eqref{eq.to bound} is $\leq C'\hat\theta\log K$ with $C'=C_0+\frac{C_1}{\log 2}$ in this case.

We are now ready for the final bound on \eqref{eq.to bound} in which we apply no constraints on the $x_i$'s and $y_j$'s.
To do this, we consider the positive and negative real and imaginary parts of these coefficients:
\begin{equation*}
x_i=\sum_{k=0}^3x_{i,k}\mathrm{i}^k\quad \mbox{s.t.}\quad x_{i,k}\geq0\quad \forall k,
\end{equation*}
and similarly for the $y_j$'s.
With this decomposition, we apply the triangle inequality to get
\begin{align*}
\bigg|\bigg\langle \sum_{i\in\mathcal{I}}x_i\varphi_i,\sum_{j\in\mathcal{J}}y_j\varphi_j \bigg\rangle\bigg|
&=\bigg|\bigg\langle \sum_{i\in\mathcal{I}}\sum_{k_1=0}^3x_{i,k_1}\mathrm{i}^{k_1}\varphi_i,\sum_{j\in\mathcal{J}}\sum_{k_2=0}^3y_{j,k_2}\mathrm{i}^{k_2}\varphi_j \bigg\rangle\bigg|\\
&\leq\sum_{k_1=0}^3\sum_{k_2=0}^3\bigg|\bigg\langle \sum_{i\in\mathcal{I}}x_{i,k_1}\varphi_i,\sum_{j\in\mathcal{J}}y_{j,k_2}\varphi_j \bigg\rangle\bigg|.
\end{align*}
Finally, we normalize the coefficients by $(\sum_{i\in\mathcal{I}}x_{i,k_1}^2)^{1/2}$ and $(\sum_{j\in\mathcal{J}}y_{j,k_2}^2)^{1/2}$ so we can apply our second bound:
\begin{align*}
\bigg|\bigg\langle \sum_{i\in\mathcal{I}}x_i\varphi_i,\sum_{j\in\mathcal{J}}y_j\varphi_j \bigg\rangle\bigg|
&\leq\sum_{k_1=0}^3\sum_{k_2=0}^3\bigg(\sum_{i\in\mathcal{I}}x_{i,k_1}^2\bigg)^{1/2}\bigg(\sum_{j\in\mathcal{J}}y_{j,k_2}^2\bigg)^{1/2} C'\hat\theta\log K\\
&\leq(C\hat\theta\log K)\|x\|\|y\|,
\end{align*}
where $C=4C'\approx74.17$ by the Cauchy-Schwarz inequality, and so we are done.
\end{proof}

\chapter{Two fundamental parameters of frame coherence}

Chapters 1--3 of this thesis were dedicated to a particularly popular understanding of compressed sensing: that matrices which satisfy the restricted isometry property (RIP) are very well-suited as sensing matrices.
However, as these chapters show, it is very difficult to deterministically construct matrices which are provably RIP.
It is therefore desirable to find a worthy alternative to RIP which admits deterministic sensing matrices.
The present chapter is dedicated to one such alternative, namely the \emph{strong coherence property}, but before we define this property, we first motivate it in the context of a support recovery method known as \emph{one-step thresholding (OST)}.

The main idea behind OST is that the noiseless measurement vector $y=\Phi x$ will look similar to the active columns of $\Phi=[\varphi_1\cdots\varphi_N]$, provided the sparsity level is sufficiently small and the nonzero members of $x$ are sufficiently large in some sense.
Using this intuition, it makes sense to find the support of $x$ by finding the large values of
\begin{equation*}
|\langle \varphi_i,y\rangle|
=\bigg|\bigg\langle \varphi_i,\sum_{j=1}^{N}x_j\varphi_j\bigg\rangle\bigg|
=\bigg|\sum_{j=1}^{N}x_j\langle \varphi_i,\varphi_j\rangle\bigg|
=\bigg|x_i+\sum_{\substack{j=1\\j\neq i}}^{N}x_j\langle \varphi_i,\varphi_j\rangle\bigg|,
\end{equation*}
assuming the columns of $\Phi$ have unit norm.
Indeed, if the nonzero entries of $x$ are larger than the contribution of the cross-column interactions, then the above calculation serves as a reasonable test for the support of $x$.
The magnitude of this contribution can be assessed using two measures of coherence.
Indeed, if the columns are incoherent, then each term of this sum is small, and so it makes sense to consider the worst-case coherence of $\Phi$:
\begin{equation}
\label{eq.mu defn}
\mu:=\max_{\substack{i,j\in\{1,\ldots,N\}\\i\neq j}}|\langle \varphi_i,\varphi_j\rangle|.
\end{equation}
However, this measure of coherence does not account for sign fluxuations in the inner products, which should bring significant cancellations in the sum.
If we assume the support of $x$ is drawn randomly, then by a concentration-of-measure argument, this sum will typically be close to its expectation, and so its size will rarely exceed some multiple of $\|x\|_1$ times the following maximum average:
\begin{equation}
\label{eq.nu defn}
\nu:=\max_{i\in\{1,\ldots,N\}}\bigg|\frac{1}{N-1}\sum_{\substack{j=1\\j\neq i}}^N\langle \varphi_i,\varphi_j\rangle\bigg|.
\end{equation}
For this reason, this notion of coherence, called \emph{average coherence}, was recently introduced in~\cite{bajwa:jcn10}.

Intuitively, worst-case coherence is a measure of dissimilarity between frame elements, whereas average coherence measures how well the frame elements are distributed in the unit hypersphere.
As we will see, both worst-case and average coherence play an important role in various portions of sparse signal processing, provided we describe the sparse signal's support with a probabilistic model.
In fact, \cite{bajwa:jcn10} used worst-case and average coherence to produce probabilistic reconstruction guarantees for OST, permitting sparsity levels on the order of $\smash{\frac{M}{\log N}}$ (akin to the RIP-based guarantees).
In accordance with our motivation above, these probabilistic guarantees require that worst-case and average coherence together satisfy the following property:

\begin{defn}
We say an $M\times N$ unit norm frame $\Phi$ satisfies the \emph{strong coherence property} if
\begin{equation*}
\mbox{(SCP-1)}~~~\mu\leq\frac{1}{164\log N}\qquad\mbox{and}\qquad\mbox{(SCP-2)}~~~\nu\leq\frac{\mu}{\sqrt{M}},
\end{equation*}
where $\mu$ and $\nu$ are given by \eqref{eq.mu defn} and \eqref{eq.nu defn}, respectively.
\end{defn}

The reader should know that the constant $164$ is not particularly essential to the above definition; it is used in \cite{bajwa:jcn10} to simplify some analysis and make certain performance guarantees explicit, but the constant is by no means optimal.
In the next section, we will use the strong coherence property to continue the work of \cite{bajwa:jcn10}.
Where \cite{bajwa:jcn10} provided guarantees for noiseless reconstruction, we will produce near-optimal guarantees for signal detection and reconstruction from \emph{noisy} measurements of sparse signals.
These guarantees are related to those in \cite{candes:annstat09,DonohoET:06,tropp:cras08,Tropp:acha08}, and we will also elaborate on this relationship.

The results given in \cite{bajwa:jcn10} and the following section, as well as the applications discussed in \cite{candes:annstat09,DonohoET:06,HolmesP:04,MixonQKF:11,StrohmerH:03,Tropp:04,Tropp:acha08,zahedi:acc10}
demonstrate a pressing need for nearly tight frames with small worst-case and average coherence, especially in sparse signal processing.
This chapter offers three additional contributions in this regard~\cite{BajwaCM:12,MixonBC:11}.
In Section~4.2, we provide a sizable catalog of frames that exhibit small spectral norm, worst-case coherence, and average coherence.
With all three frame parameters provably small, these frames are guaranteed to perform well in relevant applications.
Next, performance in many applications is dictated by worst-case coherence.
It is therefore particularly important to understand which worst-case coherence values are achievable.
To this end, the Welch bound (Theorem~\ref{thm.welch bound}) is commonly used in the literature.
However, the Welch bound is only tight when the number of frame elements $N$ is less than the square of the spatial dimension $M$ \cite{StrohmerH:03}.
Another lower bound, given in \cite{MSEA03,XiaZG:05}, beats the Welch bound when there are more frame elements, but it is known to be loose for real frames \cite{CHS96}.
Given this context, Section~4.3 gives a new lower bound on the worst-case coherence of real frames.
Our bound beats both the Welch bound and the bound in \cite{MSEA03,XiaZG:05} when the number of frame elements far exceeds the spatial dimension.
Finally, since average coherence is so new, there is currently no intuition as to when (SCP-2) is satisfied.
In Section~4.4, we use ideas akin to the switching equivalence of graphs to transform a frame that satisfies (SCP-1) into another frame with the same spectral norm and worst-case coherence that additionally satisfies (SCP-2).

\section{Implications of worst-case and average coherence}

Frames with small spectral norm, worst-case coherence, and/or average coherence have found use in recent years with applications involving sparse signals.
Donoho et al.~used the worst-case coherence in \cite{DonohoET:06} to provide uniform bounds on the signal and support recovery performance of combinatorial and convex optimization methods and greedy algorithms.
Later, Tropp \cite{Tropp:acha08} and Cand\`{e}s and Plan \cite{candes:annstat09} used both the spectral norm and worst-case coherence to provide tighter bounds on the signal and support recovery performance of convex optimization methods for most support sets under the additional assumption that the sparse signals have independent nonzero entries with zero median.
Recently, Bajwa et al.~\cite{bajwa:jcn10} made use of the spectral norm and both coherence parameters to report tighter bounds on the noisy model selection and noiseless signal recovery performance of an incredibly fast greedy algorithm called \emph{one-step thresholding (OST)} for most support sets and \emph{arbitrary} nonzero entries.
In this section, we discuss further implications of the spectral norm and worst-case and average coherence of frames in applications involving sparse signals.

\subsection{The weak restricted isometry property}
A common task in signal processing applications is to test whether a collection of measurements corresponds to mere noise \cite{kay:98b}.
For applications involving sparse signals, one can test measurements $y \in \mathbb{C}^M$ against the null hypothsis $H_0: y = z$ and alternative hypothesis $H_1: y = \Phi x+z$, where the entries of the noise vector $z\in \mathbb{C}^M$ are independent, identical zero-mean complex-Gaussian random variables and the signal $x\in\mathbb{C}^N$ is $K$-sparse.
The performance of such signal detection problems is directly proportional to the energy in $\Phi x$ \cite{davenport:jstsp10,haupt:icassp07,kay:98b}.
In particular, existing literature on the detection of sparse signals \cite{davenport:jstsp10,haupt:icassp07} leverages the fact that $\|\Phi x\|^2 \approx \|x\|^2$ when $\Phi$ satisfies the restricted isometry property (RIP) of order $K$.
In contrast, we now show that the strong coherence property also guarantees $\|\Phi x\|^2 \approx \|x\|^2$ for most $K$-sparse vectors.
We start with a definition:

\begin{defn}
\label{def:WRIP}
We say an $M\times N$ frame $\Phi$ satisfies the \emph{$(K,\delta,p)$-weak restricted isometry property (weak RIP)} if for every $K$-sparse vector $y \in \mathbb{C}^N$, a random permutation $x$ of $y$'s entries satisfies
\begin{equation}
\label{thmeqn:REP}
(1-\delta)\|x\|^2 \leq \|\Phi x\|^2 \leq (1+\delta)\|x\|^2
\end{equation}
with probability exceeding $1-p$.
\end{defn}

At first glance, it may seem odd that we introduce a random permutation when we might as well define weak RIP in terms of a $K$-sparse vector whose support is drawn randomly from all $\smash{\binom{N}{K}}$ possible choices.
In fact, both versions would be equivalent in distribution, but we stress that in the present definition, the values of the nonzero entries of $x$ are \emph{not} random; rather, the only randomness we have is in the locations of the nonzero entries.
We wish to distinguish our results from those in \cite{candes:annstat09}, which explicitly require randomness in the values of the nonzero entries.
We also note the distinction between RIP and weak RIP---weak RIP requires that $\Phi$ preserves the energy of \emph{most} sparse vectors.
Moreover, the manner in which we quantify ``most'' is important.
For each sparse vector, $\Phi$ preserves the energy of most permutations of that vector, but for different sparse vectors, $\Phi$ might not preserve the energy of permutations with the same support.
That is, unlike RIP, weak RIP is \emph{not} a statement about the singular values of submatrices of $\Phi$.
Certainly, matrices for which most submatrices are well-conditioned, such as those discussed in \cite{tropp:cras08,Tropp:acha08}, will satisfy weak RIP, but weak RIP does not require this.
That said, the following theorem shows, in part, the significance of the strong coherence property.

\begin{thm}
\label{thm.WRIP}
Any $M\times N$ unit norm frame $\Phi$ with the strong coherence property
satisfies the $(K,\delta,\frac{4K}{N^2})$-weak restricted isometry property
provided $N \geq 128$ and $\smash{2K\log{N} \leq \min\{\frac{\delta^2}{100\mu^2},M\}}$.
\end{thm}

\begin{proof}
Let $x$ be as in Definition~\ref{def:WRIP}.
Note that \eqref{thmeqn:REP} is equivalent to $\big|\|\Phi x\|^2-\|x\|^2\big|\leq\delta\|x\|^2$.
Defining $\mathcal{K}:=\{n:|x_n|>0\}$, then the Cauchy-Schwarz inequality gives
\begin{align}
\nonumber
\big|\|\Phi x\|^2-\|x\|^2\big|
&=|x_\mathcal{K}^*(\Phi_\mathcal{K}^*\Phi_\mathcal{K}^{}-I_K)x_\mathcal{K}|\\
\label{pfeqn:REP_1}
&\leq\|x_\mathcal{K}\|\|(\Phi_\mathcal{K}^*\Phi_\mathcal{K}^{}-I_K)x_\mathcal{K}\|
\leq\sqrt{K}\|x_\mathcal{K}\|\|(\Phi_\mathcal{K}^*\Phi_\mathcal{K}^{}-I_K)x_\mathcal{K}\|_\infty,
\end{align}
where the last inequality uses the fact that $\|\cdot\|\leq\sqrt{K}\|\cdot\|_\infty$ in $\mathbb{C}^K$.
We now consider Lemma~3 of~\cite{bajwa:jcn10}, which states that for any $\varepsilon \in [0,1)$ and $a \geq 1$, $\|(\Phi_\mathcal{K}^*\Phi_\mathcal{K}^{}-I_K)x_\mathcal{K}\|_\infty \leq \varepsilon \|x_\mathcal{K}\|$ with probability exceeding $1-4Ke^{- (\varepsilon-\sqrt{K}\nu)^2/16(2+a^{-1})^2\mu^2}$ provided $K \leq \min\{\varepsilon^2\nu^{-2}, (1+a)^{-1}N\}$. 
We claim that \eqref{pfeqn:REP_1} together with Lemma~3 of~\cite{bajwa:jcn10} guarantee $\big|\|\Phi x\|^2-\|x\|^2\big|\leq\delta\|x\|^2$ with probability exceeding $1-\frac{4K}{N^2}$.
In order to establish this claim, we fix $\varepsilon=10\mu\sqrt{2\log{N}}$ and $a=2\log{128}-1$. 
It is then easy to see that (SCP-1) gives $\varepsilon < 1$, and also that (SCP-2) and $2K\log{N} \leq M$ give $K \leq \varepsilon^2\nu^{-2}/9$. 
Therefore, since the assumption that $N \geq 128$ together with $2K\log{N} \leq M$ implies $K \leq (1+a)^{-1}N$, we obtain $e^{- (\varepsilon - \sqrt{K}\nu)^2/16(2+a^{-1})^2\mu^2} \leq \frac{1}{N^2}$. 
The result now follows from the observation that $2K\log{N} \leq \frac{\delta^2}{100\mu^2}$ implies $\sqrt{K}\varepsilon \leq \delta$.
\end{proof}

This theorem shows that having small worst-case and average coherence is enough to guarantee weak RIP.
This contrasts with related results by Tropp \cite{tropp:cras08,Tropp:acha08} that require $\Phi$ to be nearly tight.
In fact, the proof of Theorem~\ref{thm.WRIP} does not even use the full power of the strong coherence property; instead of (SCP-1), it suffices to have $\mu\leq1/(15\sqrt{\log N})$, part of what \cite{bajwa:jcn10} calls the coherence property.
Also, if $\Phi$ has worst-case coherence $\mu=\mathrm{O}(1/\sqrt{M})$ and average coherence $\nu=\mathrm{O}(1/M)$, then even if $\Phi$ has large spectral norm, Theorem~\ref{thm.WRIP} states that $\Phi$ preserves the energy of most $K$-sparse vectors with $K=\mathrm{O}(M/\log N)$, i.e., the sparsity regime which is linear in the number of measurements.

\subsection{Reconstruction of sparse signals from noisy measurements}
Another common task in signal processing applications is to reconstruct a
$K$-sparse signal $x\in\mathbb{C}^N$ from a small collection of linear
measurements $y\in\mathbb{C}^M$. Recently, Tropp \cite{Tropp:acha08} used both
the worst-case coherence and spectral norm of frames to find bounds on the
reconstruction performance of \emph{basis pursuit (BP)} \cite{donoho:siamjsc98}
for most support sets under the assumption that the nonzero entries of $x$ are
independent with zero median. In contrast, \cite{bajwa:jcn10} used the spectral
norm and worst-case and average coherence of frames to find bounds on the
reconstruction performance of OST for most support sets and \emph{arbitrary}
nonzero entries. However, both \cite{bajwa:jcn10} and \cite{Tropp:acha08} limit
themselves to recovering $x$ in the absence of noise, corresponding to $y =
\Phi x$, a rather ideal scenario.

Our goal in this section is to provide guarantees for the reconstruction of
sparse signals from noisy measurements $y=\Phi x+z$, where the entries of the noise
vector $z\in \mathbb{C}^M$ are independent, identical complex-Gaussian random
variables with mean zero and variance $\sigma^2$. In particular, and in
contrast with \cite{DonohoET:06}, our guarantees will hold for arbitrary unit norm
frames $\Phi$ without requiring the signal's sparsity level to satisfy $K=\mathrm{O}(\mu^{-1})$. 
The reconstruction algorithm that we analyze here is the OST
algorithm of \cite{bajwa:jcn10}, which is described in
Algorithm~\ref{alg:OST_recon}. The following theorem extends the analysis of
\cite{bajwa:jcn10} and shows that the OST algorithm leads to near-optimal
reconstruction error for certain important classes of sparse signals.

Before proceeding further, we first define some notation. We use
$\textsc{snr}:=\|x\|^2/\mathbb{E}[\|z\|^2]$ to denote the
\emph{signal-to-noise ratio} associated with the signal reconstruction problem.
Also, we use 
\begin{equation*}
\mathcal{T}_\sigma(t)
:=\bigg\{n:|x_n|>\frac{2\sqrt{2}}{1-t}\sqrt{2 \sigma^2 \log{N}}\bigg\}
\end{equation*}
for any $t \in (0,1)$ to denote the locations of all the entries of $x$ that, roughly speaking, lie above the \emph{noise floor} $\sigma$. Finally, we use
\begin{equation*}
\mathcal{T}_\mu(t)
:=\bigg\{n: |x_n| > \frac{20}{t}\mu\|x\|\sqrt{2\log{N}}\bigg\}
\end{equation*}
to denote the locations of entries that, roughly speaking,
lie above the \emph{self-interference floor} $\mu\|x\|$.
%***************
\begin{algorithm*}[t]
\caption{One-Step Thresholding (OST) for sparse signal reconstruction \cite{bajwa:jcn10}}
\label{alg:OST_recon}
\textbf{Input:} An $M \times N$ unit norm frame $\Phi$, a vector $y=\Phi x+z$, and a threshold $\lambda > 0$\\
\textbf{Output:} An estimate $\hat{x} \in \mathbb{C}^N$ of the true sparse signal $x$
\begin{algorithmic}
\STATE $\hat{x} \leftarrow 0$ \hfill \COMMENT{Initialize}
\STATE $\tilde{x} \leftarrow \Phi^* y$ \hfill \COMMENT{Form signal proxy}
\STATE $\hat{\mathcal{K}} \leftarrow \{n : |\tilde{x}_n| > \lambda\}$ \hfill \COMMENT{Select indices via OST}
\STATE $\hat{x}_{\hat{\mathcal{K}}} \leftarrow (\Phi_{\hat{\mathcal{K}}})^\dagger y$ \hfill \COMMENT{Reconstruct signal via least-squares}
\end{algorithmic}
\end{algorithm*}
%***************

\begin{thm}[Reconstruction of sparse signals]
\label{thm:RSP}
Take an $M\times N$ unit norm frame $\Phi$ which satisfies the strong coherence property, pick $t\in(0,1)$, and choose $\lambda = \sqrt{2\sigma^2\log{N}}\max \{\frac{10}{t}\mu\sqrt{M\textsc{snr}}, \frac{\sqrt{2}}{1-t}\}$. 
Further, suppose $x \in \mathbb{C}^N$ has support $\mathcal{K}$ drawn uniformly at random from all possible $K$-subsets of $\{1,\ldots,N\}$.
Then provided
\begin{equation}
\label{thmeqn:RSP}
K \leq \frac{N}{c_1^2\|\Phi\|_2^2\log{N}},
\end{equation}
Algorithm~\ref{alg:OST_recon} produces $\hat{\mathcal{K}}$ such that $\mathcal{T}_\sigma(t) \cap \mathcal{T}_\mu(t) \subseteq \hat{\mathcal{K}} \subseteq \mathcal{K}$ and $\hat{x}$ such that
\begin{equation}
\label{thmeqn:RSP_2}
\|x-\hat{x}\| \leq c_2 \sqrt{\sigma^2|\hat{\mathcal{K}}|\log{N}} + c_3\|x_{\mathcal{K} \setminus \hat{\mathcal{K}}}\|
\end{equation}
with probability exceeding $1 - 10N^{-1}$. Finally, defining $T:=|\mathcal{T}_\sigma(t) \cap \mathcal{T}_\mu(t)|$, we further have
\begin{equation}
\label{thmeqn:RSP_3}
\|x-\hat{x}\| \leq c_2 \sqrt{\sigma^2 K \log{N}} + c_3\|x - x_T\|
\end{equation}
in the same probability event.
Here, $c_1 = 37e$, $c_2 = \frac{2}{1-e^{-1/2}}$, and $c_3 = 1 + \frac{e^{-1/2}}{1-e^{-1/2}}$ are numerical constants.
\end{thm}

\begin{proof}
To begin, note that since $\|\Phi\|_2^2\geq\frac{N}{M}$, we have from \eqref{thmeqn:RSP} that $K\leq M/(2\log{N})$. 
It is then easy to conclude from Theorem~5 of~\cite{bajwa:jcn10} that $\hat{\mathcal{K}}$ satisfies $\mathcal{T}_\sigma(t) \cap \mathcal{T}_\mu(t) \subseteq \hat{\mathcal{K}} \subseteq \mathcal{K}$ with probability exceeding $1 - 6N^{-1}$. Therefore, conditioned on the event $\mathcal{E}_1 := \{\mathcal{T}_\sigma(t) \cap \mathcal{T}_\mu(t) \subseteq \hat{\mathcal{K}} \subseteq \mathcal{K}\}$, we can make use of the triangle inequality to write
\begin{equation}
\label{pfeqn:RSP_1}
\|x - \hat{x}\|\leq \|x_{\hat{\mathcal{K}}} - \hat{x}_{\hat{\mathcal{K}}}\| + \|x_{\mathcal{K}\setminus\hat{\mathcal{K}}}\|.
\end{equation}
Next, we may use \eqref{thmeqn:RSP} and the fact that $\Phi$ satisfies the strong coherence property to conclude from~\cite{tropp:cras08} (see, e.g., Proposition~3 of~\cite{bajwa:jcn10}) that 
$\|\Phi_\mathcal{K}^*\Phi_\mathcal{K}^{}-I_K\|_2 < e^{-1/2}$ with probability exceeding $1 - 2N^{-1}$. 
Hence, conditioning on $\mathcal{E}_1$ and $\mathcal{E}_2 := \{\|\Phi_\mathcal{K}^*\Phi_\mathcal{K}^{}-I_K\|_2 < e^{-1/2}\}$, we have that $(\Phi_{\hat{\mathcal{K}}})^\dagger = (\Phi_{\hat{\mathcal{K}}}^* \Phi_{\hat{\mathcal{K}}}^{})^{-1} \Phi_{\hat{\mathcal{K}}}^*$ since $\Phi_{\hat{\mathcal{K}}}$ is a submatrix of a full column rank matrix $\Phi_\mathcal{K}$.
Therefore, given $\mathcal{E}_1$ and $\mathcal{E}_2$, we may write
\begin{equation}
\label{pfeqn:RSP_1.5}
\hat{x}_{\hat{\mathcal{K}}} 
= (\Phi_{\hat{\mathcal{K}}})^\dagger (\Phi x+z) 
=  x_{\hat{\mathcal{K}}} + (\Phi_{\hat{\mathcal{K}}})^\dagger \Phi_{\mathcal{K} \setminus \hat{\mathcal{K}}}x_{\mathcal{K} \setminus \hat{\mathcal{K}}} + (\Phi_{\hat{\mathcal{K}}})^\dagger z,
\end{equation}
and so substituting \eqref{pfeqn:RSP_1.5} into \eqref{pfeqn:RSP_1} and applying the triangle inequality gives
\begin{align}
\nonumber
\|x - \hat{x}\| 
&\leq \|(\Phi_{\hat{\mathcal{K}}})^\dagger \Phi_{\mathcal{K} \setminus \hat{\mathcal{K}}}x_{\mathcal{K} \setminus \hat{\mathcal{K}}}\| + \|(\Phi_{\hat{\mathcal{K}}})^\dagger z\| + \|x_{\mathcal{K} \setminus \hat{\mathcal{K}}}\|\\
\label{pfeqn:RSP_2}
&\leq \Big(1 + \|(\Phi_{\hat{\mathcal{K}}}^* \Phi_{\hat{\mathcal{K}}}^{})^{-1}\|_2 \|\Phi_{\hat{\mathcal{K}}}^* \Phi_{\mathcal{K} \setminus \hat{\mathcal{K}}}^{}\|_2\Big)\|x_{\mathcal{K} \setminus \hat{\mathcal{K}}}^{}\| + \|(\Phi_{\hat{\mathcal{K}}}^* \Phi_{\hat{\mathcal{K}}}^{})^{-1}\|_2 \|\Phi_{\hat{\mathcal{K}}}^* z\|.
\end{align}
Since, given $\mathcal{E}_1$, we have that $\Phi_{\hat{\mathcal{K}}}^* \Phi_{\hat{\mathcal{K}}}^{} - I_K$ and $\Phi_{\hat{\mathcal{K}}}^* \Phi_{\mathcal{K} \setminus \hat{\mathcal{K}}}^{}$ are submatrices of $\Phi_\mathcal{K}^* \Phi_\mathcal{K}^{} - I_K$, and since the spectral norm of a matrix provides an upper bound for the spectral norms of its submatrices, we have the following given $\mathcal{E}_1$ and $\mathcal{E}_2$:
$\|\Phi_{\hat{\mathcal{K}}}^* \Phi_{\mathcal{K} \setminus \hat{\mathcal{K}}}^{}\|_2 \leq e^{-1/2}$
and
$\|(\Phi_{\hat{\mathcal{K}}}^* \Phi_{\hat{\mathcal{K}}}^{})^{-1}\|_2 \leq \frac{1}{1-e^{-1/2}}$.
We can now substitute these bounds into \eqref{pfeqn:RSP_2} and make use of the fact that $\|\Phi_{\hat{\mathcal{K}}}^* z\| \leq |\hat{\mathcal{K}}|^{1/2}\|\Phi_{\hat{\mathcal{K}}}^* z\|_\infty$ to conclude that
\begin{equation*}
\|x - \hat{x}\| \leq \frac{|\hat{\mathcal{K}}|^{1/2}}{1-e^{-1/2}} \|\Phi_{\hat{\mathcal{K}}}^* z\|_\infty + \Big(1 + \frac{e^{-1/2}}{1-e^{-1/2}}\Big)\|x_{\mathcal{K} \setminus \hat{\mathcal{K}}}\|,
\end{equation*}
given $\mathcal{E}_1$ and $\mathcal{E}_2$. 
At this point, define the event $\mathcal{E}_3 = \{\|\Phi_{\hat{\mathcal{K}}}^* z\|_\infty \leq 2\sqrt{\sigma^2 \log{N}}\}$ and note from Lemma~6 of~\cite{bajwa:jcn10} that $\Pr(\mathcal{E}_3^\mathrm{c}) \leq 2(\sqrt{2\pi\log{N}}~N)^{-1}$. 
A union bound therefore gives \eqref{thmeqn:RSP_2} with probability exceeding $1 - 10N^{-1}$. 
For \eqref{thmeqn:RSP_3}, note that $\hat{\mathcal{K}} \subseteq \mathcal{K}$ implies $|\hat{\mathcal{K}}| \leq K$, and so $\mathcal{T}_\sigma(t) \cap \mathcal{T}_\mu(t) \subseteq \hat{\mathcal{K}}$ implies that $\|x_{\mathcal{K} \setminus \hat{\mathcal{K}}}\| \leq \|x_{\mathcal{K} \setminus (\mathcal{T}_\sigma(t) \cap \mathcal{T}_\mu(t))}\| = \|x - x_T\|$.
\end{proof}

A few remarks are in order now for Theorem~\ref{thm:RSP}. First, if $\Phi$
satisfies the strong coherence property \emph{and} $\Phi$ is nearly tight, then
OST handles sparsity that is almost linear in $M$: $K = \mathrm{O}(M/\log{N})$ from
\eqref{thmeqn:RSP}. 
Second, we do not impose any control over the size of $T$, but rather we state the result in generality in terms of $T$; its size is determined by the signal class $x$ belongs to, the worst-case coherence of the frame $\Phi$ we use to measure $x$, and the magnitude of the noise that perturbs $\Phi x$.
Third, the $\ell_2$ error associated with the OST
algorithm is the near-optimal (modulo the $\log$ factor) error of
$\sqrt{\sigma^2 K \log{N}}$ \emph{plus} the best $T$-term approximation
error caused by the inability of the OST algorithm to recover signal entries
that are smaller than $\mathrm{O}(\mu\|x\|\sqrt{2 \log{N}})$. 
In particular, if the $K$-sparse signal $x$, the worst-case coherence $\mu$, and the noise $z$ together satisfy $\|x - x_T\| = \mathrm{O}(\sqrt{\sigma^2 K \log{N}})$,
then the OST algorithm succeeds with a near-optimal $\ell_2$ error of
$\|x-\hat{x}\| = \mathrm{O}(\sqrt{\sigma^2 K \log{N}})$. 
To see why this error is near-optimal, note that a $K$-dimension vector of random entries with mean zero and variance $\sigma^2$ has expected squared norm $\sigma^2 K$; in our case, we pay an additional log factor to find the locations of the $K$ nonzero entries among the entire $N$-dimensional signal.
It is important to recognize that the optimality condition $\|x - x_T\| = \mathrm{O}(\sqrt{\sigma^2 K \log{N}})$
depends on the signal class, the noise variance, and the worst-case coherence of the frame; in particular, the condition is satisfied whenever $\|x_{\mathcal{K} \setminus
\mathcal{T}_\mu(t)}\| = \mathrm{O}(\sqrt{\sigma^2 K \log{N}})$, since
\begin{equation*}
    \|x - x_T\| \leq \|x_{\mathcal{K} \setminus \mathcal{T}_\sigma(t)}\| +
\|x_{\mathcal{K} \setminus \mathcal{T}_\mu(t)}\| = \mathrm{O}\Big(\sqrt{\sigma^2
K \log{N}}\Big) + \|x_{\mathcal{K} \setminus \mathcal{T}_\mu(t)}\|.
\end{equation*}
The following lemma provides classes of sparse signals that satisfy
$\|x_{\mathcal{K} \setminus \mathcal{T}_\mu(t)}\| =
\mathrm{O}(\sqrt{\sigma^2 K \log{N}})$
given sufficiently small noise variance and worst-case coherence, and consequently the OST
algorithm is near-optimal for the reconstruction of such signal classes.

\begin{lem}\label{lem:OST_opt_cond}
Take an $M \times N$ unit norm frame $\Phi$ with worst-case coherence
$\mu\leq\frac{c_0}{\sqrt{M}}$ for some $c_0>0$, and suppose that $K\leq\frac{N}{c_1^2\|\Phi\|_2^2\log N}$ for some $c_1>0$.
Fix a constant $\beta \in (0,1]$, and suppose the magnitudes of $\beta K$ nonzero entries of $x$ are some $\alpha =
\Omega(\sqrt{\sigma^2 \log{N}})$, while the magnitudes of the remaining
$(1-\beta)K$ nonzero entries are not necessarily same, but are smaller than $\alpha$ and scale as $\mathrm{O}(\sqrt{\sigma^2 \log{N}})$. 
Then $\|x_{\mathcal{K} \setminus \mathcal{T}_\mu(t)}\| = \mathrm{O}(\sqrt{\sigma^2 K \log{N}})$, provided $c_0\leq\frac{tc_1}{20\sqrt{2}}$.
\end{lem}

\begin{proof}
Let $\mathcal{K}$ be the support of $x$, and define $\mathcal{I} := \{n : |x_n| = \alpha\}$.
We wish to show that $\mathcal{I}\subseteq\mathcal{T}_\mu(t)$, since this implies $\|x_{\mathcal{K} \setminus
\mathcal{T}_\mu(t)}\| \leq \|x_{\mathcal{K} \setminus \mathcal{I}}\| =
\mathrm{O}(\sqrt{\sigma^2 K\log{N}})$.
In order to prove $\mathcal{I}\subseteq\mathcal{T}_\mu(t)$, notice that
\begin{equation*}
\|x\|^2 
= \|x_{\mathcal{I}}\|^2+\|x_{\mathcal{K}\setminus\mathcal{I}}\|^2 
< \beta K \alpha^2 + (1-\beta)K\alpha^2 
= K\alpha^2,
\end{equation*}
and so combining this with the fact that $\|\Phi\|_2^2\geq\frac{N}{M}$ gives
\begin{equation*}
\mu \|x\| \sqrt{\log{N}}
< \frac{c_0}{\sqrt{M}} \sqrt{K} \alpha \sqrt{\log{N}}
\leq \frac{c_0}{\sqrt{M}} \sqrt{\frac{N}{c_1^2\|\Phi\|_2^2\log N}} ~\alpha\sqrt{\log{N}}
\leq \frac{c_0}{c_1}\alpha.
\end{equation*}
Therefore, provided $c_0\leq\frac{tc_1}{20\sqrt{2}}$, we have that $\mathcal{I}\subseteq\mathcal{T}_\mu(t)$.
\end{proof}

In words, Lemma~\ref{lem:OST_opt_cond} implies that OST is near-optimal for
those $K$-sparse signals whose entries above the noise floor have roughly the
same magnitude. This subsumes a very important class of signals
that appears in applications such as multi-label prediction
\cite{HsuNips2009}, in which all the nonzero entries take values $\pm \alpha$.
Theorem~\ref{thm:RSP} is the first result in the sparse signal
processing literature that does not require RIP and still provides near-optimal
reconstruction guarantees for such signals from noisy measurements, while
using either random or deterministic frames, even when $K = \mathrm{O}(M/\log{N})$.

Note that our techniques can be extended to reconstruct noisy signals, that is, we may consider measurements of the form $y=\Phi(x+n)+z$, where $n\in\mathbb{C}^N$ is also a noise vector of independent, identical zero-mean complex-Gaussian random variables.
In particular, if the frame $\Phi$ is tight, then our measurements will not color the noise, and so noise in the signal may be viewed as noise in the measurements: $y=\Phi x+(\Phi n+z)$; if the frame is not tight, then the noise will become correlated in the measurements, and performance would be depend nontrivially on the frame's Gram matrix.
Also, Theorem~\ref{thm:RSP} can be generalized to approximately sparse signals; the analysis follows similiar lines, but is rather cumbersome, and it appears as though the end result is only strong enough in the case of very nearly sparse signals.
As such, we omit this result.

\section{Frame constructions}

In this section, we consider a range of nearly tight frames with small worst-case and average coherence.  
We investigate various ways of selecting frames at random from different libraries, and we show that for each of these frames, the spectral norm, worst-case coherence, and average coherence are all small with high probability.
Later, we will consider deterministic constructions that use Gabor and chirp systems, spherical designs, equiangular tight frames, and error-correcting codes.
For the reader's convenience, all of these constructions are summarized in Table~\ref{table.constructions}.
Before we go any further, we consider the following lemma, which gives three different sufficient conditions for a frame to satisfy (SCP-2).
These conditions will prove quite useful in this section and throughout the chapter.

\begin{lem} 
\label{lem.sufficient conditions}
For any $M\times N$ unit norm frame $\Phi$, each of the following conditions implies $\nu\leq\frac{\mu}{\sqrt{M}}$:
\begin{enumerate}
\item[(i)] $\langle \varphi_k,\sum_{n=1}^N \varphi_n\rangle=\frac{N}{M}$ for every $k=1,\ldots,N$,
\item[(ii)] $N\geq2M$ and $\sum_{n=1}^N \varphi_n=0$,
\item[(iii)] $N\geq M^2+3M+3$ and $\|\sum_{n=1}^N \varphi_n\|^2\leq N$.
\end{enumerate}
\end{lem}

\begin{proof}
For condition (i), we have
\begin{equation*}
\nu
=\frac{1}{N-1}\max_{i\in\{1,\ldots,N\}}\bigg|\sum_{\substack{j=1\\j\neq i}}^N\langle \varphi_i,\varphi_j\rangle\bigg|
=\frac{1}{N-1}\max_{i\in\{1,\ldots,N\}}\bigg|\bigg\langle \varphi_i,\sum_{j=1}^N \varphi_j\bigg\rangle-1\bigg|
=\frac{1}{N-1}\bigg(\frac{N}{M}-1\bigg).
\end{equation*}
The Welch bound (Theorem~\ref{thm.welch bound}) therefore gives
$\nu=\frac{1}{N-1}\big(\frac{N}{M}-1\big)=\frac{N-M}{M(N-1)}\leq\mu\sqrt{\frac{N-M}{M(N-1)}}
\leq\frac{\mu}{\sqrt{M}}$.
For condition (ii), we have
\begin{equation*}
\nu
=\frac{1}{N-1}\max_{i\in\{1,\ldots,N\}}\bigg|\sum_{\substack{j=1\\j\neq i}}^N\langle \varphi_i,\varphi_j\rangle\bigg|
=\frac{1}{N-1}\max_{i\in\{1,\ldots,N\}}\bigg|\bigg\langle \varphi_i,\sum_{j=1}^N \varphi_j\bigg\rangle-1\bigg|
=\frac{1}{N-1}.
\end{equation*}
Considering the Welch bound, it suffices to show $\frac{1}{N-1}\leq\frac{1}{\sqrt{M}}\sqrt{\frac{N-M}{M(N-1)}}$.
Rearranging gives
\begin{equation}
\label{eq.lemma ii.2}
N^2-(M+1)N-M(M-1)\geq0.
\end{equation}
When $N=2M$, the left-hand side of \eqref{eq.lemma ii.2} becomes $(M-1)^2$, which is trivially nonnegative.  Otherwise, we have
\begin{equation*}
N\geq2M+1\geq M+1+\sqrt{M(M-1)}\geq\frac{M+1}{2}+\sqrt{\Big(\frac{M+1}{2}\Big)^2+M(M-1)}.
\end{equation*}
In this case, by the quadratic formula and the fact that the left-hand side of \eqref{eq.lemma ii.2} is concave up in $N$, we have that \eqref{eq.lemma ii.2} is indeed satisfied.
For condition (iii), we use the triangle and Cauchy-Schwarz inequalities to get
\begin{equation*}
\nu
=\frac{1}{N-1}\max_{i\in\{1,\ldots,N\}}\bigg|\bigg\langle \varphi_i,\sum_{j=1}^N \varphi_j\bigg\rangle-1\bigg|
\leq\frac{1}{N-1}\bigg(\max_{i\in\{1,\ldots,N\}}\bigg|\bigg\langle \varphi_i,\sum_{j=1}^N \varphi_j\bigg\rangle\bigg|+1\bigg)
\leq\frac{\sqrt{N}+1}{N-1}.
\end{equation*}
Considering the Welch bound, it suffices to show $\frac{\sqrt{N}+1}{N-1}\leq\frac{1}{\sqrt{M}}\sqrt{\frac{N-M}{M(N-1)}}$.
Taking $x:=\sqrt{N}$ and rearranging gives a polynomial: 
$x^4-(M^2+M+1)x^2-2M^2x-M(M-1)\geq0$.
By convexity and monotonicity of the polynomial in $[M+\frac{3}{2},\infty)$, it can be shown that the largest real root of this polynomial is always smaller than $M+\frac{3}{2}$.
Also, considering it is concave up in $x$, it suffices that $\sqrt{N}=x\geq M+\frac{3}{2}$, which we have since
$N\geq M^2+3M+3\geq(M+\frac{3}{2})^2$.
\end{proof}

\subsection{Normalized Gaussian frames}

Construct a matrix with independent, Gaussian-distributed entries that have zero mean and unit variance.
By normalizing the columns, we get a matrix called a \emph{normalized Gaussian frame}.
This is perhaps the most widely studied type of frame in the signal processing and statistics literature.
To be clear, the term ``normalized'' is intended to distinguish the results presented here from results reported in earlier works, such as \cite{bajwa:jcn10,BaraniukDDW:08,CandesT:05,wainwright:tit09}, 
which only ensure that Gaussian frame elements have unit norm in expectation. 
In other words, normalized Gaussian frame elements are independently and uniformly distributed on the unit hypersphere in $\mathbb{R}^M$.
The following theorem characterizes the spectral norm and the worst-case and average coherence of normalized Gaussian frames.

\begin{thm}[Geometry of normalized Gaussian frames]
\label{thm.normalized gaussian frames}
Build a real $M\times N$ frame $\Psi$ by drawing entries independently at random from a Gaussian distribution of zero mean and unit variance.
Next, construct a normalized Gaussian frame $\Phi$ by taking $\varphi_n:=\frac{\psi_n}{\|\psi_n\|}$ for every $n=1,\ldots,N$.
Provided $60\log{N}\leq M\leq\frac{N-1}{4\log{N}}$, then the following simultaneously hold with probability exceeding $1 - 11N^{-1}$:
\begin{enumerate}
\item[(i)] $\mu \leq \frac{\sqrt{15\log{N}}}{\sqrt{M} - \sqrt{12\log{N}}}$,
\item[(ii)] $\nu \leq \frac{\sqrt{15\log{N}}}{M - \sqrt{12M\log{N}}}$,
\item[(iii)] $\|\Phi\|_2 \leq \frac{\sqrt{M} + \sqrt{N} + \sqrt{2\log{N}}}{\sqrt{M - \sqrt{8M\log{N}}}}$.
\end{enumerate}
\end{thm}

\begin{proof}
Theorem~\ref{thm.normalized gaussian frames}(i) can be shown to hold with probability exceeding $1 - 2N^{-1}$ by using a bound on the norm of a Gaussian random vector in Lemma~1 of~\cite{LaurentM:00} and a bound on the magnitude of the inner product of two independent Gaussian random vectors in Lemma~6 of~\cite{haupt:tit10}. Specifically, pick any two distinct indices $i, j \in \{1,\dots,N\}$, and define probability events $\mathcal{E}_1 := \{|\langle \psi_i,\psi_j\rangle| \leq \varepsilon_1\}$, $\mathcal{E}_2 := \{\|\psi_i\|^2 \geq M(1 - \varepsilon_2)\}$, and $\mathcal{E}_3 := \{\|\psi_j\|^2 \geq M(1 - \varepsilon_2)\}$ for $\varepsilon_1 = \sqrt{15M\log{N}}$ and $\varepsilon_2 = \sqrt{(12\log{N})/M}$. Then it follows from the union bound that
\begin{equation*}
\Pr\bigg(|\langle \varphi_i,\varphi_j\rangle| > \frac{\varepsilon_1}{M(1-\varepsilon_2)} \bigg)
= \Pr\bigg(\frac{|\langle \psi_i,\psi_j\rangle|}{\|\psi_i\|\|\psi_j\|} > \frac{\varepsilon_1}{M(1-\varepsilon_2)} \bigg)
\leq \Pr(\mathcal{E}_1^\mathrm{c}) + \Pr(\mathcal{E}_2^\mathrm{c}) + \Pr(\mathcal{E}_3^\mathrm{c}).
\end{equation*}
One can verify that $\Pr(\mathcal{E}_2^\mathrm{c}) = \Pr(\mathcal{E}_3^\mathrm{c}) \leq N^{-3}$ because of Lemma~1 of~\cite{LaurentM:00}, and we further have $\Pr(\mathcal{E}_1^\mathrm{c}) \leq 2N^{-3}$ because of Lemma~6 of~\cite{haupt:tit10} and the fact that $M \geq 60\log{N}$. 
Thus, for any fixed $i$ and $j$, 
$|\langle \varphi_i,\varphi_j\rangle| \leq \sqrt{15\log{N}}/(\sqrt{M} - \sqrt{12\log{N}})$
with probability exceeding $1 - 4N^{-3}$. 
It therefore follows by taking a union bound over all $\binom{N}{2}$ choices for $i$ and $j$ that Theorem~\ref{thm.normalized gaussian frames}(i) holds with probability exceeding $1 - 2N^{-1}$.

Theorem~\ref{thm.normalized gaussian frames}(ii) can be shown to hold with probability exceeding $1 - 6N^{-1}$ by appealing to the preceding analysis and Hoeffding's inequality for a sum of independent, bounded random variables~\cite{hoeffding:jasa63}. Specifically, fix any index $i \in \{1,\dots,N\}$, and define random variables $Z_{ij} := \frac{1}{N-1}\langle \varphi_i, \varphi_j\rangle$. Next, define the probability event 
\begin{equation*}
\mathcal{E}_4 := \bigcap_{\substack{j=1\\j\neq i}}^N\bigg\{|Z_{ij}| \leq \frac{1}{N-1}\frac{\sqrt{15\log{N}}}{\sqrt{M} - \sqrt{12\log{N}}}\bigg\}. 
\end{equation*}
Using the analysis for the worst-case coherence of $\Phi$ and taking a union bound over the $N-1$ possible $j$'s gives $\Pr(\mathcal{E}_4^\mathrm{c}) \leq 4N^{-2}$. 
Furthermore, taking $\varepsilon_3 := \sqrt{15\log{N}}/(M - \sqrt{12M\log{N}})$, then elementary probability analysis gives
\begin{align}
\nonumber
\Pr\Bigg(\bigg|\sum_{\substack{j=1\\j \not= i}}^{N} Z_{ij}\bigg| > \varepsilon_3\Bigg) 
&\leq \Pr\Bigg(\bigg|\sum_{\substack{j=1\\j \not= i}}^{N} Z_{ij}\bigg| > \varepsilon_3 ~\Bigg|~ \mathcal{E}_4\Bigg) + \Pr(\mathcal{E}_4^\mathrm{c})\\
\label{pfeqn:gaussian_avc1}
&\leq \int_{\mathbb{S}^{M-1}} \!\!\! \Pr\Bigg(\bigg|\sum_{\substack{j=1\\j \not= i}}^{N} Z_{ij}\bigg| > \varepsilon_3 ~\Bigg|~ \mathcal{E}_4, \varphi_i=x\Bigg)~p_{\varphi_i}(x)~\mathrm{dH}^{M-1}(x) + 4N^{-2},
\end{align}
where $\mathbb{S}^{M-1}$ denotes the unit hypersphere in $\mathbb{R}^M$, $\mathrm{H}^{M-1}$ denotes the $(M-1)$-dimensional Hausdorff measure on $\mathbb{S}^{M-1}$, and $p_{\varphi_i}(x)$ denotes the probability density function for the random vector $\varphi_i$. The first thing to note here is that the random variables $\{Z_{ij}: j\not=i\}$ are bounded and jointly independent when conditioned on $\mathcal{E}_4$ and $\varphi_i$. This assertion mainly follows from Bayes' rule and the fact that $\{\varphi_j:j\not=i\}$ are jointly independent when conditioned on $\varphi_i$. The second thing to note is that $\mathbb{E}[Z_{ij}~|~\mathcal{E}_4, \varphi_i] = 0$ for every $j\neq i$. This comes from the fact that the random vectors $\{\varphi_n\}_{n=1}^N$ are independent and have a uniform distribution over $\mathbb{S}^{M-1}$, which in turn guarantees that the random variables $\{Z_{ij}: j\not=i\}$ have a symmetric distribution around zero when conditioned on $\mathcal{E}_4$ and $\varphi_i$. We can therefore make use of Hoeffding's inequality \cite{hoeffding:jasa63} to bound the probability expression inside the integral in \eqref{pfeqn:gaussian_avc1} as
\begin{equation}
\label{eq.prob sum bound}
\Pr\Bigg(\bigg|\sum_{\substack{j=1\\j \not= i}}^{N} Z_{ij}\bigg| > \varepsilon_3 ~\Bigg|~ \mathcal{E}_4, \varphi_i=x\Bigg)
\leq 2e^{-(N-1)/2M},
\end{equation}
which is bounded above by $2N^{-2}$ provided $M \leq \frac{N-1}{4\log{N}}$. We can now substitute \eqref{eq.prob sum bound} into \eqref{pfeqn:gaussian_avc1} and take the union bound over the $N$ possible choices for $i$ to conclude that Theorem~\ref{thm.normalized gaussian frames}(ii) holds with probability exceeding $1 - 6N^{-1}$.

Lastly, Theorem~\ref{thm.normalized gaussian frames}(iii) can be shown to hold with probability exceeding $1 - 3N^{-1}$ by using a bound on the spectral norm of standard Gaussian random matrices reported in~\cite{rudelson:icm10} along with Lemma~1 of~\cite{LaurentM:00}. Specifically, define an $N \times N$ diagonal matrix $D := \mathrm{diag}(\|\psi_1\|^{-1}, \dots, \|\psi_N\|^{-1})$, and note that the entries of $\Psi := \Phi D^{-1}$ are independently and normally distributed with zero mean and unit variance. We therefore have from (2.3) in \cite{rudelson:icm10} that
\begin{equation}
\label{pfeqn:gaussian_spnorm1}
    \Pr\Big(\|\Psi\|_2 > \sqrt{M} + \sqrt{N} + \sqrt{2\log{N}}\Big) \leq 2 N^{-1}.
\end{equation}
In addition, we can appeal to the preceding analysis for the probability bound on Theorem~\ref{thm.normalized gaussian frames}(i) and conclude using Lemma~1 of~\cite{LaurentM:00} and a union bound over the $N$ possible choices for $i$ that
\begin{equation}
\label{pfeqn:gaussian_spnorm2}
    \Pr\Big(\|D\|_2 > \Big(M - \sqrt{8M\log{N}}\Big)^{-1/2}\Big) \leq N^{-1}.
\end{equation}
Finally, since $\|\Phi\|_2 \leq \|\Psi\|_2 \|D\|_2$, we can take a union bound over \eqref{pfeqn:gaussian_spnorm1} and \eqref{pfeqn:gaussian_spnorm2} to argue that Theorem~\ref{thm.normalized gaussian frames}(iii) holds with probability exceeding $1 - 3N^{-1}$. 

The complete result now follows by taking a union bound over the failure probabilities for the conditions (i)-(iii) in Theorem~\ref{thm.normalized gaussian frames}.
\end{proof}

\begin{exmp}
To illustrate the bounds in Theorem~\ref{thm.normalized gaussian frames}, we ran simulations in MATLAB.
Picking $N=50000$, we observed $30$ realizations of normalized Gaussian frames for each $M=700,900,1100$.
The distributions of $\mu$, $\nu$, and $\|\Phi\|_2$ were rather tight, so we only report the ranges of values attained, along with the bounds given in Theorem~\ref{thm.normalized gaussian frames}:
\begin{equation*}
\begin{array}{rrcll}
M=700: &  \qquad\mu & \in & [0.1849,0.2072]                & \qquad\leq0.8458 \\
       &  \qquad\nu & \in & [0.5643,0.6613]\times10^{-3}  & \qquad\leq0.0320 \\
       &  \qquad\|\Phi\|_2 & \in & [8.0521,8.0835]              & \qquad\leq11.9565\\
\\
M=900: &  \qquad\mu & \in & [0.1946,0.2206]                & \qquad\leq0.6848 \\
       &  \qquad\nu & \in & [0.5800,0.7501]\times10^{-3}  & \qquad\leq0.0229 \\
       &  \qquad\|\Phi\|_2 & \in & [8.4352,8.4617]              & \qquad\leq10.3645\\
\\
M=1100: &  \qquad\mu & \in & [0.1807,0.1988]                & \qquad\leq0.5852 \\
       &  \qquad\nu & \in & [0.5260,0.6713]\times10^{-3}  & \qquad\leq0.0177 \\
       &  \qquad\|\Phi\|_2 & \in & [7.7262,7.7492]              & \qquad\leq9.2927
\end{array}
\end{equation*}
These simulations seem to indicate that our bounds on $\mu$ and $\|\Phi\|_2$ reflect real-world behavior, at least within an order of magnitude, whereas the bound on $\nu$ is rather loose.
\end{exmp}

\subsection{Random harmonic frames}

Random harmonic frames, constructed by randomly selecting rows of a discrete Fourier transform (DFT) matrix and normalizing the resulting columns, have received considerable attention lately in the compressed sensing literature \cite{CandesRT:06,CandesT:06,RudelsonV:08}.
However, there is no result in the literature that gives the worst-case coherence of random harmonic frames.
To fill this gap, the following theorem gives the spectral norm and the worst-case and average coherence of random harmonic frames.

\begin{thm}[Geometry of random harmonic frames]
\label{thm.random harmonic frames}
Let $F$ be an $N\times N$ non-normalized discrete Fourier transform matrix, explicitly, $F_{k\ell}:= e^{2\pi\mathrm{i}k\ell/N}$ for each $k,\ell=0,\ldots,N-1$.
Next, let $\{B_i\}_{i=0}^{N-1}$ be a collection of independent Bernoulli random variables with mean $\frac{M}{N}$, and take $\mathcal{M}:=\{i:B_i=1\}$.
Finally, construct an $|\mathcal{M}|\times N$ harmonic frame $\Phi$ by collecting rows of $F$ which correspond to indices in $\mathcal{M}$ and normalizing the columns.
Then $\Phi$ is a unit norm tight frame: $\|\Phi\|_2^2=\frac{N}{|\mathcal{M}|}$.
Also, provided $16\log{N}\leq M\leq \frac{N}{3}$, the following simultaneously hold with probability exceeding $1 - 4N^{-1} - N^{-2}$:
\begin{enumerate}
\item[(i)] $\frac{1}{2}M \leq |\mathcal{M}| \leq \frac{3}{2}M$,
\item[(ii)] $\nu\leq\frac{\mu}{\sqrt{|\mathcal{M}|}}$,
\item[(iii)] $\mu \leq \sqrt{\frac{118(N-M)\log{N}}{MN}}$.
\end{enumerate}
\end{thm}

\begin{proof}
The claim that $\Phi$ is tight follows trivially from the fact that the rows of $F$ are orthogonal and that the rows of $\Phi$ correspond to a subset of the rows of $F$. 
Next, we define the probability events $\mathcal{E}_1 := \{|\mathcal{M}| \leq \frac{3}{2}M\}$ and $\mathcal{E}_2 := \{|\mathcal{M}| \geq \frac{1}{2}M\}$, and claim that $\Pr(\mathcal{E}_1^\mathrm{c} \cup \mathcal{E}_2^\mathrm{c}) \leq N^{-1} + N^{-2}$. The proof of this claim follows from a Bernstein-like large deviation inequality. 
Specifically, note that $|\mathcal{M}| = \sum_{i=0}^{N-1} B_i$ with $\mathbb{E}[|\mathcal{M}|] = M$, and so we have from Theorems~A.1.12 and~A.1.13 of~\cite{alon:00} and page~4 of~\cite{RudelsonV:08} that for any $\varepsilon_1 \in [0,1)$,
\begin{equation}
\label{pfeqn:size_dft}
\Pr\Big(|\mathcal{M}| > (1 + \varepsilon_1)M\Big) 
\leq e^{-M\varepsilon_1^2(1-\varepsilon_1)/2}
\qquad\mbox{and}\qquad
\Pr\Big(|\mathcal{M}| < (1 - \varepsilon_1)M\Big) 
\leq e^{-M\varepsilon_1^2/2}.
\end{equation}
Taking $\varepsilon_1 := \frac{1}{2}$, then a union bound gives $\Pr(\mathcal{E}_1^\mathrm{c} \cup \mathcal{E}_2^\mathrm{c}) \leq N^{-1} + N^{-2}$ provided $M \geq 16 \log{N}$. 
Conditioning on $\mathcal{E}_1 \cap \mathcal{E}_2$, we have that Theorem~\ref{thm.random harmonic frames}(i) holds trivially, while Theorem~\ref{thm.random harmonic frames}(ii) follows from Lemma~\ref{lem.sufficient conditions}. 
Specifically, we have that $\frac{N}{3} \geq M$ guarantees $N \geq 2|\mathcal{M}|$ because of the conditioning on $\mathcal{E}_1 \cap \mathcal{E}_2$, which in turn implies that $\Phi$ satisfies either condition (i) or (ii) of Lemma~\ref{lem.sufficient conditions}, depending on whether $0\in\mathcal{M}$. 
This therefore establishes that Theorem~\ref{thm.random harmonic frames}(i)-(ii) simultaneously hold with probability exceeding $1 - N^{-1} - N^{-2}$.

The only remaining claim is that $\mu \leq \varepsilon_2 := \sqrt{(118(N-M)\log{N})/MN}$ with high probability. 
To this end, define $p:=\frac{M}{N}$, and pick any two distinct indices $i, j \in \{0,\dots,N-1\}$.
Note that
\begin{equation}
\label{eq.ip of random dft fs}
\langle \varphi_i,\varphi_j\rangle 
=\frac{1}{|\mathcal{M}|}\sum_{k=0}^{N-1} B_kF_{ki}\overline{F_{kj}} 
=\frac{1}{|\mathcal{M}|}\sum_{k=0}^{N-1} (B_k-p)F_{ki}\overline{F_{kj}},
\end{equation}
where the last equality follows from the fact that $F$ has orthogonal columns.
Next, we write $F_{ki}\overline{F_{kj}}=\cos(\theta_k)+\mathrm{i}\sin(\theta_k)$ for some $\theta_k\in[0,2\pi)$.
Then applying the union bound to \eqref{eq.ip of random dft fs} and to the real and imaginary parts of $F_{ki}\overline{F_{kj}}$ gives
\begin{align}
\nonumber
&\Pr\Big(|\langle \varphi_i,\varphi_j\rangle| > \varepsilon_2\Big) \\
\nonumber
&\leq \Pr\bigg(\Big|\sum_{k=0}^{N-1} (B_k - p)F_{ki}\overline{F_{kj}}\Big| > \frac{M\varepsilon_2}{2\sqrt{2}}\bigg) + \Pr\Big(|\mathcal{M}| < \frac{M}{2\sqrt{2}}\Big)\\
\label{pfeqn:wc_dft}
&\leq \Pr\bigg(\Big|\sum_{k=0}^{N-1} (B_k - p)\cos(\theta_k) \Big| > \frac{M\varepsilon_2}{4}\bigg) + \Pr\bigg(\Big|\sum_{k=0}^{N-1} (B_k - p)\sin(\theta_k) \Big| > \frac{M\varepsilon_2}{4}\bigg) + N^{-3},
\end{align}
where the last term follows from \eqref{pfeqn:size_dft} and the fact that $M \geq 16 \log{N}$. 
Define random variables $Z_k:=(B_k-p)\cos(\theta_k)$.
Note that the $Z_k$'s have zero mean and are jointly independent.
Also, the $Z_k$'s are bounded by $1-p$ almost surely since $|(B_k-p)\cos(\theta_k)|\leq\max\{p,1-p\}$ and $N\geq 2M$.
Moreover, the variance of each $Z_k$ is bounded: $\mathrm{Var}(Z_\ell)\leq p(1-p)$.
Therefore, we may use the Bernstein inequality for a sum of independent, bounded random variables \cite{bennett:jasa62} to bound the probability that $|\sum_{k=0}^{N-1} Z_k|$ deviates from $\varepsilon_3 := \frac{M\varepsilon_2}{4}$:
\begin{equation*}
\Pr\bigg(\Big|\sum_{k=0}^{N-1} (B_k - p)\cos(\theta_k) \Big| > \varepsilon_3\bigg)
\leq 2e^{-\varepsilon_3^2/(2Np(1-p) + 2(1-p)\varepsilon_3/3)} \leq 2N^{-3}.
\end{equation*}
Similarly, the probability that $|\sum_{k=0}^{N-1} (B_k - p)\sin(\theta_k)| > \varepsilon_3$ is also bounded above by $2N^{-3}$. 
Substituting these probability bounds into \eqref{pfeqn:wc_dft} gives $|\langle \varphi_i,\varphi_j\rangle| > \varepsilon_2$ with probability at most $5N^{-3}$ provided $M \geq 16\log{N}$. 
Finally, we take a union bound over the $\binom{N}{2}$ possible choices for $i$ and $j$ to get that Theorem~\ref{thm.random harmonic frames}(iii) holds with probability exceeding $1 - 3N^{-1}$. 

The result now follows by taking a final union bound over $\mathcal{E}_1^\mathrm{c} \cup \mathcal{E}_2^\mathrm{c}$ and $\{\mu > \varepsilon_2\}$.
\end{proof}

As stated earlier, random harmonic frames are not new to sparse signal processing.
Interestingly, for the application of compressed sensing, \cite{CandesT:05,RudelsonV:08}  provides performance guarantees for both random harmonic and Gaussian frames, but requires more rows in a random harmonic frame to accommodate the same level of sparsity.
This suggests that random harmonic frames may be inferior to Gaussian frames as compressed sensing matrices, but practice suggests otherwise \cite{donoho:ptrsa09}.
In a sense, Theorem~\ref{thm.random harmonic frames} helps to resolve this gap in understanding; there exist compressed sensing algorithms whose performance is dictated by worst-case coherence \cite{bajwa:jcn10,DonohoET:06,Tropp:04,Tropp:acha08}, and Theorem~\ref{thm.random harmonic frames} states that random harmonic frames have near-optimal worst-case coherence, being on the order of the Welch bound with an additional $\sqrt{\log N}$ factor.

\begin{exmp}
To illustrate the bounds in Theorem~\ref{thm.random harmonic frames}, we ran simulations in MATLAB.
Picking $N=5000$, we observed $30$ realizations of random harmonic frames for each $M=1000,1250,1500$.
The distributions of $|\mathcal{M}|$, $\nu$, and $\mu$ were rather tight, so we only report the ranges of values attained, along with the bounds given in Theorem~\ref{thm.random harmonic frames}.
Notice that Theorem~\ref{thm.random harmonic frames} gives a bound on $\nu$ in terms of both $|\mathcal{M}|$ and $\mu$.
To simplify matters, we show that $\nu\leq\frac{\min\mu}{\sqrt{\max|\mathcal{M}|}}\leq\frac{\mu}{\sqrt{|\mathcal{M}|}}$, where the minimum and maximum are taken over all realizations in the sample:
\begin{equation*}
\begin{array}{rrcll}
M=1000: &  \qquad|\mathcal{M}| & \in & [961,1052]                & \qquad\subseteq[500,1500] \\
        &  \qquad\nu & \in & [0.2000,0.8082]\times10^{-3}  & \qquad\leq0.0023\approx\tfrac{0.0746}{\sqrt{1052}} \\
        &  \qquad\mu & \in & [0.0746,0.0890]                & \qquad\leq0.8967 \\
\\
M=1250: &  \qquad|\mathcal{M}| & \in & [1207,1305]                & \qquad\subseteq[625,1875] \\
        &  \qquad\nu & \in & [0.2000,0.6273]\times10^{-3}  & \qquad\leq0.0018\approx\tfrac{0.0623}{\sqrt{1305}} \\
        &  \qquad\mu & \in & [0.0623,0.0774]                & \qquad\leq0.7766 \\
\\
M=1500: &  \qquad|\mathcal{M}| & \in & [1454,1590]                & \qquad\subseteq[750,2250] \\
        &  \qquad\nu & \in & [0.2000,0.4841]\times10^{-3}  & \qquad\leq0.0015\approx\tfrac{0.0571}{\sqrt{1590}} \\
        &  \qquad\mu & \in & [0.0571,0.0743]                & \qquad\leq0.6849
\end{array}
\end{equation*}
The reader may have noticed how consistently the average coherence value of $\nu\approx0.2000\times10^{-3}$ was realized.
This occurs precisely when the zeroth row of the DFT is not selected, as the frame elements sum to zero in this case:
\begin{equation*}
\nu
:=\frac{1}{N-1}\max_{i\in\{1,\ldots,N\}}\bigg|\sum_{\substack{j=1\\j\neq i}}^N\langle \varphi_i,\varphi_j\rangle\bigg|
=\frac{1}{N-1}\max_{i\in\{1,\ldots,N\}}\bigg|\bigg\langle \varphi_i,\sum_{j=1}^N\varphi_j\bigg\rangle-\|\varphi_i\|^2\bigg|
=\frac{1}{N-1}.
\end{equation*}
These simulations seem to indicate that our bounds on $|\mathcal{M}|$, $\nu$, and $\mu$ leave room for improvement.
The only bound that lies within an order of magnitude of real-world behavior is our bound on $|\mathcal{M}|$.
\end{exmp}

\subsection{Gabor and chirp frames}
Gabor frames constitute an important class of frames, as they appear in a variety of applications such as radar \cite{herman:sp09}, speech processing \cite{wolfe:spaa01}, and quantum information theory \cite{ScottG:10}.
Given a nonzero seed function $f:\mathbb{Z}_M\rightarrow\mathbb{C}$, we produce all time- and frequency-shifted versions: $f_{xy}(t):=f(t-x)e^{2\pi\mathrm{i}yt/M}$, $t\in\mathbb{Z}_M$.
Viewing these shifted functions as vectors in $\mathbb{C}^M$ gives an $M\times M^2$ Gabor frame.
The following theorem characterizes the spectral norm and the worst-case and average coherence of Gabor frames generated from either a deterministic Alltop vector~\cite{alltop:tit80} or a random Steinhaus vector.

\begin{thm}[Geometry of Gabor frames]
\label{thm.gabor}
Take an Alltop function defined by $f(t):=\frac{1}{\sqrt{M}}e^{2\pi\mathrm{i}t^3/M}$, $t\in\mathbb{Z}_M$.
Also, take a random Steinhaus function defined by $g(t):=\frac{1}{\sqrt{M}}e^{2\pi\mathrm{i}\theta_t}$, $t\in\mathbb{Z}_M$, where the $\theta_t$'s are independent random variables distributed uniformly on the unit interval.
Then the $M\times M^2$ Gabor frames $\Phi$ and $\Psi$ generated by $f$ and $g$, respectively, are unit norm and tight, i.e., $\|\Phi\|_2=\|\Psi\|_2=\sqrt{M}$.  Also, both frames have average coherence $\leq\frac{1}{M+1}$.
Furthermore, if $M\geq5$ is prime, then $\mu_\Phi=\frac{1}{\sqrt{M}}$, while if $M\geq13$, then $\mu_\Psi\leq\sqrt{(13\log{M})/M}$ with probability exceeding $1 - 4M^{-1}$.
\end{thm}

\begin{proof}
The tightness claim follows from~\cite{lawrence:jfaa05}, in which it was shown that Gabor frames generated by nonzero seed vectors are tight. The bound on average coherence is a consequence of Theorem~7 of~\cite{bajwa:jcn10} concerning arbitrary Gabor frames. The claim concerning $\mu_\Phi$ follows directly from \cite{StrohmerH:03}, while the claim concerning $\mu_\Psi$ is a simple consequence of Theorem~5.1 of~\cite{pfander:tsp08}.
\end{proof}

Instead of taking all translates and modulates of a seed function, \cite{CF06} constructs \emph{chirp frames} by taking all powers and modulates of a chirp function.
Picking $M$ to be prime, we start with a chirp function $h_M:\mathbb{Z}_M\rightarrow\mathbb{C}$ defined by $h_M(t):=e^{\pi\mathrm{i}t(t-M)/M}$, $t\in\mathbb{Z}_M$.
The $M^2$ frame elements are then defined entrywise by $h_{ab}(t):=\frac{1}{\sqrt{M}}h_M(t)^ae^{2\pi\mathrm{i}bt/M}$, $t\in\mathbb{Z}_M$.
Certainly, chirp frames are, at the very least, similar in spirit to Gabor frames.
As a matter of fact, the chirp frame is in some sense equivalent to the Gabor frame generated by the Alltop function:
it is easy to verify that $h_{(-6x,y-3x^2)}(t)=e^{2\pi\mathrm{i}(t^3+x^3)/M}f_{xy}(t)$, and when $M\geq 5$, the map $(x,y)\mapsto(-6x,y-3x^2)$ is a permutation over $\mathbb{Z}_M^2$.
Using terminology from Definition~\ref{def.flipping and wiggling}, we say the chirp frame is \emph{wiggling equivalent} to a unitary rotation of permuted Alltop Gabor frame elements.
As such, by Lemma~\ref{lem:geom_eqframes}, the chirp frame has the same spectral norm and worst-case coherence as the Alltop Gabor frame, but the average coherence may be different.
In this case, the average coherence still satisfies (SCP-2).
Indeed, adding the frame elements gives
\begin{align*}
\sum_{a=0}^{M-1}\sum_{b=0}^{M-1}h_{ab}(t)
&=\frac{1}{\sqrt{M}}\sum_{a=0}^{M-1}h_M(t)^a\sum_{b=0}^{M-1}e^{2\pi\mathrm{i}bt/M}\\
&=\frac{1}{\sqrt{M}}\sum_{a=0}^{M-1}h_M(t)^aM\delta_0(t)
=\sqrt{M}\bigg(\sum_{a=0}^{M-1}h_M(0)^a\bigg)~\delta_0(t)
=M^{3/2}\delta_0(t),
\end{align*}
and so $\langle h_{a'b'},\sum_{a=0}^{M-1}\sum_{b=0}^{M-1}h_{ab}\rangle=\langle h_{a'b'},M^{3/2}\delta_0\rangle=M^{3/2}h_{a'b'}(0)=M=\frac{M^2}{M}$.
Therefore, applying Lemma~\ref{lem.sufficient conditions}(i) gives the result:

\begin{thm}[Geometry of chirp frames]
\label{thm.chirp}
Pick $M$ prime, and let $\Phi$ be the $M\times M^2$ frame of all powers and modulates of the chirp function $h_M$.
Then $\Phi$ is a unit norm tight frame with $\|\Phi\|_2=\sqrt{M}$, and has worst case coherence $\mu=\frac{1}{\sqrt{M}}$ and average coherence $\nu\leq\frac{\mu}{\sqrt{M}}$.
\end{thm}

\begin{exmp}
To illustrate the bounds in Theorems~\ref{thm.gabor} and~\ref{thm.chirp}, we consider the examples of an Alltop Gabor frame and a chirp frame, each with $M=5$.
In this case, the Gabor frame has $\nu\approx0.1348\leq0.1667\approx\frac{1}{M+1}$, while the chirp frame has $\nu=\frac{1}{6}\leq\frac{1}{5}=\frac{\mu}{\sqrt{M}}$.
Note the Gabor and chirp frames have different average coherences despite being equivalent in some sense.
For the random Steinhaus Gabor frame, we ran simulations in MATLAB and observed $30$ realizations for each $M=60,70,80$.
The distributions of $\nu$ and $\mu$ were rather tight, so we only report the ranges of values attained, along with the bounds given in Theorem~\ref{thm.gabor}:
\begin{equation*}
\begin{array}{rrcll}
M=60:   &  \qquad\nu & \in & [0.3916,0.5958]\times10^{-2}  & \qquad\leq0.0164 \\
        &  \qquad\mu & \in & [0.3242,0.4216]                & \qquad\leq0.9419 \\
\\
M=70:   &  \qquad\nu & \in & [0.3151,0.4532]\times10^{-2}  & \qquad\leq0.0141 \\
        &  \qquad\mu & \in & [0.2989,0.3814]                & \qquad\leq0.8883 \\
\\
M=80:   &  \qquad\nu & \in & [0.2413,0.3758]\times10^{-2}  & \qquad\leq0.0124 \\
        &  \qquad\mu & \in & [0.2711,0.3796]                & \qquad\leq0.8439 
\end{array}
\end{equation*}
These simulations seem to indicate that bound on $\nu$ is conservative by an order of magnitude.
\end{exmp}

\subsection{Spherical 2-designs}

Lemma~\ref{lem.sufficient conditions}(ii) leads one to consider frames of vectors that sum to zero.
In \cite{HolmesP:04}, it is proved that real unit norm tight frames with this property make up another well-studied class of vector packings: spherical 2-designs.  To be clear, a collection of unit-norm vectors $\Phi\subseteq\mathbb{R}^M$ is called a spherical $t$-design if, for every polynomial $g(x_1,\ldots,x_M)$ of degree at most $t$, we have
\begin{equation*}
\frac{1}{\mathrm{H}^{M-1}(\mathbb{S}^{M-1})}\int_{\mathbb{S}^{M-1}}g(x)~\mathrm{d}\mathrm{H}^{M-1}(x)=\frac{1}{|\Phi|}\sum_{\varphi\in \Phi}g(\varphi),
\end{equation*}
where $\mathbb{S}^{M-1}$ is the unit hypersphere in $\mathbb{R}^M$ and $\mathrm{H}^{M-1}$ denotes the $(M-1)$-dimensional Hausdorff measure on $\mathbb{S}^{M-1}$.
In words, vectors that form a spherical $t$-design serve as good representatives when calculating the average value of a degree-$t$ polynomial over the unit hypersphere.
Today, such designs find application in quantum state estimation \cite{HHH05}.

Since real unit norm tight frames always exist for $N\geq M+1$, one might suspect that spherical 2-designs are equally common, but this intuition is faulty---the sum-to-zero condition introduces certain issues.
For example, there is no spherical 2-design when $M$ is odd and $N=M+2$.
In~\cite{M90}, spherical 2-designs are explicitly characterized by construction.
The following theorem gives a construction based on harmonic frames:

\begin{thm}[Geometry of spherical 2-designs]
\label{thm.spherical 2-designs}
Pick $M$ even and $N\geq2M$.
Take an $\frac{M}{2}\times N$ harmonic frame $\Psi$ by collecting rows from a discrete Fourier transform matrix according to a set of nonzero indices $\mathcal{M}$ and normalizing the columns.
Let $m(n)$ denote $n$th largest index in $\mathcal{M}$, and define a real $M\times N$ frame $\Phi$ by
\begin{equation*}
\Phi_{k\ell}:=\left\{ \begin{array}{ll}\sqrt{\frac{2}{M}}\cos(\frac{2\pi m((k+1)/2)\ell}{N}),&k\mbox{ odd}\\\sqrt{\frac{2}{M}}\sin(\frac{2\pi m(k/2)\ell}{N}),&k\mbox{ even}\end{array} \right., \qquad k=1,\ldots,M, ~\ell=0,\ldots,N-1.
\end{equation*}
Then $\Phi$ is unit norm and tight, i.e., $\|\Phi\|_2^2=\frac{N}{M}$, with worst-case coherence $\mu_\Phi\leq\mu_\Psi$ and average coherence $\nu\leq\frac{\mu}{\sqrt{M}}$.
\end{thm}

\begin{proof}
It is easy to verify that $\Phi$ is a unit norm tight frame using the geometric sum formula.
Also, since the frame elements sum to zero and $N\geq 2M$, the claim regarding average coherence follows from Lemma~\ref{lem.sufficient conditions}(ii).
It remains to prove $\mu_\Phi\leq\mu_\Psi$.  For each pair of indices $i,j\in\{1,\ldots,N\}$, we have
\begin{align*}
\langle \varphi_i,\varphi_j\rangle
&=\frac{2}{M}\sum_{m\in\mathcal{M}}\bigg(\cos\Big(\frac{2\pi mi}{N}\Big)\cos\Big(\frac{2\pi mj}{N}\Big)+\sin\Big(\frac{2\pi mi}{N}\Big)\sin\Big(\frac{2\pi mj}{N}\Big)\bigg)\\
&=\frac{2}{M}\sum_{m\in\mathcal{M}}\cos\Big(\frac{2\pi m(i-j)}{N}\Big)\\
&=\mathrm{Re}\langle \psi_i,\psi_j\rangle,
\end{align*}
and so $|\langle \varphi_i,\varphi_j\rangle|=|\mathrm{Re}\langle \psi_i,\psi_j\rangle|\leq|\langle \psi_i,\psi_j\rangle|$.
This gives the result.
\end{proof}

\begin{exmp}
To illustrate the bounds in Theorem~\ref{thm.spherical 2-designs}, we consider the spherical 2-design constructed from a $9\times 37$ harmonic equiangular tight frame~\cite{XiaZG:05}.
Specifically, we take a $37\times 37$ DFT matrix, choose nonzero row indices
\begin{equation*}
\mathcal{M}=\{1,7,9,10,12,16,26,33,34\},
\end{equation*}
and normalize the columns to get a harmonic frame $\Psi$ whose worst-case coherence achieves the Welch bound: $\smash{\mu_\Psi=\sqrt{\frac{37-9}{9(37-1)}}\approx0.2940}$.
Following Theorem~\ref{thm.spherical 2-designs}, we produce a spherical 2-design $\Phi$ with $\mu_\Phi\approx0.1967\leq\mu_\Psi$ and $\nu\approx0.0278\leq0.0464\approx\frac{\mu}{\sqrt{M}}$.
\end{exmp}

\subsection{Steiner equiangular tight frames}

We now consider the construction of Chapter~1: Steiner equiangular tight frames (ETFs).
Recall that these fail to break the square-root bottleneck as deterministic RIP matrices.
By contrast, Steiner ETFs are particularly well-suited as sensing matrices for one-step thresholding.
To be clear, every Steiner ETF satisfies $N\geq2M$.
Moreover, if in step (iii) of Theorem~\ref{theorem.steiner etfs}, we choose the distinct rows to be the $\frac{v-1}{k-1}$ rows of the (complex) Hadamard matrix $H$ that are not all-ones, then the sum of columns of each $F_j$ is zero, meaning the sum of columns of $F$ is also zero.
This was done in \eqref{eq.steiner etf example}, and the columns sum to zero, accordingly.
Therefore, by Lemma~\ref{lem.sufficient conditions}(ii), Steiner ETFs satisfy (SCP-2).
This gives the following theorem:

\begin{thm}[Geometry of Steiner equiangular tight frames]
\label{thm.steiner etf coherence}
Build an $M\times N$ matrix $\Phi$ according to Theorem~\ref{theorem.steiner etfs}, and in step (iii), choose rows from the (complex) Hadamard matrix $H$ that are not all-ones.
Then $\Phi$ is an equiangular tight frame, meaning $\|\Phi\|_2^2=\frac{N}{M}$ and $\mu^2=\frac{N-M}{M(N-1)}$, and has average coherence $\nu\leq\frac{\mu}{\sqrt{M}}$.
\end{thm}

\begin{exmp}
To illustrate the bound in Theorem~\ref{thm.steiner etf coherence}, we note that the example given in \eqref{eq.steiner etf example} has $\nu=\frac{1}{11}\leq\frac{1}{3\sqrt{2}}=\frac{\mu}{\sqrt{M}}$.
\end{exmp}

\subsection{Code-based frames}

Many structures in coding theory are also useful in frame theory.
In this section, we build frames from a code that originally emerged with Berlekamp in~\cite{B70}, and found recent reincarnation with~\cite{YG06}.
We build a $2^m\times2^{(t+1)m}$ frame, indexing rows by elements of $\mathbb{F}_{2^m}$ and indexing columns by $(t+1)$-tuples of elements from $\mathbb{F}_{2^m}$.
For $x\in\mathbb{F}_{2^m}$ and $\alpha\in\mathbb{F}_{2^m}^{t+1}$, the corresponding entry of the matrix $\Phi$ is given by
\begin{equation}
\label{eq.matrix defn}
\Phi_{x\alpha}=\frac{1}{\sqrt{2^{m}}}(-1)^{\mathrm{Tr}\big[\alpha_0x+\sum_{i=1}^t\alpha_ix^{2^i+1}\big]},
\end{equation}
where $\mathrm{Tr}:\mathbb{F}_{2^m}\rightarrow\mathbb{F}_2$ denotes the trace map, defined by $\mathrm{Tr}(z)=\sum_{i=0}^{m-1}z^{2^i}$.
The following theorem gives the spectral norm and the worst-case and average coherence of this frame.

\begin{table}
\begin{center}
\begin{scriptsize}
\begin{tabular}{lllll}
\hline
Name &$\mathbb{R}/\mathbb{C}$	&Size &$\mu_F$ &$\nu_F$ \\
\hline
Normalized Gaussian &$\mathbb{R}$ &$M\times N$ &$\leq \frac{\sqrt{15\log{N}}}{\sqrt{M} - \sqrt{12\log{N}}}$ &$\leq \frac{\sqrt{15\log{N}}}{M - \sqrt{12M\log{N}}}$ \\

Random harmonic &$\mathbb{C}$ &$|\mathcal{M}|\times N$, $\frac{1}{2}M\leq|\mathcal{M}|\leq\frac{3}{2}M$ &$\leq \sqrt{\frac{118(N-M)\log{N}}{MN}}$ &$\leq\frac{\mu_F}{\sqrt{|\mathcal{M}|}}$ \\

Alltop Gabor &$\mathbb{C}$ &$M\times M^2$ &$=\frac{1}{\sqrt{M}}$ &$\leq\frac{1}{M+1}$ \\

Steinhaus Gabor &$\mathbb{C}$ &$M\times M^2$ &$\leq\sqrt{\frac{13\log M}{M}}$ &$\leq\frac{1}{M+1}$ \\

Chirp &$\mathbb{C}$ &$M\times M^2$ &$=\frac{1}{\sqrt{M}}$ &$\leq\frac{\mu_F}{\sqrt{M}}$ \\

$\overset{\mbox{Spherical 2-design}}{\mbox{from harmonic }G}$ &$\mathbb{R}$ &$M\times N$ &$\leq\mu_G$ &$\leq\frac{\mu_F}{\sqrt{M}}$ \\

Steiner &$\mathbb{C}$ &$M\times N$, $M=\frac{v(v-1)}{k(k-1)}$, $N=v(1+\frac{v-1}{k-1})$ &$=\sqrt{\frac{N-M}{M(N-1)}}$ &$\leq\frac{\mu_F}{\sqrt{M}}$ \\

Code-based &$\mathbb{R}$ &$2^m\times 2^{(t+1)m}$ &$\leq\frac{1}{\sqrt{2^{m-2t-1}}}$ &$\leq\frac{\mu_F}{\sqrt{2^m}}$ \\

\hline
\end{tabular}
\end{scriptsize}
\caption{\label{table.constructions}
Eight constructions detailed in this chapter.
The bounds given for the normalized Gaussian, random harmonic and Steinhaus Gabor frames are satisfied with high probability.
All of the frames above are unit norm tight frames except for the normalized Gaussian frame, which has squared spectral norm $\|\Phi\|_2^2\leq(\!\sqrt{M}+\!\sqrt{N}+\!\sqrt{2\log{N}})^2/(M-\!\sqrt{8M\log{N}})$ in the same probability event.
}
\end{center}
\end{table}

\begin{thm}[Geometry of code-based frames]
\label{thm.code-based coherence}
The $2^m\times2^{(t+1)m}$ frame defined by \eqref{eq.matrix defn} is unit norm and tight, i.e., $\|\Phi\|_2^2=2^{tm}$, with worst-case coherence $\mu\leq \frac{1}{\sqrt{2^{m-2t-1}}}$ and average coherence $\nu\leq\frac{\mu}{\sqrt{2^{m}}}$.
\end{thm}

\begin{proof}
For the tightness claim, we use the linearity of the trace map to write the inner product of rows $x$ and $y$:
\begin{align*} 
&\sum_{\alpha\in\mathbb{F}_{2^m}^{t+1}}\!\!\frac{1}{\sqrt{2^{m}}}(-1)^{\mathrm{Tr}\big[\alpha_0x+\sum_{i=1}^t\alpha_ix^{2^i+1}\big]}\frac{1}{\sqrt{2^{m}}}(-1)^{\mathrm{Tr}\big[\alpha_0y+\sum_{i=1}^t\alpha_iy^{2^i+1}\big]}\\
&\qquad=\frac{1}{2^m}\bigg(\!\sum_{\alpha_0\in\mathbb{F}_{2^m}}(-1)^{\mathrm{Tr}[\alpha_0(x+y)]}\bigg)\!\!\sum_{\alpha_1\in\mathbb{F}_{2^m}}\!\!\cdots\!\!\sum_{\alpha_t\in\mathbb{F}_{2^m}}\!\!(-1)^{\mathrm{Tr}\big[\sum_{i=1}^t\alpha_i(x^{2^i+1}+y^{2^i+1})\big]}.
\end{align*}
This expression is $2^{tm}$ when $x=y$.  
Otherwise, note that $\alpha_0\mapsto(-1)^{\mathrm{Tr}[\alpha_0(x+y)]}\in\{\pm1\}$ defines a homomorphism on $\mathbb{F}_{2^m}$.
Since $(x+y)^{-1}\mapsto-1$, the inverse images of $\pm1$ under this homomorphism must form two cosets of equal size, and so $\sum_{\alpha_0\in\mathbb{F}_{2^m}}(-1)^{\mathrm{Tr}[\alpha_0(x+y)]}=0$, meaning distinct rows in $\Phi$ are orthogonal.  
Thus, $\Phi$ is a unit norm tight frame.

For the worst-case coherence claim, we first note that the linearity of the trace map gives
\begin{equation*} (-1)^{\mathrm{Tr}\big[\alpha_0x+\sum_{i=1}^t\alpha_ix^{2^i+1}\big]}(-1)^{\mathrm{Tr}\big[\alpha'_0x+\sum_{i=1}^t\alpha'_ix^{2^i+1}\big]}=(-1)^{\mathrm{Tr}\big[(\alpha_0+\alpha'_0)x+\sum_{i=1}^t(\alpha_i+\alpha'_i)x^{2^i+1}\big]},
\end{equation*}
i.e., every inner product between columns of $\Phi$ is a sum over another column.
Thus, there exists $\alpha\in\mathbb{F}_{2^m}^{t+1}$ such that
\begin{align*}
2^{2m}\mu^2
&=\bigg(\sum_{x\in\mathbb{F}_{2^m}}(-1)^{\mathrm{Tr}\big[\alpha_0x+\sum_{i=1}^t\alpha_ix^{2^i+1}\big]}\bigg)^2\\
&=2^m+\sum_{x\in\mathbb{F}_{2^m}}\sum_{\substack{y\in\mathbb{F}_{2^m}\\y\neq x}}(-1)^{\mathrm{Tr}\big[\alpha_0(x+y)+\sum_{i=1}^t\alpha_i\big((x+y)^{2^i+1}+\sum_{j=0}^{i-1}(xy)^{2^j}(x+y)^{2^i-2^{j+1}+1}\big)\big]},
\end{align*}
where the last equality is by the identity
$(x+y)^{2^i+1}=x^{2^i+1}+y^{2^i+1}+\sum_{j=0}^{i-1}(xy)^{2^j}(x+y)^{2^i-2^{j+1}+1}$, whose proof is a simple exercise of induction.
From here, we perform a change of variables: $u:= x+y$ and $v:= xy$.
Notice that $(u,v)$ corresponds to $(x,y)$ for some $x\neq y$ whenever $(z+x)(z+y)=z^2+uz+v$ has two solutions, that is, whenever $\smash{\mathrm{Tr}(\frac{v}{u^2})=0}$.
Since $(u,v)$ corresponds to both $(x,y)$ and $(y,x)$, we must correct for under-counting:
\begin{align}
2^{2m}\mu^2
\nonumber
&=2^m+2\sum_{\substack{u\in\mathbb{F}_{2^m}\\u\neq0}} \sum_{\substack{v\in\mathbb{F}_{2^m}\\\mathrm{Tr}(v/u^2)=0}}(-1)^{\mathrm{Tr}\big[\alpha_0u+\sum_{i=1}^t\alpha_i\big(u^{2^i+1}+\sum_{j=0}^{i-1}v^{2^j}u^{2^i-2^{j+1}+1}\big)\big]}\\
\nonumber
&=2^m+2\sum_{\substack{u\in\mathbb{F}_{2^m}\\u\neq0}}(-1)^{\mathrm{Tr}\big[\alpha_0u+\sum_{i=1}^t\alpha_iu^{2^i+1}\big]}\sum_{\substack{v\in\mathbb{F}_{2^m}\\\mathrm{Tr}(v/u^2)=0}
}(-1)^{\mathrm{Tr}\big[\big(\sum_{i=1}^t\sum_{j=0}^{i-1}\alpha_i^{2^{-j}}u^{2^{i-j}-2+2^{-j}}\big)v\big]}\\
\label{eq.big}
&\leq2^m+2\sum_{\substack{u\in\mathbb{F}_{2^m}\\u\neq0}}~\bigg|\!\!\!\sum_{\substack{v\in\mathbb{F}_{2^m}
\\\mathrm{Tr}(v/u^2)=0}}\!\!\!(-1)^{\mathrm{Tr}[p(u)v]}~\bigg|,
\end{align}
where the second equality is by repeated application of $\mathrm{Tr}(z)=\mathrm{Tr}(z^2)$, and 
\begin{equation*}
p(u):=\sum_{i=1}^t\sum_{j=0}^{i-1}\alpha_i^{2^{-j}}u^{2^{i-j}-2+2^{-j}}.
\end{equation*}
To bound $\mu$, we will count the $u$'s that produce nonzero summands in \eqref{eq.big}.

For each $u\neq0,$ we have a homomorphism $\chi_u\colon\{v\in\mathbb{F}_{2^m}:\mathrm{Tr}(\frac{v}{u^2})=0\}\rightarrow\{\pm1\}$ defined by
$\chi_u(v):=(-1)^{\mathrm{Tr}[p(u)v]}$.
Pick $u\neq0$ for which there exists a $v$ such that both $\smash{\mathrm{Tr}(\frac{v}{u^2})=0}$ and $\mathrm{Tr}[p(u)v]=1$.
Then $\chi_u(v)=-1$, and so the kernel of $\chi_u$ is the same size as the coset $\smash{\{v\in\mathbb{F}_{2^m}:\mathrm{Tr}(\frac{v}{u^2})=0,\chi_u(v)=-1\}}$, meaning the summand associated with $u$ in \eqref{eq.big} is zero.
Hence, the nonzero summands in \eqref{eq.big} require $\smash{\mathrm{Tr}(\frac{v}{u^2})=0}$ and $\mathrm{Tr}[p(u)v]=0$.
This is certainly possible whenever $p(u)=0$.
Exponentiation gives 
\begin{equation*}
p(u)^{2^{t-1}}=\sum_{i=1}^t\sum_{j=0}^{i-1}\alpha_i^{2^{t-j-1}}u^{2^{t+i-j-1}-2^t+2^{t-j-1}}, 
\end{equation*}
which has degree $2^{2t-1}-2^{t-1}$.
Thus, $p(u)=0$ has at most $2^{2t-1}-2^{t-1}$ solutions, and each such $u$ produces a summand in \eqref{eq.big} of size $2^{m-1}$.
Next, we consider the $u$'s for which $\smash{\mathrm{Tr}(\frac{v}{u^2})=0}$, $\mathrm{Tr}[p(u)v]=0$, and $p(u)\neq0$.
In this case, the hyperplanes defined by $\smash{\mathrm{Tr}(\frac{v}{u^2})=0}$ and $\mathrm{Tr}[p(u)v]=0$ are parallel, and so $\smash{p(u)=\frac{1}{u^2}}$.
Here, 
\begin{equation*}
1=(u^2p(u))^{2^{t-1}}=\sum_{i=1}^t\sum_{j=0}^{i-1}\alpha_i^{2^{t-j-1}}u^{2^{t+i-j-1}+2^{t-j-1}}, 
\end{equation*}
which has degree $2^{2t-1}+2^{t-1}$.
Thus, $\smash{p(u)=\frac{1}{u^2}}$ has at most $2^{2t-1}+2^{t-1}$ solutions, and each such $u$ produces a summand in \eqref{eq.big} of size $2^{m-1}$.
We can now continue the bound from \eqref{eq.big}: $2^{2m}\mu^2\leq 2^m+2(2^{2t-1}-2^{t-1}+2^{2t-1}+2^{t-1})2^{m-1}\leq 2^{m+2t+1}$. 
From here, isolating $\mu$ gives the claim.

Lastly, for average coherence, pick some $x\in\mathbb{F}_{2^m}$.
Then summing the entries in the $x$th row gives
\begin{align*}
&\sum_{\alpha\in\mathbb{F}_{2^m}^{t+1}}\frac{1}{\sqrt{2^{m}}}(-1)^{\mathrm{Tr}\big[\alpha_0x+\sum_{i=1}^t\alpha_ix^{2^i+1}\big]}\\
&\qquad=\frac{1}{\sqrt{2^{m}}}\bigg(\sum_{\alpha_0\in\mathbb{F}_{2^m}}(-1)^{\mathrm{Tr}(\alpha_0x)}\bigg)\sum_{\alpha_1\in\mathbb{F}_{2^m}}\cdots\sum_{\alpha_t\in\mathbb{F}_{2^m}}(-1)^{\mathrm{Tr}\big[\sum_{i=1}^t\alpha_ix^{2^i+1}\big]}\\
&\qquad=\left\{\begin{array}{lc}2^{(t+1/2)m},&x=0\\0,&x\neq0\end{array}\right. .
\end{align*}
That is, the frame elements sum to a multiple of an identity basis element: $\sum_{\alpha\in\mathbb{F}_{2^m}^{t+1}}\varphi_\alpha=2^{(t+1/2)m}\delta_0$.
Since every entry in row $x=0$ is $\smash{\frac{1}{\sqrt{2^{m}}}}$, we have  
$\langle \varphi_{\alpha'},\sum_{\alpha\in\mathbb{F}_{2^m}^{t+1}}\varphi_\alpha \rangle=\frac{2^{(t+1)m}}{2^m}$
for every $\alpha'\in\mathbb{F}_{2^m}^{t+1}$, and so by Lemma~\ref{lem.sufficient conditions}(i), we are done.
\end{proof}

\begin{exmp}
To illustrate the bounds in Theorem~\ref{thm.code-based coherence}, we consider the example where $m=4$ and $t=1$.
This is a $16\times 256$ code-based frame $\Phi$ with $\smash{\mu=\frac{1}{2}\leq\frac{1}{\sqrt{2}}=\frac{1}{\sqrt{2^{m-2t-1}}}}$ and $\smash{\nu=\frac{1}{17}\leq\frac{1}{8}=\frac{\mu}{\sqrt{2^m}}}$.
\end{exmp}

\section{Fundamental limits on worst-case coherence}

In many applications of frames, performance is dictated by worst-case coherence \cite{bajwa:jcn10,candes:annstat09,DonohoET:06,HolmesP:04,MixonQKF:11,StrohmerH:03,Tropp:04,Tropp:acha08,zahedi:acc10}. 
It is therefore particularly important to understand which worst-case coherence values are achievable.
To this end, the Welch bound is commonly used in the literature.
When worst-case coherence achieves the Welch bound, the frame is equiangular and tight \cite{StrohmerH:03}.
However, equiangular tight frames cannot have more vectors than the square of the spatial dimension \cite{StrohmerH:03}, meaning the Welch bound is not tight whenever $N>M^2$.
When the number of vectors $N$ is exceedingly large, the following theorem gives a better bound:

\begin{thm}[\cite{alon:dm03,nelson:jat11}]
\label{thm.asymptotic bound}
Every sufficiently large $M\times N$ unit norm frame with $N\geq2M$ and worst-case coherence $\mu<\frac{1}{2}$ satisfies
\begin{equation}
\mu^2\log\frac{1}{\mu}\geq\frac{C\log N}{M}
\end{equation}
for some constant $C>0$.
\end{thm}

For a fixed worst-case coherence $\mu<\frac{1}{2}$, this bound indicates that the number of vectors $N$ cannot exceed some exponential in the spatial dimension $M$, that is, $N\leq a^M$ for some $a>0$.
However, since the constant $C$ is not established in this theorem, it is unclear which base $a$ is appropriate for each $\mu$.
The following theorem is a little more explicit in this regard:

\begin{thm}[\cite{MSEA03,XiaZG:05}]
\label{thm.complex bound}
Every $M\times N$ unit norm frame has worst-case coherence $\mu\geq1-2N^{-1/(M-1)}$.
Furthermore, taking $N=\Theta(a^M)$, this lower bound goes to $1-\frac{2}{a}$ as $M\rightarrow\infty$.
\end{thm}

For many applications, it does not make sense to use a complex frame, but the bound in Theorem~\ref{thm.complex bound} is known to be loose for real frames \cite{CHS96}.
We therefore improve Theorems~\ref{thm.asymptotic bound} and~\ref{thm.complex bound} for the case of real unit norm frames:

\begin{thm}
\label{thm.bound}
Every real $M\times N$ unit norm frame has worst-case coherence
\begin{equation}
\label{eq.my bound}
\mu\geq\cos\bigg[\pi\bigg(\frac{M-1}{N\pi^{1/2}}\cdot\frac{\Gamma(\frac{M-1}{2})}{\Gamma(\frac{M}{2})}\bigg)^{\frac{1}{M-1}}\bigg].
\end{equation}
Furthermore, taking $N=\Theta(a^M)$, this lower bound goes to $\cos(\frac{\pi}{a})$ as $M\rightarrow\infty$.
\end{thm}

Before proving this theorem, we first consider the special case where the dimension is $M=3$:

\begin{lem}
\label{lem.3d points}
Given $N$ points on the unit sphere $\mathbb{S}^{2}\subseteq\mathbb{R}^3$, the smallest angle between points is $\leq2\cos^{-1}\big(1-\frac{2}{N}\big)$.
\end{lem}

\begin{proof}
We first claim there exists a closed spherical cap in $\mathbb{S}^{2}$ with area $\smash{\frac{4\pi}{N}}$ that contains two of the $N$ points.  
Suppose otherwise, and take $\gamma$ to be the angular radius of a spherical cap with area $\smash{\frac{4\pi}{N}}$.  
That is, $\gamma$ is the angle between the center of the cap and every point on the boundary.  
Since the cap is closed, we must have that the smallest angle $\alpha$ between any two of our $N$ points satisfies $\alpha>2\gamma$.  
Let $C(p,\theta)$ denote the closed spherical cap centered at $p\in \mathbb{S}^2$ of angular radius $\theta$, and let $P$ denote our set of $N$ points.  
Then we know for $p\in P$, the $C(p,\gamma)$'s are disjoint, $\frac{\alpha}{2}>\gamma$, and $\bigcup_{p\in P}C(p,\tfrac{\alpha}{2})\subseteq \mathbb{S}^2$, and so taking 2-dimensional Hausdorff measures on the sphere gives
\begin{equation*} 
\mathrm{H}^2(\mathbb{S}^2)=4\pi=\mathrm{H}^2\bigg(\bigcup_{p\in P}C(p,\gamma)\bigg)<\mathrm{H}^2\bigg(\bigcup_{p\in P}C(p,\tfrac{\alpha}{2})\bigg)\leq\mathrm{H}^2(\mathbb{S}^2),
\end{equation*} 
a contradiction.  

Since two of the points reside in a spherical cap of area $\smash{\frac{4\pi}{N}}$, we know $\alpha$ is no more than twice the radius of this cap.  
We use spherical coordinates to relate the cap's area to the radius:
$\smash{\mathrm{H}^2(C(\cdot,\gamma))=2\pi\int_0^\gamma\sin\phi~\mathrm{d}\phi=2\pi(1-\cos\gamma)}$. 
Therefore, when $\smash{\mathrm{H}^2(C(\cdot,\gamma))=\frac{4\pi}{N}}$, we have $\gamma=\cos^{-1}(1-\frac{2}{N})$, and so $\alpha\leq2\gamma$ gives
the result.
\end{proof}

\begin{thm}
\label{thm.3d points}
Every real $3\times N$ unit norm frame has worst-case coherence $\mu\geq1-\frac{4}{N}+\frac{2}{N^2}$.
\end{thm}

\begin{proof}
Packing $N$ unit vectors in $\mathbb{R}^3$ corresponds to packing $2N$ antipodal points in $\mathbb{S}^2$, 
and so Lemma~\ref{lem.3d points} gives $\alpha\leq2\cos^{-1}(1-\frac{1}{N})$.
Applying the double angle formula to 
\begin{equation*}
\mu=\cos\alpha\geq\cos[2\cos^{-1}(1-\tfrac{1}{N})]
\end{equation*}
gives the result.
\end{proof}

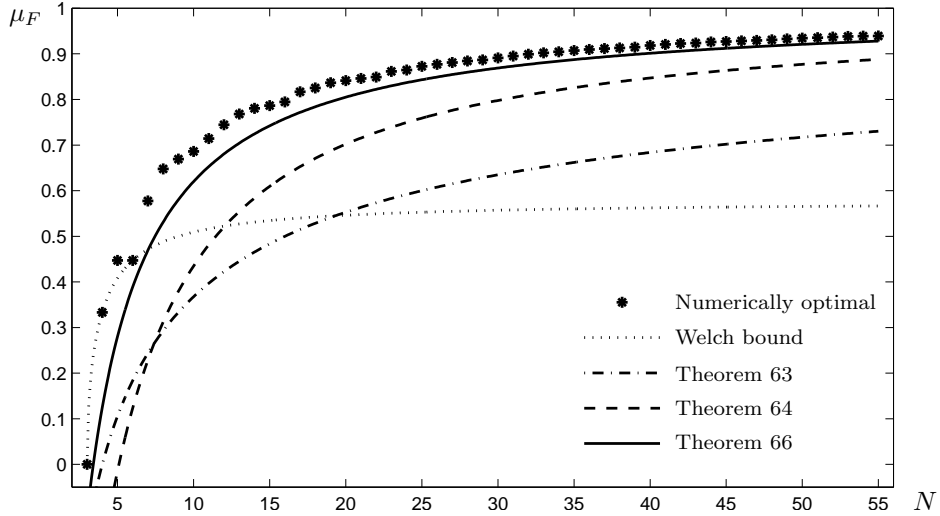
\begin{figure}[t]
\centering
\begin{picture}(320,193)(0,0)
\put(0,0){\includegraphics[width=4.5in]{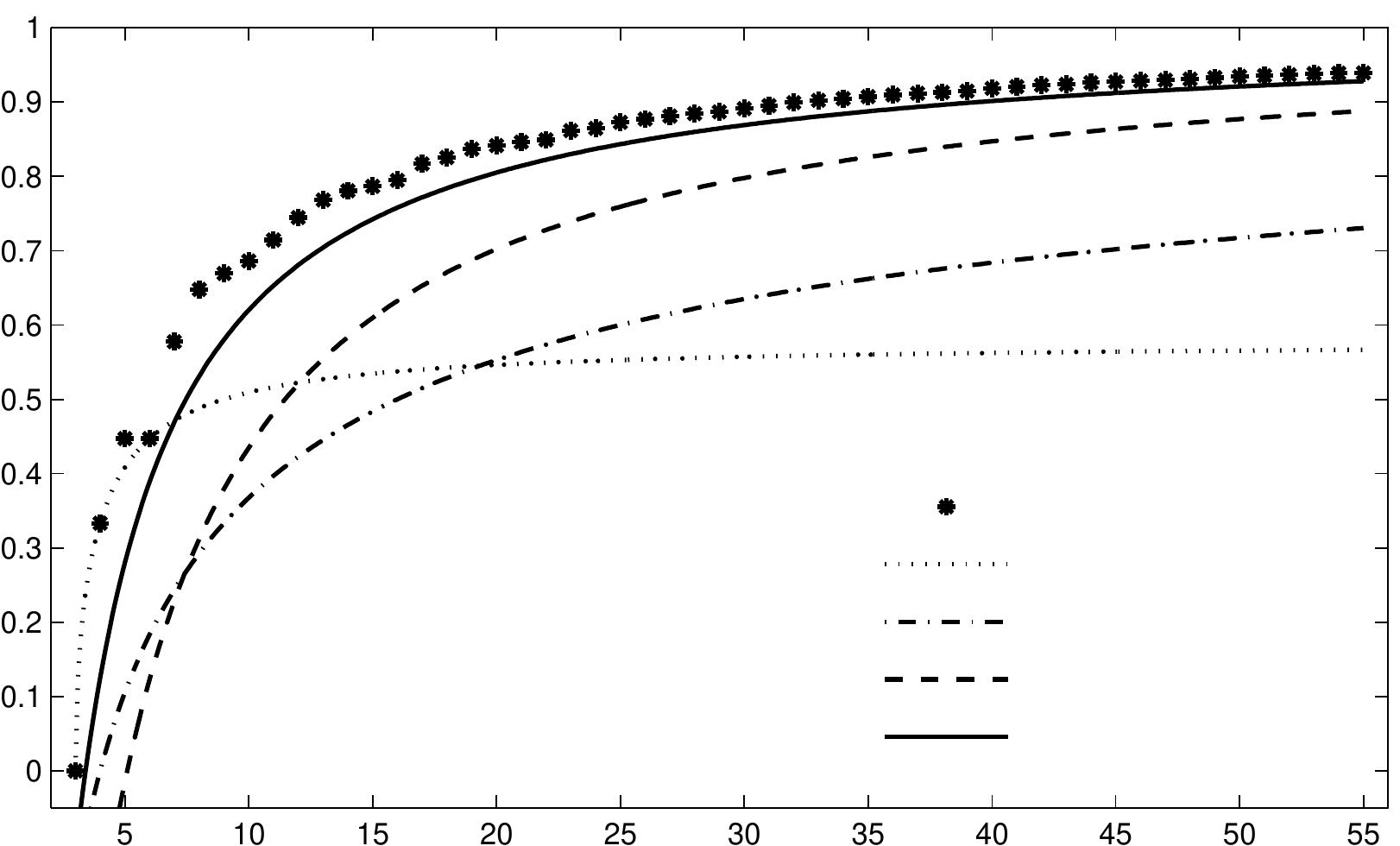}}

\put(330,0){$N$}
\put(-12,185){$\mu_F$}

\put(240,76){\footnotesize{Numerically optimal}}
\put(240,63){\footnotesize{Welch bound}}
\put(240,49){\footnotesize{Theorem~\ref{thm.complex bound}}}
\put(240,36){\footnotesize{Theorem~\ref{thm.bound}}}
\put(240,23){\footnotesize{Theorem~\ref{thm.3d points}}}

\end{picture}

\caption{Different bounds on worst-case coherence for $M=3$, $N=3,\ldots,55$.
Stars give numerically determined optimal worst-case coherence of $N$ real unit vectors, found in \cite{CHS96}.
Dotted curve gives Welch bound, dash-dotted curve gives bound from Theorem~\ref{thm.complex bound}, dashed curve gives bound from Theorem~\ref{thm.bound}, and solid curve gives bound from Theorem~\ref{thm.3d points}.
 \label{4 figure}}
\end{figure}

Now that we understand the special case where $M=3$, we tackle the general case:

\begin{proof}[Proof of Theorem~\ref{thm.bound}]
As in the proof of Theorem~\ref{thm.3d points}, we relate packing $N$ unit vectors to packing $2N$ points in the hypersphere $\mathbb{S}^{M-1}\subseteq\mathbb{R}^M$.
The argument in the proof of Lemma~\ref{lem.3d points} generalizes so that two of the $2N$ points must reside in some closed hyperspherical cap of hypersurface area $\frac{1}{2N}\mathrm{H}^{M-1}(\mathbb{S}^{M-1})$.
Therefore, the smallest angle $\alpha$ between these points is no more than twice the radius of this cap.
Let $C(\gamma)$ denote a hyperspherical cap of angular radius $\gamma$.
Then we use hyperspherical coordinates to get
\begin{align}
\nonumber
\mathrm{H}^{M-1}(C(\gamma))
&=\int_{\phi_1=0}^\gamma\int_{\phi_2=0}^\pi\cdots\int_{\phi_{M-2}=0}^\pi\int_{\phi_{M-1}=0}^{2\pi}\sin^{M-2}(\phi_1)\cdots\sin^1(\phi_{M-2})~\mathrm{d}\phi_{M-1}\cdots\mathrm{d}\phi_1\\
\nonumber
&=2\pi\bigg(\prod_{j=1}^{M-3}\pi^{1/2}\frac{\Gamma(\frac{j+1}{2})}{\Gamma(\frac{j}{2}+1)}\bigg)\int_0^\gamma\sin^{M-2}\phi ~\mathrm{d}\phi\\
\label{eq.gamma}
&=\frac{2\pi^{(M-1)/2}}{\Gamma(\frac{M-1}{2})}\int_0^\gamma\sin^{M-2}\phi ~\mathrm{d}\phi.
\end{align}
We wish to solve for $\gamma$, but analytically inverting $\int_0^\gamma\sin^{M-2}\phi ~\mathrm{d}\phi$ is difficult.
Instead, we use $\sin\phi\geq\frac{2\phi}{\pi}$ for $\phi\in[0,\frac{\pi}{2}]$.
Note that we do not lose generality by forcing $\gamma\leq\frac{\pi}{2}$, since this is guaranteed with $N\geq2$.
Continuing \eqref{eq.gamma} gives
\begin{equation}
\label{eq.cap hypersurface area} 
\mathrm{H}^{M-1}(C(\gamma))
\geq\frac{2\pi^{(M-1)/2}}{\Gamma(\frac{M-1}{2})}\int_0^\gamma\Big(\frac{2\phi}{\pi}\Big)^{M-2}\mathrm{d}\phi
=\frac{(2\gamma)^{M-1}}{(M-1)\pi^{(M-3)/2}\Gamma(\frac{M-1}{2})}.
\end{equation}
Using the formula for a hypersphere's hypersurface area, we can express the left-hand side of \eqref{eq.cap hypersurface area}:
\begin{equation*}
\frac{(2\gamma)^{M-1}}{(M-1)\pi^{(M-3)/2}\Gamma(\frac{M-1}{2})}
\leq \mathrm{H}^{M-1}(C(\gamma))
=\frac{1}{2N}\mathrm{H}^{M-1}(\mathbb{S}^{M-1})
=\frac{\pi^{M/2}}{N\Gamma(\frac{d}{2})}.
\end{equation*}
Isolating $2\gamma$ above and using $\alpha\leq2\gamma$ and $\mu=\cos\alpha$ gives \eqref{eq.my bound}.
The second part of the result comes from a simple application of Stirling's approximation.
\end{proof}

In \cite{CHS96}, numerical results are given for $M=3$, and we compare these results to Theorems~\ref{thm.complex bound} and~\ref{thm.bound} in Figure~\ref{4 figure}.
Considering this figure, we note that the bound in Theorem~\ref{thm.complex bound} is inferior to the maximum of the Welch bound and the bound in Theorem~\ref{thm.bound}, at least when $M=3$.
This illustrates the degree to which Theorem~\ref{thm.bound} improves the bound in Theorem~\ref{thm.complex bound} for real frames.
In fact, since $\cos(\frac{\pi}{a})\geq 1-\frac{2}{a}$ for all $a\geq2$, the bound for real frames in Theorem~\ref{thm.bound} is asymptotically better than the bound for complex frames in Theorem~\ref{thm.complex bound}.
Moreover, for $M=2$, Theorem~\ref{thm.bound} says $\mu\geq\cos(\frac{\pi}{N})$, and~\cite{BenedettoK:06} proved this bound to be tight for every $N\geq2$.
Lastly, Figure~\ref{4 figure} illustrates that Theorem~\ref{thm.3d points} improves the bound in Theorem~\ref{thm.bound} for the case $M=3$.

In many applications, large dictionaries are built to obtain sparse reconstruction, but the known guarantees on sparse reconstruction place certain requirements on worst-case coherence.
Asymptotically, the bounds in Theorems~\ref{thm.complex bound} and \ref{thm.bound} indicate that certain exponentially large dictionaries will not satisfy these requirements.
For example, if $N=\Theta(3^M)$, then $\mu_F=\Omega(\frac{1}{3})$ by Theorem~\ref{thm.complex bound}, and if the frame is real, we have $\mu=\Omega(\frac{1}{2})$ by Theorem~\ref{thm.bound}.
Such a dictionary will only work for sparse reconstruction if the sparsity level $K$ is sufficiently small; deterministic guarantees require $K<\mu^{-1}$ \cite{DonohoET:06,Tropp:04}, while probabilistic guarantees require $K<\mu^{-2}$ \cite{bajwa:jcn10,tropp:cras08}, and so in this example, the dictionary can, at best, only accommodate sparsity levels that are smaller than 10.
Unfortunately, in real-world applications, we can expect the sparsity level to scale with the signal dimension.
This in mind, Theorems~\ref{thm.complex bound} and \ref{thm.bound} tell us that dictionaries can only be used for sparse reconstruction if $N=O((2+\varepsilon)^M)$ for some sufficiently small $\varepsilon>0$.  
To summarize, the Welch bound is known to be tight only if $N\leq M^2$, and Theorems~\ref{thm.complex bound} and \ref{thm.bound} give bounds which are asympotically better than the Welch bound whenever $N=\Omega(2^M)$.
When $N$ is between $M^2$ and $2^M$, the best bound to date is the (loose) Welch bound, and so more work needs to be done to bound worst-case coherence in this parameter region.

\section{Reducing average coherence}

In \cite{bajwa:jcn10}, average coherence is used to derive a number of guarantees on sparse signal processing.
Since average coherence is so new to the frame theory literature, this section will investigate how average coherence relates to worst-case coherence and the spectral norm.
We start with a definition:

\begin{defn}[Wiggling and flipping equivalent frames]
\label{def.flipping and wiggling}
We say the frames $\Phi$ and $\Psi$ are \emph{wiggling equivalent} if there exists a diagonal matrix $D$ of unimodular entries such that $\Psi=\Phi D$.
Furthermore, they are \emph{flipping equivalent} if $D$ is real, having only $\pm1$'s on the diagonal.
\end{defn}

The terms ``wiggling'' and ``flipping'' are inspired by the fact that individual frame elements of such equivalent frames are related by simple unitary operations.
Note that every frame with $N$ nonzero frame elements belongs to a flipping equivalence class of size $2^N$, while being wiggling equivalent to uncountably many frames.
The importance of this type of frame equivalence is, in part, due to the following lemma, which characterizes the shared geometry of wiggling equivalent frames:

\begin{lem}[Geometry of wiggling equivalent frames]\label{lem:geom_eqframes}
Wiggling equivalence preserves the norms of frame elements, the worst-case coherence, and the spectral norm.
\end{lem}

\begin{proof}
Take two frames $\Phi$ and $\Psi$ such that $\Psi=\Phi D$.
The first claim is immediate.
Next, the Gram matrices are related by $\Psi^*\Psi=D^*\Phi^*\Phi D$.
Since corresponding off-diagonal entries are equal in modulus, we know the worst-case coherences are equal.
Finally, $\|\Psi\|_2^2=\|\Psi\Psi^*\|_2^2=\|\Phi DD^*\Phi^*\|_2^{}=\|\Phi\Phi^*\|_2^{}=\|\Phi\|_2^2$, and so we are done.
\end{proof}

Wiggling and flipping equivalence are not entirely new to frame theory.
For a real equiangular tight frame $\Phi$, the Gram matrix $\Phi^*\Phi$ is completely determined by the sign pattern of the off-diagonal entries, which can in turn be interpreted as the Seidel adjacency matrix of a graph $G_\Phi$.
As such, flipping a frame element $\varphi\in \Phi$ has the effect of negating the corresponding row and column in the Gram matrix, which further corresponds to \emph{switching} the adjacency rule for that vertex $v_\varphi\in V(G_\Phi)$ in the graph---vertices are adjacent to $v_\varphi$ after switching precisely when they were not adjacent before switching. 
Graphs are called \emph{switching equivalent} if there is a sequence of switching operations that produces one graph from the other; this equivalence was introduced in~\cite{vanLintS:66} and was later extensively studied by Seidel in \cite{Seidel:73,seidel:laa68}.
Since flipping equivalent real equiangular tight frames correspond to switching equivalent graphs, the terms have become interchangeable.
For example,~\cite{bodmann:jfa10} uses switching (i.e., wiggling and flipping) equivalence to make progress on an important problem in frame theory called the \emph{Paulsen problem}, which asks how close a nearly unit norm, nearly tight frame must be to a unit norm tight frame.

Now that we understand wiggling and flipping equivalence, we are ready for the main idea behind this section.
Suppose we are given a unit norm frame with acceptable spectral norm and worst-case coherence, but we also want the average coherence to satisfy (SCP-2).
Then by Lemma~\ref{lem:geom_eqframes}, all of the wiggling equivalent frames will also have acceptable spectral norm and worst-case coherence, and so it is reasonable to check these frames for good average coherence.
In fact, the following theorem guarantees that at least one of the flipping equivalent frames will have good average coherence, with only modest requirements on the original frame's redundancy.

\begin{thm}[Constructing frames with low average coherence]\label{thm:avc_rand}
Let $\Phi$ be an $M\times N$ unit norm frame with $\smash{M < \frac{N-1}{4\log 4N}}$.
Then there exists a frame $\Psi$ that is flipping equivalent to $\Phi$ and satisfies $\smash{\nu\leq\frac{\mu}{\sqrt{M}}}$.
\end{thm}

\begin{proof}
Take $\{R_n\}_{n=1}^N$ to be a Rademacher sequence that independently takes values $\pm1$, each with probability $\frac{1}{2}$.
We use this sequence to randomly flip $\Phi$; define $Z:=\Phi~\mathrm{diag}\{R_n\}_{n=1}^N$.
Note that if $\smash{\Pr(\nu_Z\leq\frac{\mu_\Phi}{\sqrt{M}})>0}$, we are done.
Fix some $i\in\{1,\ldots,N\}$.
Then
\begin{equation}
\label{pfeqn:avc_rand_tail1}
\Pr\Bigg(\frac{1}{N-1} \bigg|\sum_{\substack{j=1\\ j\neq i}}^N \langle z_i,z_j\rangle\bigg| > \frac{\mu_\Phi}{\sqrt{M}}\Bigg) 
=\Pr\Bigg(\bigg|\sum_{\substack{j=1\\ j\neq i}}^N R_j\langle \varphi_i,\varphi_j\rangle\bigg| > \frac{(N-1)\mu_\Phi}{\sqrt{M}}\Bigg). 
\end{equation}
We can view $\sum_{j\neq i} R_j\langle \varphi_i,\varphi_j\rangle$ as a sum of $N-1$ independent zero-mean complex random variables that are bounded by $\mu_\Phi$.
We can therefore use a complex version of Hoeffding's inequality \cite{hoeffding:jasa63} (see, e.g., Lemma~3.8 of~\cite{bajwa:thesis}) to bound the probability expression in \eqref{pfeqn:avc_rand_tail1} as $\leq4e^{-(N-1)/4M}$.
From here, a union bound over all $N$ choices for $i$ gives $\Pr(\nu_Z\leq\frac{\mu_\Phi}{\sqrt{M}})\geq 1-4Ne^{-(N-1)/4M}$,
and so $M < \frac{N-1}{4\log 4N}$ implies $\Pr(\nu_Z\leq\frac{\mu_\Phi}{\sqrt{M}})>0$, as desired.
\end{proof}

While Theorem~\ref{thm:avc_rand} guarantees the existence of a flipping equivalent frame with good average coherence, the result does not describe how to find it.
Certainly, one could check all $2^N$ frames in the flipping equivalence class, but such a procedure is computationally slow.
As an alternative, we propose a linear-time flipping algorithm (Algorithm~\ref{alg:flipping}).
The following theorem guarantees that linear-time flipping will produce a frame with good average coherence, but it requires the original frame's redundancy to be higher than what suffices in Theorem~\ref{thm:avc_rand}.

\begin{algorithm*}[t]
\caption{Linear-time flipping}
\label{alg:flipping}
\textbf{Input:} An $M\times N$ unit norm frame $\Phi$\\
\textbf{Output:} An $M\times N$ unit norm frame $\Psi$ that is flipping equivalent to $\Phi$
\begin{algorithmic}
\STATE $\psi_1\leftarrow \varphi_1$ \hfill \COMMENT{Keep first frame element}
\FOR{$n=2$ to $N$}
\IF{$\|\sum_{i=1}^{n-1}\psi_i+\varphi_n\|\leq\|\sum_{i=1}^{n-1}\psi_i-\varphi_n\|$} 
\STATE $\psi_n\leftarrow \varphi_n$ \hfill \COMMENT{Keep frame element to make sum length shorter}
\ELSE
\STATE $\psi_n\leftarrow -\varphi_n$ \hfill \COMMENT{Flip frame element to make sum length shorter}
\ENDIF
\ENDFOR
\end{algorithmic}
\end{algorithm*}

\begin{thm}
\label{thm.alg}
Suppose $N\geq M^2+3M+3$.
Then Algorithm~\ref{alg:flipping} outputs an $M\times N$ frame $\Psi$ that is flipping equivalent to $\Phi$ and satisfies $\nu\leq\frac{\mu}{\sqrt{M}}$.
\end{thm}

\begin{proof}
Considering Lemma~\ref{lem.sufficient conditions}(iii), it suffices to have $\|\sum_{n=1}^N \psi_n\|^2\leq N$.
We will use induction to show $\|\sum_{n=1}^k \psi_n\|^2\leq k$ for $k=1,\ldots,N$.
Clearly, $\|\sum_{n=1}^1 \psi_n\|^2=\|\varphi_n\|^2=1\leq1$.
Now assume $\|\sum_{n=1}^k \psi_n\|^2\leq k$.
Then by our choice for $\psi_{k+1}$ in Algorithm~\ref{alg:flipping}, we know that
$\|\sum_{n=1}^k\psi_n+\psi_{k+1}\|^2\leq\|\sum_{n=1}^k\psi_n-\psi_{k+1}\|^2$.
Expanding both sides of this inequality gives
\begin{equation*}
\bigg\|\sum_{n=1}^k\psi_n\bigg\|^2+2\mathrm{Re}\bigg\langle\sum_{n=1}^k\psi_n,\psi_{k+1}\bigg\rangle+\|\psi_{k+1}\|^2
\leq\bigg\|\sum_{n=1}^k\psi_n\bigg\|^2-2\mathrm{Re}\bigg\langle\sum_{n=1}^k\psi_n,\psi_{k+1}\bigg\rangle+\|\psi_{k+1}\|^2,
\end{equation*}
and so $\mathrm{Re}\langle\sum_{n=1}^k\psi_n,\psi_{k+1}\rangle\leq0$.
Therefore,
\begin{equation*} 
\bigg\|\sum_{n=1}^{k+1}\psi_n\bigg\|^2
=\bigg\|\sum_{n=1}^k\psi_n\bigg\|^2+2\mathrm{Re}\bigg\langle\sum_{n=1}^k\psi_n,\psi_{k+1}\bigg\rangle+\|\psi_{k+1}\|^2
\leq\bigg\|\sum_{n=1}^k\psi_n\bigg\|^2+\|\psi_{k+1}\|^2
\leq k+1,
\end{equation*}
where the last inequality uses the inductive hypothesis.
\end{proof}

\begin{exmp}
Apply linear-time flipping to reduce average coherence in the following matrix:
\begin{equation*}
\Phi:= \frac{1}{\sqrt{5}}\left[ \begin{array}{cccccccccc} +&+&+&+&-&+&+&+&+&-\\+&-&+&+&+&-&-&-&+&-\\+&+&+&+&+&+&+&+&-&+\\-&-&-&+&-&+&+&-&-&-\\-&+&+&-&-&+&-&-&-&- \end{array} \right].
\end{equation*}
Here, $\smash{\nu_\Phi\approx0.3778>0.2683\approx\frac{\mu_\Phi}{\sqrt{M}}}$, and linear-time flipping produces the flipping pattern $D:=\mathrm{diag}(+-+--++-++)$.
Then $\Phi D$ has average coherence $\smash{\nu_{\Phi D}\approx0.1556<\frac{\mu_{\Phi}}{\sqrt{M}}=\frac{\mu_{\Phi D}}{\sqrt{M}}}$.
This illustrates that the condition $N\geq M^2+3M+3$ in Theorem~\ref{thm.alg} is sufficient but not necessary.
\end{exmp}

% \section*{Acknowledgments}
% The authors thank the anonymous referees for their helpful suggestions, Matthew Fickus for his insightful comments on chirp frames, and Samuel Feng and Michael A.~Schwemmer for their help with using the computer clusters in Princeton's mathematics department.
% This work was supported by the Office of Naval Research under grant N00014-08-1-1110,
% by the Air Force Office of Scientific Research under grants FA9550-09-1-0551 and
% FA 9550-09-1-0643, and by NSF under grant DMS-0914892.
% Mixon was supported by the A.B. Krongard Fellowship.
% The views expressed in this article are those of the authors and do not reflect the official policy or position of the United States Air Force, Department of Defense, or the U.S. Government.
% 
% \section*{References}
% 
% \begin{thebibliography}{00}
% 
% 
% \end{thebibliography}
% 
% \end{document}

\end{document}